\documentclass[11pt]{amsbook}
\usepackage{amsfonts, amsthm, amssymb, amsmath, stmaryrd, color, enumerate, relsize}
\usepackage[pdftex]{graphicx}
\usepackage{mathrsfs,array}
\usepackage{xy}
\usepackage{hyperref}
\usepackage{tikz}
\input xy
\xyoption{all}

\DeclareMathOperator{\Spa}{Spa}

\DeclareMathOperator{\Spv}{Spv}

\DeclareMathOperator{\GL}{GL}

\DeclareMathOperator{\Hom}{Hom}
\DeclareMathOperator{\Ext}{Ext}
\newcommand{\intHom}{\underline{\Hom}}

\DeclareMathOperator{\Spec}{Spec}

\DeclareMathOperator{\Sets}{Sets}

\DeclareMathOperator{\Cond}{Cond}
\DeclareMathOperator{\Ab}{Ab}
\DeclareMathOperator{\Ch}{Ch}

\DeclareMathOperator{\Latt}{Latt}
\DeclareMathOperator{\Fun}{Fun}
\DeclareMathOperator{\LCA}{LCA}

\newcommand{\proet}{{\mathrm{pro\acute{e}t}}}

\newcommand{\solid}{{\mathsmaller{\square}}}
\DeclareMathOperator{\Solid}{Solid}

\newcommand{\Mod}{{\text-\mathrm{Mod}}}

\renewcommand*{\tilde}{\widetilde}
\renewcommand{\det}{\mathrm{det}}

\numberwithin{equation}{section}
\newtheorem{theorem}{Theorem}
\numberwithin{theorem}{section}
\newtheorem{lemma}[theorem]{Lemma}
\newtheorem{corollary}[theorem]{Corollary}
\newtheorem{proposition}[theorem]{Proposition}

\theoremstyle{definition}
\newtheorem{remark}[theorem]{Remark}

\newtheorem{warning}[theorem]{Warning}
\newtheorem{definition}[theorem]{Definition}
\newtheorem{question}[theorem]{Question}
\newtheorem{example}[theorem]{Example}
\newtheorem{examples}[theorem]{Examples}
\newtheorem{observation}[theorem]{Observation}

\date{\today}

\title{Lectures on Condensed Mathematics}

\author{Peter Scholze}

\begin{document}

\maketitle

\tableofcontents

\chapter*{Lectures on Condensed Mathematics}

\section*{Preface}

This is an updated version of the lectures notes for a course on condensed mathematics taught in the summer term 2019 at the University of Bonn. The material presented is joint work with Dustin Clausen.\\

This is intended as a stable citable version of the original lectures, with mostly cosmetic changes to the original document, together with some small corrections. The content of these lectures has been expanded on at various places, including \cite{Asgeirsson}, \cite{Brink}, \cite{Deglise}, \cite{leStum}, \cite{Mair}, \cite{MannThesis}, and the reader is strongly encouraged to consider these sources as well.\\

\hfill{May 2026, Peter Scholze}

\newpage

\section{Lecture I: Condensed Sets}

The basic question to be addressed in this course is the following.

\begin{question} How to do algebra when rings/modules/groups carry a topology?
\end{question}

This question comes up rather universally, and many solutions have been found so that in any particular situation it is usually clear what to do. Here are some sample situations of interest:

\begin{enumerate}
\item[{\rm (i)}] Representations of topological groups like $\GL_n(\mathbb R)$, $\GL_n(\mathbb Q_p)$, possibly on Banach spaces or other types of topological vector spaces, or more generally topological abelian groups.
\item[{\rm (ii)}] Continuous group cohomology: If $G$ is a topological group acting on a topological abelian group $M$, there are continuous group cohomology groups $H^i_{\mathrm{cont}}(G,M)$ defined via an explicit cochain complex of continuous cochains.
\item[{\rm (iii)}] Algebraic (or rather analytic) geometry over topological fields $K$ such as $\mathbb R$ or $\mathbb Q_p$.
\end{enumerate}

\noindent However, various foundational problems remain, for example:\\

{\bf Problems.}\begin{enumerate}
\item[{\rm(i)}] Topological abelian groups do not form an abelian category. For example, consider the map
\[
(\mathbb R,\mathrm{discrete\ topology})\to (\mathbb R,\mathrm{natural\ topology})\ .
\]
In an abelian category, the failure of this map to be an isomorphism has to be explained by a nontrivial kernel or cokernel. Note that this example is not pathological: We certainly want to include all abstract discrete abelian groups, as well as $\mathbb R$, and we certainly want this map to be a map of topological abelian groups. So if we want to change the setup to obtain an abelian category, we need to include also some kernel or cokernel to the displayed map.
\item[{\rm (ii)}] For a topological group $G$, a short exact sequence of continuous $G$-modules does not in general give long exact sequences of continuous group cohomology groups. More abstractly, the theory of derived categories does not mix well with topological structures.
\item[{\rm (iii)}] In the context of analytic geometry over some topological field $K$, a theory of quasicoherent sheaves presents problems: For example, if $A\to B$ is a map of topological $K$-algebras and $M$ is a topological $A$-module, one would like to form a base change $M\widehat{\otimes}_A B$, but it is usually not clear how to complete.
\end{enumerate}

The goal of this course is to give a unified approach to the problem of doing algebra when rings/modules/groups/... carry a topology, and resolve those and other foundational problems. The ideas started in this course have been further developed in our joint work with Clausen, see \cite{Analytic}, \cite{Complex} and \cite{AnalyticStacks}.

The solution comes from a somewhat unexpected (to the author) direction. In \cite{BhattScholzeProEtale}, the pro-\'etale site $X_\proet$ of any scheme $X$ was defined, whose objects are roughly projective limits of \'etale maps to $X$. For this course, only the simplest case where $X$ is a (geometric) point is relevant. Indeed, unlike in classical topologies such as the Zariski or \'etale topology, sheaves on the pro-\'etale site of a point are not merely sets. What seemed like a bug of the pro-\'etale site is actually a feature for us here, as Clausen explained to the author.

\begin{definition}\label{def:condensed} The pro-\'etale site $\ast_\proet$ of a point is the category of profinite sets $S$, with finite families of jointly surjective maps as covers. A \emph{condensed set} is a sheaf of sets on $\ast_\proet$. Similarly, a \emph{condensed ring/group/...} is a sheaf of rings/groups/... on $\ast_\proet$.

Generally, if $\mathcal C$ is any category, the category
\[
\Cond(\mathcal C)
\]
of condensed objects of $\mathcal C$ is the category of $\mathcal C$-valued sheaves on $\ast_\proet$.
\end{definition}

In other words, a condensed set/ring/group/... is a functor
\[\begin{aligned}
T: \{\mathrm{profinite\ sets}\}^{\mathrm{op}}&\to \{\mathrm{sets/rings/groups/...}\}\\
S&\mapsto T(S)
\end{aligned}\]
satisfying $T(\emptyset) = \ast$ and the following two conditions (which are equivalent to the sheaf condition):
\begin{enumerate}
\item[{\rm (i)}] For any profinite sets $S_1,S_2$, the natural map
\[
T(S_1\sqcup S_2)\to T(S_1)\times T(S_2)
\]
is a bijection.
\item[{\rm (ii)}] For any surjection $S^\prime\to S$ of profinite sets with the fibre product $S^\prime\times_S S^\prime$ and its two projections $p_1,p_2$ to $S^\prime$, the map
\[
T(S)\to \{x\in T(S^\prime)\mid p_1^\ast(x) = p_2^\ast(x)\in T(S^\prime\times_S S^\prime)\}
\]
is a bijection.
\end{enumerate}

Given a condensed set $T$, we sometimes refer to $T(\ast)$ as its underlying set. The following example presents the key translation from topological structures to condensed structures:

\begin{example} Let $T$ be any topological space. There is an associated condensed set $\underline{T}$, defined via sending any profinite set $S$ to the set of continuous maps from $S$ to $T$. This clearly satisfies (i) above, while (ii) follows from the fact that any surjection $S^\prime\to S$ of compact Hausdorff spaces is a quotient map, so any map $S\to T$ for which the composite $S^\prime\to S\to T$ is continuous is itself continuous.

If $T$ is a topological ring/group/..., then $\underline{T}$ is a condensed ring/group/... .
\end{example}

\begin{remark} Definition~\ref{def:condensed} presents set-theoretic problems: The category of profinite sets is large, so it is not a good idea to consider functors defined on all of it. We will resolve this issue over this and the next lecture, and in doing so we will actually pass to a minor modification of the category above, cf.~also Warning~\ref{war:septopcardinal} in relation to the previous example.

In a first step, we will choose a cardinal bound on the profinite sets, i.e.~fix a suitable uncountable cardinal $\kappa$ and use only profinite sets $S$ of cardinality less than $\kappa$ to define a site $\ast_{\kappa\text-\proet}$. In these lectures, we will always assume that $\kappa$ is an uncountable strong limit cardinal, i.e.~$\kappa$ is uncountable and for all $\lambda<\kappa$, also $2^\lambda<\kappa$.\footnote{However, the case of regular $\kappa$ is also interesting and useful; in particular, the choice $\kappa=\aleph_1$ leads to the notion of ``light condensed sets'' which are used in the theory of analytic stacks \cite{AnalyticStacks}.} It is easy to construct such $\kappa$: For example, define $\beth_\alpha$ inductively for all ordinals $\alpha$ via $\beth_0=\aleph_0$, $\beth_{\alpha^+} = 2^{\beth_\alpha}$ for a successor ordinal and for a limit ordinal as the union of all smaller $\beth_{\alpha}$'s. Then for any limit ordinal $\alpha$, the cardinal $\kappa=\beth_\alpha$ is an uncountable strong limit cardinal.

For any uncountable strong limit cardinal $\kappa$, the category of $\kappa$-condensed sets is the category of sheaves on $\ast_{\kappa\text-\proet}$.
\end{remark}

\begin{remark} If $\kappa^\prime>\kappa$ are uncountable strong limit cardinals, there is an obvious forgetful functor from $\kappa^\prime$-condensed sets to $\kappa$-condensed sets. This functor admits a natural left adjoint, which is a functor from $\kappa$-condensed sets to $\kappa^\prime$-condensed sets. This functor is fully faithful and commutes with all colimits and many limits (in particular, all finite limits), as we will prove in Proposition~\ref{prop:changekappa} below. We will then define the category of condensed sets as the (large) colimit of the category of $\kappa$-condensed sets along the filtered poset of all uncountable strong limit cardinals $\kappa$.\footnote{In \cite{BarwickHaine}, Barwick--Haine set up closely related foundations, but using different set-theoretic conventions. In particular, they assume the existence of universes, fixing in particular a ``tiny" and a ``small" universe, and look at sheaves on tiny profinite sets with values in small sets; they term these \emph{pyknotic} sets. In our language, placing ourselves in the small universe, this would be $\kappa$-condensed sets for the first strongly inaccessible cardinal $\kappa$ they consider (the one giving rise to the tiny universe).}
\end{remark}

Recall that a topological space $X$ is \emph{compactly generated} if a map $f: X\to Y$ to another topological space $Y$ is continuous as soon as the composite $S\to X\to Y$ is continuous for all compact Hausdorff spaces $S$ mapping to $X$. The inclusion of compactly generated spaces into all topological spaces admits a right adjoint $X\mapsto X^{\mathrm{cg}}$ that sends a topological space $X$ to its underlying set equipped with the quotient topology for the map $\sqcup_{S\to X} S\to X$, where the disjoint union runs over all compact Hausdorff spaces $S$ mapping to $X$.

Note that any compact Hausdorff space $S$ admits a surjection from a profinite set (which is then automatically a quotient map), for example by taking the Stone--\v{C}ech compactification of $S$ considered as a discrete set. This means that in the preceding paragraph, one can replace compact Hausdorff spaces by profinite sets without altering the definition of compactly generated spaces and the functor $X\mapsto X^{\mathrm{cg}}$.

We will need the following set-theoretic variant:\footnote{Removing the cardinal bounds from Proposition~\ref{prop:condtopkappa} below results in some surprising twist, cf.~Warning~\ref{war:septopcardinal}, so we are careful with the cardinal bounds here.} choose an uncountable strong limit cardinal $\kappa$ as above, and say that a topological space $X$ is $\kappa$-compactly generated if it is equipped with the quotient topology from $\sqcup_{S\to X} S\to X$ where now $S$ runs only over compact Hausdorff spaces with $|S|<\kappa$. Any compact Hausdorff space $S$ with $|S|<\kappa$ admits a surjection from the profinite set $S^\prime$ which is the Stone--\v{C}ech compactification of $S$ considered as a discrete set, and then $|S^\prime|\leq 2^{2^{|S|}}<\kappa$ as $\kappa$ is a strong limit cardinal and $S^\prime$ is the set of ultrafilters on the set $S$, i.e.~a subset of the powerset of the powerset of $S$. We write $X\mapsto X^{\kappa\text-\mathrm{cg}}$ for the right adjoint of the inclusion of $\kappa$-compactly generated topological spaces into all topological spaces; here $X^{\kappa\text-\mathrm{cg}}$ has the underlying set of $X$ equipped with the quotient topology of $\sqcup_{S\to X} S\to X$ where $S$ runs over $\kappa$-small compact Hausdorff $S$.

\begin{remark}\label{rem:firstcountable} Any first-countable topological space $X$ (in particular, any metrizable topological space) is compactly generated; in fact $\kappa$-compactly generated for any uncountable $\kappa$. Indeed, assume that $V\subset X$ is a subset such that for all $\kappa$-small compact Hausdorff spaces $S$ mapping to $X$, the preimage of $V$ in $S$ is closed. We need to see that $V$ is closed. Take any point $x\in X$ in the closure of $V$, and choose a sequence $U_1\supset U_2\supset \ldots $ of cofinal open neighborhoods of $x$ (using that $X$ is first-countable). Then $V\cap U_n$ is nonempty for all $n$, so choosing a point $x_n\in V\cap U_n$ we get a map $\mathbb N\cup \{\infty\}\to X$ sending $n$ to $x_n$ and $\infty$ to $x$; this map is continuous for the usual profinite topology on $\mathbb N\cup \{\infty\}$. But then the preimage of $V$ in $\mathbb N\cup \{\infty\}$ is closed and contains $\mathbb N$, and thus also $\{\infty\}$, so that indeed $x\in V$, i.e.~$V$ is closed.
\end{remark}

\begin{proposition}\label{prop:condtopkappa} The functor $X\mapsto \underline{X}$ from the category of topological spaces/rings/groups/... to the category of $\kappa$-condensed sets/rings/groups/... is faithful, and fully faithful when restricted to the full subcategory of all $T$ that are $\kappa$-compactly generated as a topological space.

The functor $X\mapsto \underline{X}$ from topological spaces to $\kappa$-condensed sets admits a left adjoint $T\mapsto T(\ast)_{\mathrm{top}}$ sending any condensed set $T$ to its underlying set $T(\ast)$ equipped with the quotient topology of $\sqcup_{S\to T} S\to T(\ast)$ where the disjoint union runs over all $\kappa$-small profinite sets $S$ with a map to $T$, i.e.~an element of $T(S)$. The counit $\underline{X}(\ast)_{\mathrm{top}}\to X$ of the adjunction agrees with the counit $X^{\kappa\text-\mathrm{cg}}\to X$ of the adjunction between $\kappa$-compactly generated spaces and all topological spaces; in particular, $\underline{X}(\ast)_{\mathrm{top}}\cong X^{\kappa\text-\mathrm{cg}}$.
\end{proposition}

\begin{remark} If $X$ is a condensed group (or ring/...), it is not necessarily the case that $\underline{X}(\ast)_{\mathrm{top}}$ is a topological group (or ring/...). The problem is that the functor $X\mapsto X(\ast)_{\mathrm{top}}$ does not necessarily commute with products.
\end{remark}

\begin{proof} Faithfulness is clear as a map of topological spaces is determined by the underlying map of sets. It suffices to check fully faithfulness for the functor from $\kappa$-compactly generated topological spaces to $\kappa$-condensed sets, as the other cases (of rings and groups) are just certain diagram categories. Here, the claim reduces formally to the claim about the adjoint functor and its counit.

Thus, let $X$ be a topological space and $T$ a $\kappa$-condensed set. We claim that there is a functorial bijection
\[
\Hom(T,\underline{X}) = \Hom(T(\ast)_{\mathrm{top}},X)\ .
\]
Note that as for all profinite $S$, the map $\underline{X}(S)\to \prod_{s\in S} \underline{X}(\{s\}) = \prod_{s\in S} X$ is injective (it identifies with the inclusion from the set of continuous maps from $S$ to $X$ into all maps from $S$ to $X$), giving a map $T\to \underline{X}$ is equivalent to giving a map of sets $T(\ast)\to X$ such that for all $\kappa$-small profinite $S$ with a map $S\to T$, the induced map $S\to T(\ast)\to X$ is continuous. But these are precisely the continuous maps $T(\ast)_{\mathrm{top}}\to X$ by definition of the topology on $T(\ast)_{\mathrm{top}}$.

It is clear from the definitions that for a topological space $X$, there is natural identification $\underline{X}(\ast)_{\mathrm{top}}\cong X^{\kappa\text-\mathrm{cg}}$.
\end{proof}

\begin{example} Let $\mathbb R_{\mathrm{disc}}$ denote the real numbers with the discrete topology, and $\mathbb R$ the real numbers with their natural topology. Then the map
\[
\underline{\mathbb R_{\mathrm{disc}}}\to \underline{\mathbb R}
\]
is injective, and its cokernel $Q$ is a condensed abelian group with ``underlying'' abelian group $Q(\ast)=0$. However, for a general profinite set $S$, one has
\[
Q(S) = \{\mathrm{continuous\ maps}\ S\to \mathbb R\}/\{\mathrm{locally\ constant\ maps}\ S\to \mathbb R\}\neq 0\ .
\]
We note that a priori $Q$ is the sheafification of this functor, but it turns out that sheafification is unnecessary, by the vanishing $H^1_{\mathrm{cond}}(S,\mathbb R_{\mathrm{disc}})=0$ proved in Theorem~\ref{thm:cohomcorrect}.
\end{example}

In fact, condensed abelian groups form an abelian category of the nicest possible sort. In particular, out of Grothendieck's axioms, it satisfies the same as the category of abelian groups.

\begin{theorem} The category of condensed abelian groups is an abelian category which satisfies Grothendieck's axioms (AB3), (AB4), (AB5), (AB6), (AB3*) and (AB4*), to wit: all limits (AB3*) and colimits (AB3) exist, arbitrary products (AB4*), arbitrary direct sums (AB4) and filtered colimits (AB5) are exact, and (AB6) for any index set $J$ and filtered categories $I_j$, $j\in J$, with functors $i\mapsto M_i$ from $I_j$ to condensed abelian groups, the natural map
\[
\varinjlim_{(i_j\in I_j)_j} \prod_{j\in J} M_{i_j}\to \prod_{j\in J} \varinjlim_{i_j\in I_j} M_{i_j}
\]
is an isomorphism.

Moreover, the category of condensed abelian groups is generated by compact projective objects.
\end{theorem}

Recall that an object $M$ of an abelian category $C$ is compact if $\Hom(M,-)$ commutes with filtered colimits.

\begin{remark} Sheaves of abelian groups on any site form an abelian category satisfying (AB3), (AB4), (AB5) and (AB3*). However, in almost all cases, they fail (AB4*) and (AB6). These axioms are however a simple consequence of the existence of compact projective generators. In our case, the compact projective generators will come from extremally disconnected sets, as explained in the next lecture.
\end{remark}
\newpage

\section{Lecture II: Condensed Abelian Groups}

For a choice of uncountable strong limit cardinal $\kappa$, we made the following definition.

\begin{definition} The site $\ast_{\kappa\text-\proet}$ is the site of $\kappa$-small profinite sets $S$ with covers given by finite families of jointly surjective maps. A $\kappa$-condensed set/group/ring/... is a sheaf of sets/groups/rings/... on $\ast_{\kappa\text-\proet}$.
\end{definition}

Today, we prove the following proposition.

\begin{theorem}\label{thm:niceabcat} The category of $\kappa$-condensed abelian groups is an abelian category which satisfies Grothendieck's axioms (AB3), (AB4), (AB5), (AB6), (AB3*) and (AB4*), to wit: all limits (AB3*) and colimits (AB3) exist, arbitrary products (AB4*), arbitrary direct sums (AB4) and filtered colimits (AB5) are exact, and (AB6) for any index set $J$ and filtered categories $I_j$, $j\in J$, with functors $i\mapsto M_i$ from $I_j$ to $\kappa$-condensed abelian groups, the natural map
\[
\varinjlim_{(i_j\in I_j)_j} \prod_{j\in J} M_{i_j}\to \prod_{j\in J} \varinjlim_{i_j\in I_j} M_{i_j}
\]
is an isomorphism.

Moreover, the category of $\kappa$-condensed abelian groups is generated by compact projective objects.
\end{theorem}

What makes condensed sets a priori hard to control is the sheaf condition: Forming a colimit requires sheafification, and it is not clear that this can be controlled. In fact, a priori the sheafification involved might range over larger and larger diagrams depending on the choice of $\kappa$. Fortunately, this turns out not to be the case. To prove this, we study two different sites whose categories of sheaves are equivalent to condensed sets.

\begin{proposition}\label{prop:chausvsprofin} Consider the site of all $\kappa$-small compact Hausdorff spaces, with covers given by finite families of jointly surjective maps. Its category of sheaves is equivalent to $\kappa$-condensed sets via restriction to profinite sets.
\end{proposition}

\begin{proof} For any compact Hausdorff space $S$, one can find a profinite $S^\prime$ with a surjection $S^\prime\to S$. For example, one can take for $S^\prime$ the Stone--\v{C}ech compactification of $S$ considered as a discrete set. For any sheaf $T$, this determines the value $T(S)$ in terms of the values $T(S^\prime)$, $T(S^\prime\times_S S^\prime)$ on the profinite sets $S^\prime$, $S^\prime\times_S S^\prime\subset S^\prime\times S^\prime$. The desired equivalence follows easily (noting that $|S^\prime|\leq 2^{2^{|S|}}<\kappa$ to verify that the cardinal bounds check out).
\end{proof}

The Stone--\v{C}ech compactifications of discrete sets appearing in the previous proof are actually a very special type of profinite set:

\begin{definition}[Gleason, \cite{Gleason}] A compact Hausdorff space $S$ is \emph{extremally disconnected} if any surjection $S^\prime\to S$ from a compact Hausdorff space splits.
\end{definition}

\begin{example} Assume that $S=\beta S_0$ is the Stone--\v{C}ech compactification of a discrete set $S_0$. Then $S$ is extremally disconnected. Indeed, for any surjection $S^\prime\to S$, we may first lift the discrete set $S_0\subset S$ arbitrarily to $S^\prime$; this extends to a unique section $S=\beta S_0\to S^\prime$ by the universal property of the Stone--\v{C}ech compactification.

In particular, any compact Hausdorff space $S$ admits a surjection from an extremally disconnected space $S^\prime = \beta S$, with $|S^\prime|\leq 2^{2^{|S|}}$.
\end{example}

\begin{warning} Being extremally disconnected is a much stronger condition than being totally disconnected, and it is in fact hard to come up with examples that are not Stone--\v{C}ech compactifications or very related. (Indeed, any extremally disconnected set is a retract of a Stone--\v{C}ech compactification.) In particular, it is known that for two infinite extremally disconnected sets $S$, $S^\prime$, their product $S\times S^\prime$ is never extremally disconnected. Also, if $S$ is extremally disconnected and $x_0,x_1,\ldots$ is a sequence of points in $S$ that converges, then it is eventually constant, cf.~\cite{Gleason}.
\end{warning}

\begin{remark} Extremally disconnected compact Hausdorff spaces play an important role in set theory, or more particularly the theory of forcing, to find interesting models of ZFC. Namely, any extremally disconnected compact Hausdorff space can be used to construct a forcing extension of the ground model, and conversely any forcing extension arises this way. In the language ETCSR of the elementary theory of the category of sets with replacement --- which is essentially an equivalent presentation of ZFC, characterizing it in terms of its category of sets, cf.~e.g.~\cite{ShulmanSetTheory} --- the category of sets of the forcing extension is given as follows, cf.~\cite{SheavesGeometryLogic}. Let $S$ be an extremally disconnected compact Hausdorff space, and pick a point $s\in S$. Consider the site of open and closed subsets of $S$, where a family $\{U_i\subset U\}_i$ is a cover if $\bigcup_i U_i\subset U$ is dense, and let $\mathrm{Shv}^\wedge(S)$ be the resulting category of sheaves. The same construction can be applied to any open and closed subset $U\subset S$. One can form the category
\[
\mathrm{colim}_{U\ni s} \mathrm{Shv}^\wedge(U),
\]
where the filtered colimit along the pullback functors is taken in the $2$-category of categories. This category turns out to be a model of ETCSR (which can then be reinterpreted as a model of ZFC). If $S$ is a Stone--\v{C}ech compactification, this is simply an ultraproduct of the category $\mathrm{Set}$, and hence passing from $\mathrm{Set}$ to this extension does not change the truth of any ZFC-formula. However, for general extremally disconnected profinite sets, this extension is very different.

This presentation shows that there is a tight relation between some of the structures of condensed mathematics, and set theory/forcing --- both can be seen to be concerned with (different kinds of) sheaves on extremally disconnected spaces. See \cite{BergfalkLambieHanson}, \cite{BannisterBasak} for very interesting work relating set theory and condensed mathematics.
\end{remark}

\begin{proposition} Consider the site of $\kappa$-small extremally disconnected sets, with covers given by finite families of jointly surjective maps. Its category of sheaves is equivalent to $\kappa$-condensed sets via restriction from profinite sets.
\end{proposition}

\begin{proof} For any profinite set $S$, we can find an extremally disconnected set $\tilde{S}$ surjecting onto $S$ and another extremally disconnected set $\tilde{\tilde{S}}$ surjecting onto $\tilde{S}\times_S \tilde{S}$. Then the value $T(S)$ is determined as the equalizer of the two maps $T(\tilde{S})\to T(\tilde{\tilde{S}})$. This easily implies the proposition.
\end{proof}

Using the proposition, we see that the category of $\kappa$-condensed sets/rings/groups/... is equivalent to the category of functors
\[
T: \{\kappa\text-\mathrm{small\ extremally\ disconnected\ sets}\}^{\mathrm{op}}\to \{\mathrm{sets/rings/groups/...}\}
\]
such that $T(\emptyset) = \ast$ and for all $\kappa$-small extremally disconnected sets $S_1$, $S_2$, the natural map
\[
T(S_1\sqcup S_2)\to T(S_1)\times T(S_2)
\]
is a bijection. Indeed, as any cover of extremally disconnected sets splits, the analogue of condition (ii) above is automatic.

\begin{proof}[Proof of Theorem~\ref{thm:niceabcat}] We identify the category of $\kappa$-condensed abelian groups with the category of functors
\[
M: \{\kappa\text-\mathrm{small\ extremally\ disconnected\ sets}\}^{\mathrm{op}}\to \{\mathrm{abelian\ groups}\}
\]
sending finite disjoint unions to finite products. As the formation of all limits and colimits in the category of abelian groups commutes with finite products (=finite coproducts), one sees that this category is stable under formation of pointwise limits and colimits. In other words, for any category $I$ and any functor $i\mapsto M_i$ to $\kappa$-condensed abelian groups, the limit (resp.~colimit) of $i\mapsto M_i$ is given by sending any $\kappa$-small extremally disconnected set $S$ to the limit (resp.~colimit) of $i\mapsto M_i(S)$. With this information, all but the last sentence of the theorem follow at once from the same assertions in the category of abelian groups.

To prove the existence of compact projective generators, we use that by the adjoint functor theorem, the forgetful functor from $\kappa$-condensed abelian groups to $\kappa$-condensed sets has a left adjoint $T\mapsto \mathbb Z[T]$.\footnote{Concretely, $\mathbb Z[T]$ is the sheafification of the functor that sends an extremally disconnected set $S$ to the free abelian group $\mathbb Z[T(S)]$ on the set $T(S)$.} In particular, for any $\kappa$-small extremally disconnected set $S$, there is a $\kappa$-condensed abelian group $\mathbb Z[S]$ such that for any $\kappa$-condensed abelian group $M$, we have $\Hom(\mathbb Z[S],M) = M(S)$. As $M\mapsto M(S)$ commutes with all limits and colimits, it follows that $\mathbb Z[S]$ is compact and projective; and they generate, as $M=0$ if and only if $M(S)=0$ for all $\kappa$-small extremally disconnected sets $S$, which implies that any $M$ admits a surjection from a direct sum of $\mathbb Z[S]$'s, via
\[
\bigoplus_{S,\alpha\in M(S)} \mathbb Z[S]\to M.
\]
\end{proof}

In the appendix to this lecture, we study how the previous notions change when one changes $\kappa$. In particular, a notion of condensed abelian groups is defined that is both independent of $\kappa$ and which poses no set-theoretic problems. Theorem~\ref{thm:niceabcat} continues to hold. We will use this formalism in the following.

Let us mention some further properties of the category $\Cond(\Ab)$ of condensed abelian groups.

\begin{enumerate}
\item[{\rm (i)}] It has a symmetric monoidal tensor product $-\otimes -$, where $M\otimes N$ is the sheafification of $S\mapsto M(S)\otimes N(S)$. Thus, $M\otimes N$ represents bilinear maps $M\times N\to -$, and one sees that the functor $T\mapsto \mathbb Z[T]$ from condensed sets to condensed abelian groups is symmetric monoidal with respect to the product and the tensor product, i.e. $\mathbb Z[T_1]\otimes \mathbb Z[T_2] = \mathbb Z[T_1\times T_2]$.

For any condensed set $T$, the condensed abelian group $\mathbb Z[T]$ is flat. To check this, it suffices to check that tensoring with the presheaf $S\mapsto \mathbb Z[T(S)]$ (of which $\mathbb Z[T]$ is the sheafification) on the presheaf level is exact, which follows from all $\mathbb Z[T(S)]$ being flat (even free) abelian groups.
\item[{\rm (ii)}] For any condensed abelian groups $M$, $N$, the group of homomorphisms $\Hom(M,N)$ has a natural enrichment to a condensed abelian group, defining an internal $\Hom$-functor object $\intHom(M,N)$. Abstractly, this can be defined via the adjunction
\[
\Hom(P,\intHom(M,N))\cong \Hom(P\otimes M,N)\ ;
\]
concretely, applying this to the free objects $P=\mathbb Z[S]$, one has
\[
\intHom(M,N)(S) = \Hom(\mathbb Z[S]\otimes M,N)
\]
for any extremally disconnected set $S$.
\item[{\rm (iii)}] As $\Cond(\Ab)$ has enough projectives, one can form the derived category $D(\Cond(\Ab))$. If $P\in \Cond(\Ab)$ is compact and projective, then $P[0]\in D(\Cond(\Ab))$ is a compact object of the derived category, i.e.~$\Hom(P,-)$ commutes with arbitrary direct sums. In particular, $D(\Cond(\Ab))$ is compactly generated. Using projective resolutions, one also defines a symmetric monoidal tensor product $-\otimes^L-$ on $D(\Cond(\Ab))$ and a derived internal Hom $R\intHom(-,-)$ on $D(\Cond(\Ab))$, still satisfying the adjunction
\[
\Hom(P,R\intHom(M,N))\cong \Hom(P\otimes^L M,N)\ .
\]
\end{enumerate}

\newpage

\section*{Appendix to Lecture II}

We make some somments on how to build a theory of condensed sets that does not suffer from set-theoretic issues while also not depending on the choice of an auxiliary cardinal $\kappa$. For this, we first prove the following result on (in)dependence of $\kappa$.

\begin{proposition}\label{prop:changekappa} Let $\kappa^\prime>\kappa$ be uncountable strong limit cardinals. There is a natural functor from $\kappa$-condensed sets to $\kappa^\prime$-condensed sets, given by sending a $\kappa$-condensed set $T$ to the $\kappa^\prime$-condensed set $T_{\kappa^\prime}$ given by the sheafification of
\[
\tilde{S}\mapsto \varinjlim_{\tilde{S}\to S} T(S)
\]
where the filtered colimit is taken over all $\kappa$-small profinite sets $S$ with a map $\tilde{S}\to S$.\footnote{In other words, $T_{\kappa^\prime}$ is the pullback of $T$ along the map of sites $\ast_{\kappa^\prime\text-\proet}\to \ast_{\kappa\text-\proet}$.}

The functor $T\mapsto T_{\kappa^\prime}$ is fully faithful and commutes with all colimits and all $\lambda$-small limits where $\lambda$ is the cofinality of $\kappa$.
\end{proposition}

\begin{remark} Except for the commutation with $\lambda$-small limits, this proposition applies verbatim to condensed objects in any category that admits all filtered colimits. For commutation with $\lambda$-small limits, one has to be careful in general; it is however certainly true if there is a conservative forgetful functor to sets that commutes with limits and filtered colimits (such as for rings, groups, modules, etc.).
\end{remark}

\begin{proof} We identify $\kappa$-condensed sets with functors
\[
\{\kappa\text-\mathrm{small\ extremally\ disconnected\ sets}\}^{\mathrm{op}}\to \{\mathrm{sets}\}
\]
sending finite disjoint unions to finite products, and similarly for $\kappa^\prime$. Under this equivalence, the functor $T\mapsto T_{\kappa^\prime}$ corresponds to left Kan extension along the full inclusion of $\kappa$-small extremally disconnected sets into $\kappa^\prime$-small extremally disconnected sets. In particular, the functor is left adjoint to the forgetful functor from $\kappa^\prime$-condensed sets to $\kappa$-condensed sets and the unit of the adjunction is an isomorphism. This implies that the functor $T\mapsto T_{\kappa^\prime}$ is fully faithful and commutes with all colimits.

To see that it commutes with $\lambda$-small limits, where $\lambda$ is the cofinality of $\kappa$, it is enough to see that for all $\kappa^\prime$-small extremally disconnected sets $\tilde{S}$, the category of all $\kappa$-small extremally disconnected sets $S$ with a map $\tilde{S}\to S$ is $\lambda$-filtered. (Here, we use the general fact that $\lambda$-filtered colimits commute with $\lambda$-small limits in sets; this is true for any regular cardinal $\lambda$, but cofinalities are always regular.) This comes down to showing that if $S_i$, $i\in I$ for some $\lambda$-small index category $I$, is a diagram of $\kappa$-small extremally disconnected sets with compatible maps $\tilde{S}\to S_i$, then there is a $\kappa$-small extremally disconnected set $S$ with a map $\tilde{S}\to S$ over which all $\tilde{S}\to S_i$ factor compatibly. Note that the limit $\varprojlim_i S_i$ is $\kappa$-small: It suffices to prove this for $\prod_i S_i$, and then the supremum $\mu$ of all $|S_i|$ is less than $\kappa$ as the cofinality of $\kappa$ is $\lambda$ and $|I|<\lambda$, so $|\prod_{i\in I} S_i|\leq \mu^\lambda\leq 2^{\mu\times \lambda}<\kappa$ as $\mu\times \lambda<\kappa$ and $\kappa$ is a strong limit cardinal. We can now take for $S$ any $\kappa$-small extremally disconnected set with a surjection $S\to \varprojlim_i S_i$; as $\tilde{S}$ is extremally disconnected, the map $\tilde{S}\to \varprojlim_i S_i$ can be lifted to a map $\tilde{S}\to S$.
\end{proof}

We are now justified in making the following definition, which is the official definition of condensed sets.

\begin{definition} The category of condensed sets is given by the filtered colimit of the category of $\kappa$-condensed sets along the filtered poset of all uncountable strong limit cardinals $\kappa$.
\end{definition}

Equivalently, this is the category of functors
\[\begin{aligned}
T: \{\mathrm{extremally\ disconnected\ sets}\}^{\mathrm{op}}&\to \{\mathrm{sets}\}\\
S&\mapsto T(S)
\end{aligned}\]
sending finite disjoint unions to finite products, and such that for some uncountable strong limit cardinal $\kappa$, it is the left Kan extension of its restriction to $\kappa$-small extremally disconnected sets.\footnote{Yet another equivalent characterization is that a sheaf $T$ defined on all of $\ast_\proet$ is a condensed set when the corresponding functor from the opposite category of profinite sets to sets, i.e.~from the ind-category of (finite sets)${}^{\mathrm{op}}$ to sets, is accessible (noting that both sets and ind-(finite sets)${}^{\mathrm{op}}$ are presentable categories).}

\begin{remark} The category of condensed sets is a (large) category like the category of all sets; the essential difference is that it does not admit a set of generators (but merely a class of generators). In particular, the resulting category is not the category of sheaves on any site, although it shares most of the features. This has been studied in detail by Brink, \cite{Brink}, under the term big topos (and big presentable categories).
\end{remark}

\begin{remark} In forming the colimit over all $\kappa$, one can restrict attention to $\kappa$ of cofinality at least $\lambda$ for any fixed $\lambda$ (as these are cofinal), and thus ensure that formation of $\lambda$-small limits is universal for any given $\lambda$. Slightly better, one can by transfinite induction define a sequence $\kappa_\alpha$ of uncountable strong limit cardinals, enumerated by ordinals $\alpha$, such that the cofinality of $\kappa_\alpha$ is at least the supremum of $\kappa_{\alpha^\prime}$ over all $\alpha^\prime<\alpha$.
\end{remark}

With this definition, the category of condensed sets has all small colimits and all small limits. Here limits can be calculated pointwise; to see that the limit is still the left Kan extension from its restriction to $\kappa$-small objects for some $\kappa$, choose $\kappa$ so that the cofinality of $\kappa$ is at least the cardinality of the index category, and all terms in the limit are left Kan extended from $\kappa$-small objects. Colimits commute with the left Kan extensions, so it is clear that they preserve left Kan extensions. Filtered colimits can be computed pointwise.

The same discussion applies to condensed objects in any category $C$ that admits all filtered colimits. In other words,
\[
\Cond(C) = \varinjlim_\kappa \Cond_\kappa(C)
\]
with (hopefully) obvious notation; this also agrees with the category of functors
\[
\{\mathrm{extremally\ disconnected\ sets}\}^{\mathrm{op}}\to C
\]
sending finite disjoint unions to finite products, and which are their left Kan extensions from the restriction to $\kappa$-small objects for some uncountable strong limit cardinal $\kappa$.

This category is large (even if $C$ is small), but is locally small, i.e.~$\Hom$ sets are actually sets. Indeed, as all transition maps are fully faithful, any $\Hom$ can be calculated in some $\Cond_\kappa(C)$.

\begin{warning}\label{war:septopcardinal} The natural functor from topological spaces to sheaves on the full pro-\'etale site $\ast_\proet$ of a point does not land in condensed sets. The problem are nonseparated spaces like the Sierpinski space $X=\{s,\eta\}$ where $\{\eta\}$ is open but $\{s\}$ is not. Indeed, then $\underline{X}$ sends any extremally disconnected set $S$ to the set of all closed subsets of $S$, and this functor is not the left Kan extension of its restriction to $\kappa$-small extremally disconnected sets for any $\kappa$ (otherwise any closed subset $Z\subset S$ of an extremally disconnected set $S$ would be an intersection of $\kappa$ clopen subsets of $S$ for some cardinal $\kappa$ independent of $S$ and $Z$). 
\end{warning}

Despite the warning, we note that for any condensed set $T$, the underlying set $T(\ast)$ can be regarded as a topological space $T(\ast)_{\mathrm{top}}$ by endowing it with the quotient topology from $\sqcup_{S\to T} S\to T(\ast)$ where $S$ runs over $\kappa$-small extremally disconnected sets for any $\kappa$ such that $T$ is determined by its values on $\kappa$-small $S$. Then one has formally that for any topological space $X$, $\Hom(T,\underline{X}) = \Hom(T(\ast)_{\mathrm{top}},X)$. However, by the warning, this does not define an adjunction anymore. Putting some separation hypothesis helps:

\begin{proposition} If $X$ is a topological space all of whose points are closed (i.e.~$X$ is $T1$), then $\underline{X}$ is a condensed set for which all maps from points are quasicompact. Conversely, if $T$ is a condensed set such that all maps from points are quasicompact, then $T(\ast)_{\mathrm{top}}$ is a compactly generated space all of whose points are closed.
\end{proposition}

In particular, one gets an adjunction between $T1$ topological spaces and condensed sets for which all maps from points are quasicompact. Note that if $X$ is not $T1$, then $\underline{X}$ is never a condensed set: Indeed, if $X$ is not $T1$ there is a pair of points $s,\eta\in X$ whose subspace topology is not discrete, and then the warning shows that already the subset $\{s,\eta\}\subset X$ does not define a condensed set.

\begin{proof} To see that $\underline{X}$ is a condensed set, we need to see that for some uncountable strong limit cardinal $\kappa$, any map $\tilde{S}\to X$ from any extremally disconnected set $\tilde{S}$ factors over a map $S\to X$ for some $\kappa$-small extremally disconnected $S$. We claim that we can take any $\kappa$ of infinite cofinality $\lambda>|X|$. Then the category $I$ of all factorizations $\tilde{S}\to S\to X$ with $S$ being $\kappa$-small is $\lambda$-filtered. We need to see that there is some $S$ such that the two maps $\tilde{S}\times_S \tilde{S}\to \tilde{S}\to X$ agree. For this, we need to arrange that $\tilde{S}\times_S \tilde{S}$ does not meet the preimage of $(x,y)\in X\times X$ for any pair of distinct points $x,y\in X$. If we can arrange this for any individual $(x,y)$, we can also arrange it for all of them as $I$ is $\lambda$-filtered for $\lambda>|X\times X|$. But note that $(x,y)\in X\times X$ is closed (as $X$ is $T1$), so the preimage of $(x,y)$ is a closed subset $T_{x,y,S}\subset \tilde{S}\times_S \tilde{S}$. The inverse limit of all $T_{x,y,S}$ is empty, so as all $T_{x,y,S}$ are profinite, it follows that one of them is empty.

If $x\in X=\underline{X}(\ast)$ is any point and $S$ is any profinite set with a map to $X$, then $\underline{S}\times_{\underline{X}} \{x\} = \underline{S\times_X \{x\}}$ and $S\times_X \{x\}\subset S$ is a closed subset, which implies that $\{x\}\to \underline{X}$ is quasicompact. Conversely, if $T$ is a condensed set such that for all $x\in T(\ast)$, the map $\{x\}\to T$ is quasicompact, then $\{x\}\to T(\ast)$ is closed: By the definition of the topology on $T(\ast)$, we have to see that for all profinite sets $S$ with a map $S\to T$, the subspace $S\times_T \{x\}\subset S$ is closed. But by assumption it is a quasicompact sub-condensed set $S^\prime\subset S$. Thus, there is some surjection $\tilde{S^\prime}\to S^\prime\subset S$ from some profinite set $\tilde{S}^\prime$, and then $S^\prime$ agrees with the image of $\tilde{S^\prime}\to S$, which is a closed subset of $S$. Thus, $S^\prime\subset S$ is a closed subset, as desired.
\end{proof}

One can translate some properties under this adjunction.

\begin{theorem}\leavevmode
\begin{enumerate}
\item[{\rm (i)}] The functor $X\mapsto \underline{X}$ induces an equivalence between compact Hausdorff spaces $X$ and qcqs condensed sets $T$.
\item[{\rm (ii)}] A compactly generated space $X$ is weak Hausdorff if and only if $\underline{X}$ is quasiseparated. For any quasiseparated condensed set $T$, the topological space $T(\ast)_{\mathrm{top}}$ is compactly generated weak Hausdorff.
\end{enumerate}
\end{theorem}

Recall that $X$ is weak Hausdorff if for any compact Hausdorff space $S$ mapping to $X$, the image is compact Hausdorff (in the subspace topology). Part (ii) implies that compactly generated weak Hausdorff spaces embed fully faithfully in quasiseparated condensed sets (and the inclusion admits a left adjoint). Quasiseparated condensed sets are arguably better behaved: In fact, they are equivalent to the category of ind-compact Hausdorff spaces $\varinjlim_i X_i$ where all transition maps are closed immersions. Any compactly generated weak Hausdorff space is also of this form, but it may happen that taking this filtered colimit in topological spaces produces a compact Hausdorff space again, so the functor $T\mapsto T(\ast)_{\mathrm{top}}$ from quasiseparated condensed sets to compactly generated weak Hausdorff spaces may ``accidentally" identify some objects. For example, there are compact Hausdorff spaces $S$ with a filtered set of proper closed subsets $S_i$ such that the map
\[
\varinjlim_i S_i\to S
\]
of topological spaces is a homeomorphism.\footnote{For an explicit example, use Remark~\ref{rem:firstcountable} to see that for any first-countable compact Hausdorff space $S$, one can take for the $S_i$'s the images of maps from finite disjoint unions of $\mathbb N\cup \{\infty\}$ to $S$, i.e.~finite collections of convergent sequences in $S$.} This can however only happen for rather complicated colimits; in particular, it cannot happen for countable colimits (by the Baire category theorem, effectively). We regard this as a defect of topological spaces: The open subsets are not enough information to remember that full filtered direct system. Passing to (quasiseparated) condensed sets resolves this issue.\footnote{In fact, quasiseparated condensed sets are equivalent to Waelbroeck's compactological spaces, cf.~e.g.~\cite{CompactologicalCondensed} for a detailed discussion.}

\begin{proof} For part (i), note that any compact Hausdorff space $X$ can be written as a quotient $S/R$ of a profinite set $S$ under a closed equivalence relation $R\subset S\times S$. Then $\underline{X} = \underline{S}/\underline{R}$ is qcqs. Conversely, assume that $T$ is a qcqs condensed set. Let $S\to T$ be a surjection from a profinite set; then $R=S\times_T S\subset S\times S$ is a quasicompact sub-condensed set, and thus a closed subset by the argument in the proof of the previous proposition. Thus, $R$ defines a closed equivalence relation on $S$, and the quotient $X=S/R$ is a compact Hausdorff space, such that the map $\underline{X}=\underline{S}/\underline{R}\to T$ is an isomorphism.

For part (ii), note that if $X$ is weak Hausdorff, then in particular all points are closed, so $\underline{X}$ is a condensed set. To see that it is quasiseparated, take any two profinite sets $S_1$, $S_2$ with maps $S_1,S_2\to X$. Then $\underline{S_1}\times_{\underline{X}} \underline{S_2} = \underline{S_1\times_X S_2}$, so it is enough to see that the fibre product $S_1\times_X S_2$ is compact Hausdorff. But for this we may replace $X$ by the image of $S_1\sqcup S_2$, which is compact Hausdorff, and fibre products of compact Hausdorff spaces are again compact Hausdorff.

For the converse, using $X\cong \underline{X}(\ast)_{\mathrm{top}}$, it is enough to see that for any quasiseparated condensed set $T$, the topological space $T(\ast)_{\mathrm{top}}$ is weak Hausdorff. Note that we can write $T$ as a filtered increasing union of quasicompact sub-condensed sets $T_i\subset T$, $i\in I$. Each $T_i$ is of the form $\underline{X_i}$ for a compact Hausdorff space $X_i$, so we have to see that the topological space $X=\varinjlim_i X_i$ is weak Hausdorff.

Thus, let $S$ be any compact Hausdorff space with a map $S\to X=\varinjlim_i X_i$. As remarked above, this does not in general factor over any $X_i$. In any case, for any $i$ the preimage of $X_i\subset X$ is a closed subset $S_i\subset S$, and $S$ is the filtered colimit of the $S_i$. For any $i$, let $\overline{S}_i\subset X_i$ be the image of $S_i\to X_i$, which is a closed subset of $X_i$ and thus compact Hausdorff. The image of $S\to X$ is given by $\overline{S} = \varinjlim_i \overline{S}_i$, equipped with the direct limit topology. We have to see that this is a compact Hausdorff space. Note that $S=\varinjlim_i S_i\to \overline{S} = \varinjlim_i \overline{S}_i$ is a quotient map (as a filtered colimit of quotient maps). Moreover, the induced equivalence relation $R\subset S\times S$ is closed, as it is the filtered colimit of the closed equivalence relations $R_i=S_i\times_{\overline{S}_i} S_i\subset S_i\times S_i$. But $S$ is compact Hausdorff, so that $\overline{S}=S/R$ is then also compact Hausdorff, as desired.
\end{proof}
\newpage

\section{Lecture III: Cohomology}

We have seen that the passage from (nice in a weak sense) topological spaces to condensed sets is lossless on the level of categories. As a next step, we need to see that certain important invariants of topological spaces can be detected on the associated condensed set. One such important invariant are the cohomology groups of a compact Hausdorff space $S$.

There are several ways to define the cohomology of a compact Hausdorff space $S$:

\begin{enumerate}
\item[{\rm (i)}] The singular cohomology groups $H^i_{\mathrm{sing}}(S,\mathbb Z)$. For this, one defines the cochain complex $C^\ast(S)$ with $C^i(S) = \Hom(C_i(S),\mathbb Z)$ where $C_i(S)$ is the free abelian group on the set of maps from the simplex $\Delta^i$ to $S$, where $\Delta^i = \{(x_0,\ldots,x_i)\in [0,1]\mid \sum_{j=0}^i x_j=1\}$.
\item[{\rm (ii)}] The \v{C}ech cohomology groups $H^i_{\mathrm{Cech}}(S,\mathbb Z)$, defined as the filtered colimit of the cohomology groups of the \v{C}ech complexes $C(S,\{U_i\})$ ranging over finite open covers $\{U_i\}$ of $S$.
\item[{\rm (iii)}] The sheaf cohomology groups $H^i_{\mathrm{sheaf}}(S,\mathbb Z)$, defined using the category of sheaves on the topological space $S$.
\end{enumerate}

There is a natural map $H^i_{\mathrm{Cech}}(S,\mathbb Z)\to H^i_{\mathrm{sheaf}}(S,\mathbb Z)$ that is an isomorphism (this holds more generally for paracompact Hausdorff spaces), cf.~\cite[Th\'eor\`eme 5.10.1]{Godement}. If $S$ is a CW complex, they also agree with $H^i_{\mathrm{sing}}(S,\mathbb Z)$, but not in general: If $S$ is a profinite set, then all cohomology groups vanish for $i>0$, and $H^0_{\mathrm{sheaf}}(S,\mathbb Z)$ are the continuous maps $S\to \mathbb Z$, while $H^0_{\mathrm{sing}}(S,\mathbb Z)$ are all maps $S\to \mathbb Z$. In particular, singular cohomology treats $S$ like a disjoint union of points, so one sees that arguably \v{C}ech or sheaf cohomology is better-behaved.

As an example, let $\mathbb T = \mathbb R/\mathbb Z$ be the circle. We will need the following result in the next lecture.

\begin{proposition}\label{prop:cohomcircle} For any set $I$, one has
\[
H^i_{\mathrm{Cech}}(\prod_I \mathbb T,\mathbb Z) = \wedge^i(\bigoplus_I \mathbb Z)\ .
\]
More precisely, in degree $1$ the map
\[
\bigoplus_I \mathbb Z\to H^1_{\mathrm{Cech}}(\prod_I \mathbb T,\mathbb Z)
\]
induced by pullback from all factors is an isomorphism, and the cup product induces isomorphisms
\[
\wedge^i H^1_{\mathrm{Cech}}(\prod_I \mathbb T,\mathbb Z)\cong H^i_{\mathrm{Cech}}(\prod_I \mathbb T,\mathbb Z)\ .
\]
\end{proposition}

\begin{proof} If $I$ is finite, this is the classical computation of the singular cohomology of tori. The case of infinite $I$ follows from the following general fact: If $S_j$, $j\in J$, is a cofiltered system of compact Hausdorff spaces with limit $S=\varprojlim_j S_j$, then the natural map
\[
\varinjlim_j H^i_{\mathrm{Cech}}(S_j,\mathbb Z)\to H^i_{\mathrm{Cech}}(S,\mathbb Z)
\]
is an isomorphism for all $i\geq 0$, cf.~\cite[Chapter X, Theorem 3.1]{EilenbergSteenrod}.
\end{proof}

On the other hand, we can regard $S$ as a condensed set, and can take its cohomology internally to the topos of condensed sets. Equivalently, $S$ is an object of the site of compact Hausdorff spaces equipped with the topology of finite families of jointly surjective maps, and one can take the cohomology groups
\[
H^i_{\mathrm{cond}}(S,\mathbb Z)(\cong \Ext^i_{\Cond(\Ab)}(\mathbb Z[S],\mathbb Z))
\]
on that site. Concretely, these are computed by taking a simplicial hypercover $S_\bullet\to S$ by extremally disconnected sets $S_i$, and forming the cohomology groups of the complex
\[
0\to \Gamma(S_0,\mathbb Z)\to \Gamma(S_1,\mathbb Z)\to \ldots\ .
\]

\begin{theorem}[{\cite[Theorem 3.11]{Dyckhoff}}]\label{thm:cohomcorrect} There are natural functorial isomorphisms
\[
H^i_{\mathrm{sheaf}}(S,\mathbb Z)\cong H^i_{\mathrm{cond}}(S,\mathbb Z)\ .
\]
\end{theorem}

After the proof of the theorem, we will simply write $H^i(S,\mathbb Z)$ to denote any of $H^i_{\mathrm{Cech}}(S,\mathbb Z)$, $H^i_{\mathrm{sheaf}}(S,\mathbb Z)$ and $H^i_{\mathrm{cond}}(S,\mathbb Z)$.

\begin{proof} We check the result first in the case that $S$ is a profinite set. Write $S=\varprojlim_j S_j$ as a cofiltered limit of finite sets $S_j$. Then $H^i_{\mathrm{sheaf}}(S,\mathbb Z)=0$ for $i>0$ and the $H^0$ is given by the continuous maps from $S$ to $\mathbb Z$, i.e.~by $\varinjlim_j \Gamma(S_j,\mathbb Z)$; this also agrees with $H^0_{\mathrm{cond}}(S,\mathbb Z)$ by definition. To check that $H^i_{\mathrm{cond}}(S,\mathbb Z)=0$ for $i>0$, it suffices to check that for any surjection $S^\prime\to S$ of profinite sets, the \v{C}ech complex
\[
0\to \Gamma(S,\mathbb Z)\to \Gamma(S^\prime,\mathbb Z)\to \Gamma(S^\prime\times_S S^\prime,\mathbb Z)\to \ldots
\]
is exact. To verify this, write $S^\prime\to S$ as a cofiltered limit of surjections $S^\prime_j\to S_j$ of finite sets, and pass to the filtered colimit of the exact sequences
\[
0\to \Gamma(S_j,\mathbb Z)\to \Gamma(S^\prime_j,\mathbb Z)\to \Gamma(S^\prime_j\times_{S_j} S^\prime_j,\mathbb Z)\to \ldots
\]
(each of which is exact as $S_j^\prime\to S_j$ is split).

To prove the theorem in general, note that one can define a natural map from the left side to the right side by noting that there is a natural map $\alpha$ of topoi from sheaves on compact Hausdorff spaces over $S$ to sheaves on the topological space $S$. We need to see that $R\alpha_\ast \mathbb Z = \mathbb Z$ in the derived category of abelian sheaves on $S$. We can check this on stalks, so fix $s\in S$. The open neighborhoods $U$ of $s$ are cofinal with the closed neighborhoods $V$ of $s$. This implies
\[
(R\alpha_\ast \mathbb Z)_s = \varinjlim_{U\ni s} R\Gamma(U,R\alpha_\ast \mathbb Z) = \varinjlim_{U\ni s} R\Gamma_{\mathrm{cond}}(U,\mathbb Z) = \varinjlim_{V\ni s} R\Gamma_{\mathrm{cond}}(V,\mathbb Z)\ .
\]
Pick a simplicial hypercover $S_\bullet\to S$ by profinite sets $S_i$. Then $R\Gamma_{\mathrm{cond}}(S,\mathbb Z)$ is computed by the complex
\[
0\to \Gamma(S_0,\mathbb Z)\to \Gamma(S_1,\mathbb Z)\to \ldots
\]
by the vanishing of $H^i_{\mathrm{cond}}(S_j,\mathbb Z)$ for $i>0$. Similarly, for any closed $V\subset S$, one can compute $R\Gamma(V,\mathbb Z)$ by
\[
0\to \Gamma(S_0\times_S V,\mathbb Z)\to \Gamma(S_1\times_S V,\mathbb Z)\to \ldots\ .
\]
Passing to the filtered colimit over all closed neighborhoods $V$ of $s$, one gets the complex
\[
0\to \Gamma(S_0\times_S \{s\},\mathbb Z)\to \Gamma(S_1\times_S \{s\},\mathbb Z)\to \ldots
\]
which computes $R\Gamma(\{s\},\mathbb Z) = \mathbb Z$. In other words,
\[
(R\alpha_\ast \mathbb Z)_s = \varinjlim_{V\ni s} R\Gamma_{\mathrm{cond}}(V,\mathbb Z) = \mathbb Z\ ,
\]
as desired.
\end{proof}

We need another computation, namely the cohomology of the condensed abelian group $\mathbb R$ on compact Hausdorff $S$.

\begin{theorem}\label{thm:enoughcontfcts} For any compact Hausdorff space $S$, one has
\[
H^i_{\mathrm{cond}}(S,\mathbb R)=0
\]
for $i>0$, while $H^0(S,\mathbb R)=C(S,\mathbb R)$ is the space of continuous real-valued functions on $S$.

More precisely, if $S_\bullet\to S$ is any simplicial hypercover of $S$ by profinite sets $S_i$, the complex of Banach spaces
\[
0\to C(S,\mathbb R)\to C(S_0,\mathbb R)\to C(S_1,\mathbb R)\to \ldots
\]
satisfies the following quantitative version of exactness: if $f\in C(S_i,\mathbb R)$ satisfies $df=0$ and $\epsilon>0$, then there exists $g\in C(S_{i-1},\mathbb R)$ with $dg=f$ and such that $||g||\leq (1+\epsilon)||f||$ (where we endow all $C(-,\mathbb R)$ with the supremum norm).
\end{theorem}

\begin{proof} Assume first that $S$ and all $S_i$ are finite. Then the hypercover splits, which explicitly means that the contravariant functor from linearly ordered finite sets to sets (sending the empty set to $S$ and $\{0,\ldots,i\}$ to $S_i$) extends to a contravariant functor from pointed (at the minimal element) linearly ordered finite sets to sets (where we map linearly ordered finite sets to pointed linearly ordered finite sets via adjoining an extra minimal element). By the Dold-Kan correspondence, this induces a contracting chain homotopy of
\[
0\to C(S,\mathbb R)\to C(S_0,\mathbb R)\to C(S_1,\mathbb R)\to \ldots
\]
where the contracting homotopy is induced by pullback along certain maps $S_{i-1}\to S_i$; in particular, the contracting homotopy has norm $\leq 1$. Thus, if $f\in C(S_i,\mathbb R)$ satisfies $df=0$, then $f=d(h_i(f))$ for the map $h_i: C(S_i,\mathbb R)\to C(S_{i-1},\mathbb R)$ induced by the contracting homotopy $h$. As $||h_i(f)||\leq ||f||$, this gives the desired result.

Now if $S$ and all $S_i$ are profinite, one can write the hypercover $S_\bullet\to S$ as a cofiltered limit of hypercovers $S_{\bullet,j}\to S_j$ of finite sets $S_j$ by finite sets $S_{i,j}$, along the cofiltered index category $J\ni j$. Passing to the filtered colimit over $j$, we get an exact sequence of normed $\mathbb R$-vector spaces
\[
0\to \varinjlim_j C(S_j,\mathbb R)\to \varinjlim_j C(S_{0,j},\mathbb R)\to \varinjlim_j C(S_{1,j},\mathbb R)\to \ldots
\]
with the property that if $f\in \varinjlim_j C(S_{i,j},\mathbb R)$ satisfies $df=0$ then there is $g\in \varinjlim_j C(S_{i-1,j},\mathbb R)$ with $dg=f$ and $||g||\leq ||f||$.

We can now pass to the completion. More precisely, note that $\varinjlim_j C(S_{i,j},\mathbb R)\to C(S_i,\mathbb R)$ is injective, dense, and isometric, and the target is the completion of the source.\footnote{The individual maps $C(S_{i,j},\mathbb R)\to C(S_i,\mathbb R)$ may not be isometric embeddings, but they are in the pro-sense, as the system $(S_{i,j})_j$ is pro-isomorphic to a system with surjective transition maps. We will pretend the individual maps are isometric embeddings; this can be corrected by increasing $j$ when necessary.} Take $f\in C(S_i,\mathbb R)$ with $df=0$, $f\neq 0$. We can find some $j$ and $f_j\in C(S_{i,j},\mathbb R)$ such that $||f-f_j||\leq \epsilon ||f||$. In particular, this implies that $d(f_j) = - d(f-f_j)$ has norm $\leq (i+2)\epsilon ||f||$ (as $d=d_i$ is an alternating sum of $i+2$ restriction maps along the face maps $S_{i+1}\to S_i$), so by the above we can replace $f_j$ by another function to ensure $df_j=0$ while retaining $||f-f_j||\leq (i+3)\epsilon ||f||$; up to replacing $\epsilon$, we can thus ensure $||f-f_j||\leq \epsilon ||f||$ and $df_j=0$. On the other hand, we can now find $g_j\in C(S_{i-1,j},\mathbb R)$ with $dg_j= f_j$ and
\[
||g_j||\leq ||f_j||\leq (1+\epsilon)||f||\ .
\]
The condition $dg_j = f_j$ implies
\[
||f-dg_j||\leq \epsilon ||f||\ .
\]
Setting $g^{(0)} = g_j$ and $f^{(1)} = f - dg^{(0)}$ and repeating the process then yields a convergent sequence
\[
f = d(g^{(0)} + g^{(1)} + \ldots) = dg
\]
with $||g^{(m)}||\leq (1+\epsilon) ||f^{(m)}||\leq \epsilon^m (1+\epsilon) ||f||$ and so
\[
||g||\leq \frac{(1+\epsilon)}{1-\epsilon}||f||\ ;
\]
up to redefining $\epsilon$, this is the claim.

Now assume that $S$ is any compact Hausdorff space with a hypercover $S_\bullet\to S$ by profinite sets $S_i$. Let $f\in C(S_i,\mathbb R)$ satisfy $df=0$ and assume $f\neq 0$. For each $s\in S$, the restriction $f_s = f|_{S_i\times_S \{s\}}\in C(S_i\times_S\{s\},\mathbb R)$ is the boundary $f_s=dg_s$ of some $g_s\in C(S_{i-1}\times_S \{s\},\mathbb R)$ with $||g_s||\leq (1+\epsilon)||f_s||$, by the profinite case already handled. We may find an extension $\tilde{g}_s\in C(S_{i-1},\mathbb R)$ of $g_s$ with $||\tilde{g}_s||\leq ||g_s||$ by Tietze's extension theorem. Moreover, we can find an open neighborhood $U_s\subset S$ of $s$ such that
\[
||(d\tilde{g}_s-f)|_{S_i\times_S U}||\leq \epsilon ||f||\ .
\]
By compactness, finitely many of the $U_s$ cover $S$; so choose $U_1,\ldots,U_n$ covering $S$ and $g_1,\ldots,g_n\in C(S_{i-1},\mathbb R)$ with $||g_j||\leq (1+\epsilon)||f||$ for $j=1,\ldots,n$ such that
\[
||(dg_j-f)|_{S_i\times_S U_j}||\leq \epsilon ||f||\ .
\]
Choose a partition of unity $1=\sum_{j=1}^n \rho_j$ for functions $\rho_j\in C(S,\mathbb R)$ with nonnegative values and support in $U_j$, and set $g^{(0)} = \sum_{j=1}^n \rho_j g_j$. Then
\[
||g^{(0)}||\leq (1+\epsilon)||f||
\]
and
\[
||dg^{(0)}-f|| = ||\sum_{j=1}^n \rho_j(dg_j-f)||\leq \epsilon ||f||\ .
\]
Setting $f^{(1)} = f - dg^{(0)}$ we have achieved that $||f^{(1)}||\leq \epsilon ||f||$ while $||g^{(0)}||\leq (1+\epsilon)||f||$. Repeating the process, we can write
\[
f = d(g^{(0)} + g^{(1)} + \ldots) = dg
\]
where $||g^{(m)}||\leq (1+\epsilon)||f^{(m)}||\leq \epsilon^m (1+\epsilon)||f||$; in particular
\[
||g||\leq \frac{1+\epsilon}{1-\epsilon}||f||\ .
\]
Up to redefining $\epsilon$, this is exactly the claim.
\end{proof}
\newpage

\section{Lecture IV: Locally compact abelian groups}

A particularly nice class of topological abelian groups are the locally compact abelian groups $A$. Recall that a topological space $X$ is locally compact if any point $x\in X$ has a basis of compact Hausdorff neighborhoods (that are not necessarily open); a topological abelian group $A$ is then locally compact if it is locally compact as a topological space. In particular, local fields such as $\mathbb R$ and $\mathbb Q_p$ are locally compact.

We recall the following structure result on locally compact abelian groups.\footnote{We refer to \cite{HoffmannSpitzweck} and the references therein for basic results on locally compact abelian groups, and their derived category.}

\begin{theorem}\leavevmode
\begin{enumerate}
\item[{\rm (i)}] Let $A$ be a locally compact abelian group. Then there is an integer $n$ and an isomorphism $A\cong \mathbb R^n\times A^\prime$ where $A^\prime$ admits a compact open subgroup. In particular, $A^\prime$ is an extension of a discrete abelian group by a compact abelian group.
\item[{\rm (ii)}] The Pontrjagin duality functor $A\mapsto \mathbb D(A) = \Hom(A,\mathbb T)$, where $\mathbb T=\mathbb R/\mathbb Z$ is the circle group, takes values in locally compact abelian groups, and induces a contravariant autoduality of the category of locally compact abelian groups. The biduality map $A\to \mathbb D(\mathbb D(A))$ is an isomorphism.
\item[{\rm (iii)}] The Pontrjagin duality functor $A\mapsto \mathbb D(A)$ restricts to a contravariant duality between compact abelian groups and discrete abelian groups.
\end{enumerate}
\end{theorem}

Here, one endows $\Hom(A,B)$ with the compact-open topology. This matches the passage to condensed abelian groups:

\begin{proposition} Let $A$ and $B$ be Hausdorff topological abelian groups and assume that $A$ is compactly generated. Then there is a natural isomorphism of condensed abelian groups
\[
\intHom(\underline{A},\underline{B})\cong \underline{\Hom(A,B)}\ .
\]
\end{proposition}

\begin{proof} First, we construct a map
\[
\intHom(\underline{A},\underline{B})\to \underline{\Hom(A,B)}\ .
\]
For any profinite set $S$, the $S$-valued points of the left are $\Hom(\underline{A}\otimes \mathbb Z[S],\underline{B})$. Such a map induces a map $S\to \Hom(A,B)$ by evaluating at points. We claim that this map is continuous for the compact open topology on $\Hom(A,B)$. But $\mathbb Z[\underline{A}\times S] = \mathbb Z[\underline{A}]\otimes \mathbb Z[S]$ surjects onto $\underline{A}\otimes \mathbb Z[S]$, so any map $\underline{A}\otimes \mathbb Z[S]\to \underline{B}$ determines and is determined by a map of condensed sets $\underline{A\times S}=\underline{A}\times S\to \underline{B}$, i.e.~continuous maps $A\times S\to B$. The latter are precisely continuous maps from $S$ to the topological space of continuous maps from $A$ to $B$ equipped with the compact-open topology, as desired.

The preceding paragraph also shows that the map
\[
\intHom(\underline{A},\underline{B})\to \underline{\Hom(A,B)}
\]
is injective. For surjectivity, we use the partial resolution
\[
\mathbb Z[\underline{A}\times \underline{A}]\to \mathbb Z[\underline{A}]\to \underline{A}\to 0\ ,
\]
where the first map sends a generator $[(a_1,a_2)]$ to $[a_1+a_2]-[a_1]-[a_2]$. Given a profinite set $S$ and a map $S\to \Hom(A,B)$ that is continuous for the compact-open topology, we get a map $\mathbb Z[\underline{A}]\otimes \mathbb Z[S]\to \underline{B}$ by reversing the steps above. We need to that this factors over $\underline{A}\otimes \mathbb Z[S]$. Equivalently, we need to see that the induced map
\[
\mathbb Z[\underline{A}\times \underline{A}]\otimes \mathbb Z[S]\to \mathbb Z[\underline{A}]\otimes \mathbb Z[S]\to \underline{B}
\]
vanishes. This is equivalent to a map $\mathbb Z[\underline{A}\times \underline{A}\times S]\to \underline{B}$, i.e.~a map of condensed sets $\underline{A\times A\times S} = \underline{A}\times \underline{A}\times S\to \underline{B}$, i.e.~a continuous map $A\times A\times S\to B$. This map is trivial as for all $s\in S$, the map $A\to B$ parametrized by $s$ is a group homomorphism.
\end{proof}

One can compute all $R\intHom$'s between the condensed abelian groups associated to locally compact abelian groups. The key computation is the following.

\begin{theorem}\label{thm:RHomoutofcompact} Let $A=\prod_I \mathbb T$ be the compact condensed abelian group, where $I$ is any (possibly infinite) set.
\begin{enumerate}
\item[{\rm (i)}] For any discrete abelian group $M$,
\[
R\intHom(A,M) = \bigoplus_I M[-1]\ ,
\]
where the map $\bigoplus_I M[-1]\to R\intHom(A,M)$ is induced by the maps
\[
M[-1] = R\intHom(\mathbb Z[1],M)\to R\intHom(\mathbb R/\mathbb Z,M)\xrightarrow{p_i^\ast} R\intHom(\prod_I \mathbb R/\mathbb Z,M) = R\intHom(A,M)\ ,
\]
using pullback under the projection $p_i: \prod_I \mathbb R/\mathbb Z\to \mathbb R/\mathbb Z$ to the $i$-th factor, $i\in I$.
\item[{\rm (ii)}]
\[
R\intHom(A,\mathbb R)=0\ .
\]
\end{enumerate}
\end{theorem}

\begin{remark} A sample consequence of the theorem is that
\[
R\intHom(\mathbb R,\mathbb R)=\mathbb R\ ,
\]
(where homomorphisms are over $\mathbb Z$!), as the difference between $R\intHom(\mathbb R,\mathbb R)$ and $R\intHom(\mathbb Z,\mathbb R)=\mathbb R$ is given by $R\intHom(\mathbb R/\mathbb Z,\mathbb R)=0$.
\end{remark}

For the proof, we need the following general resolution, extending the partial resolution used in the proof of the proposition.

\begin{theorem}[Eilenberg--MacLane, Breen, Deligne]\label{thm:deligneresolution} There is a functorial resolution of an abelian group $A$ of the form
\[
\ldots \to \bigoplus_{j=1}^{n_i} \mathbb Z[A^{r_{i,j}}]\to \ldots \to \mathbb Z[A^3]\oplus \mathbb Z[A^2]\to \mathbb Z[A^2]\to \mathbb Z[A]\to A\to 0
\]
where all $n_i$ and $r_{i,j}$ are nonnegative integers.
\end{theorem}

\begin{remark} The theorem does not seem to have a published proof in the literature. Eilenberg--MacLane, \cite[Theorem 6.2]{EilenbergMacLaneMultiplicative}, prove that there is such a resolution but allowing possibly infinite direct sums. That argument is rather formal and does not use much of their work. Breen, \cite{BreenAdditive}, relying on their other work on Eilenberg--MacLane spaces, establishes a similar resolution up to bounded torsion in each degree, and a variant involving $\mathbb Z[A^{r_1}\times \mathbb Z^{r_2}]$ instead. In \cite{BerthelotBreenMessing}, Berthelot--Breen--Messing mention that Deligne has an unpublished proof of the result.

In the appendix to this lecture, we give a proof of this theorem, which (as Breen's previous work on the subject) makes use of some (stable) homotopy theory, in the form of the homology of Eilenberg--MacLane spaces.
\end{remark}

\begin{remark} The functoriality of the construction ensures that it works automatically for abelian group objects in any topos: Apply the resolution to the values $A(U)$ of the abelian sheaf $A$ on all opens $U$ of a defining site, and sheafify. This way, we see that we can also apply it to condensed abelian groups.
\end{remark}

\begin{corollary}\label{cor:delignespecseq} For condensed abelian groups $A$, $M$ and an extremally disconnected set $S$, there is a spectral sequence
\[
E_1^{i_1i_2} = \prod_{j=1}^{n_{i_1}} H^{i_2}(A^{r_{i_1,j}}\times S,M)\Rightarrow \underline{\Ext}^{i_1+i_2}(A,M) (S)
\]
that is functorial in $A$, $M$ and $S$.
\end{corollary}

\begin{proof} Use the resolution
\[
\ldots \to \bigoplus_{j=1}^{n_i} \mathbb Z[A^{r_{i,j}}\times S]\to \ldots \to \mathbb Z[A^3\times S]\oplus \mathbb Z[A^2\times S]\to \mathbb Z[A^2\times S]\to \mathbb Z[A\times S]\to A\otimes \mathbb Z[S]\to 0
\]
and apply $R\Hom(-,M)$, noting that
\[
R\intHom(A,M)(S) = R\Hom(A\otimes \mathbb Z[S],M)\ .\qedhere
\]
\end{proof}

One might think that the Breen--Deligne resolution is not suitable for explicit computations as it is not an explicit resolution. The following proof shows that it does not matter that the Breen--Deligne resolution is inexplicit.

\begin{proof}[Proof of Theorem~\ref{thm:RHomoutofcompact}] We start with part (i). Assume first that $I$ is finite; decomposing into a finite direct sum, we can then assume that $I$ has one element, so we need to compute $R\intHom(\mathbb R/\mathbb Z,M)$. The claim is that the map
\[
M[-1] = R\intHom(\mathbb Z[1],M)\to R\intHom(\mathbb R/\mathbb Z,M)
\]
is an isomorphism; equivalently, we need to see that
\[
R\intHom(\mathbb R,M)=0
\]
for any discrete abelian group $M$. For this, it is enough to prove that the map $0\to \mathbb R$ induces an isomorphism upon applying $R\intHom(-,M)$; and for this, we use the spectral sequence from Corollary~\ref{cor:delignespecseq}, and observe that it suffices to see that the map
\[
H^i(\mathbb R^r\times S,M)\to H^i(S,M)\ ,
\]
induced via pullback along $0\times \mathrm{id}: S\to \mathbb R^r\times S$, is an isomorphism for all profinite sets $S$ and integers $r\geq 0$. The left-hand side $R\Gamma(\mathbb R^r\times S,M)$ is the derived limit of $R\Gamma([-n,n]^r\times S,M)$ over all $n\geq 1$, so it is enough to prove that $H^i([-n,n]^r\times S,M)\to H^i(S,M)$ is an isomorphism. Now both sides identify with \v{C}ech/sheaf cohomology by Theorem~\ref{thm:cohomcorrect}, and it is known that the latter is homotopy-invariant.

For general $I$, it is now enough to show that the map
\[
\varinjlim_{J\subset I} R\intHom(\prod_J \mathbb T,M)\to R\intHom(\prod_I \mathbb T,M)
\]
is an isomorphism, where $J$ runs over finite subsets of $I$. Again, Corollary~\ref{cor:delignespecseq} reduces this to
\[
\varinjlim_{J\subset I} H^i(\prod_J\mathbb T^r\times S,M)\to H^i(\prod_I\mathbb T^r\times S,M)
\]
being an isomorphism, which follows from the comparison with \v{C}ech/sheaf cohomology, and the known result there.

It remains to prove part (ii). For this, we use again the resolution $F(A)_\bullet$ of $A$ all of whose terms are of the form $F(A)_i = \bigoplus_{j=1}^{n_i} \mathbb Z[A^{r_{i,j}}]$. By Theorem~\ref{thm:enoughcontfcts}, we see that $R\intHom(A,\mathbb R)(S)$ is computed by the complex
\[
0\to \bigoplus_{j=1}^{n_0} C(A^{r_{0,j}}\times S,\mathbb R)\to \bigoplus_{j=1}^{n_1} C(A^{r_{1,j}}\times S,\mathbb R)\to \ldots\ ,
\]
which is a complex of Banach spaces.

Our idea is now that multiplication by $2$ on $A$ is bounded, while it is unbounded in $\mathbb R$: This implies for example that any map $A\to \mathbb R$ must be zero, as the image is both bounded and stable under multiplication by $2$. Extending this to higher cohomology, we use that the maps $2: F(A)\to F(A)$ of multiplication by $2$ and $[2]: F(A)\to F(A)$ induced by multiplication by $2$ of $A$ (using that the construction of $F(A)$ is functorial in $A$) are homotopic, via some homotopy $h_\bullet: F(A)_\bullet\to F(A)_{\bullet+1}$, cf.~Proposition~\ref{prop:homotopymultiply}.

Assume that $f\in \bigoplus_{j=1}^{n_i} C(A^{r_{i,j}}\times S,\mathbb R)$ satisfies $df=0$. Then $2f=[2]^\ast(f) + d(h_{i-1}^\ast(f))$, or equivalently
\[
f = \tfrac 12 [2]^\ast(f) + d(\tfrac 12 h_{i-1}^\ast(f))\ .
\]
Continuing, we find
\[
f = \tfrac 1{2^n} [2^n]^\ast(f) + d(\tfrac 12 h_{i-1}^\ast(f) + \tfrac 14 h_{i-1}^\ast([2]^\ast(f)) + \ldots + \tfrac 1{2^n} h_{i-1}^\ast([2^{n-1}]^\ast(f)))\ .
\]
Note that $[2^n]^\ast(f)\in \bigoplus_{j=1}^{n_i} C(A^{r_{i,j}}\times S,\mathbb R)$ stays bounded (in fact $||[2^n]^\ast(f)||\leq ||f||$) and
\[
h_{i-1}^\ast: \bigoplus_{j=1}^{n_i} C(A^{r_{i,j}}\times S,\mathbb R)\to \bigoplus_{j=1}^{n_{i-1}} C(A^{r_{i-1,j}}\times S,\mathbb R)
\]
is a map of Banach spaces and thus has bounded norm. It follows that we can take the limit $n\to \infty$ to get
\[
f=d(\tfrac 12 h_{i-1}^\ast(f) + \tfrac 14 h_{i-1}^\ast([2]^\ast(f)) + \ldots)\ ,
\]
as desired.
\end{proof}

In \cite{HoffmannSpitzweck}, Hoffmann--Spitzweck define a bounded derived category $D^b(\LCA)$ of the quasiabelian (but not abelian!) category $\LCA$ of locally compact abelian groups. This can be defined by starting with the category $\Ch^b(\LCA)$ of bounded complexes of locally compact abelian groups, and inverting maps whose cones are strictly exact. One gets a functor
\[
D^b(\LCA)\to D(\Cond(\Ab))
\]
as the functor $\Ch^b(\LCA)\to \Ch(\Cond(\Ab))$ sends strictly exact complexes of locally compact abelian groups to acyclic complexes. To see the latter, the key point is that a strict surjection $A\to B$ of locally compact abelian groups induces a surjection $\underline{A}\to \underline{B}$, which follows from such surjections being open, so the image of any compact neighborhood of the identity in $A$ is a compact neighborhood of the identity in $B$.

Comparing our formalism and computations with the ones from \cite{HoffmannSpitzweck}, we get the following corollary.

\begin{corollary} The functor $D^b(\LCA)\to D(\Cond(\Ab))$ is fully faithful.
\end{corollary}

\begin{proof} Let $A$ and $B$ be locally compact abelian groups. We need to see that the map
\[
R\Hom_{\LCA}(A,B)\to R\Hom_{\Cond(\Ab)}(\underline{A},\underline{B})
\]
is an isomorphism. By the structure result, we can reduce to the case that $A$ is either $\mathbb R$, discrete, or compact. In fact, the discrete case reduces to $A=\mathbb Z$ (by filtered colimits and resolutions), which is clear, and the case of $\mathbb R$ reduces to $\mathbb T=\mathbb R/\mathbb Z$, so we can assume that $A$ is compact. Dually to a $2$-term free resolution of an abelian group, we can then assume that $A=\prod_I \mathbb T$ for some set $I$. Similarly, applying the structure result for $B$, we can assume that $B$ is either $\mathbb R$, discrete, or compact. The compact case again reduces to $B=\prod_J \mathbb T$, and then to $B=\mathbb T$ by pulling out the product, which follows from the case of discrete $B$ and $B=\mathbb R$. Finally, we have reduced to $A=\prod_I \mathbb T$ and $B$ being either discrete or $\mathbb R$. In these cases, Theorem~\ref{thm:RHomoutofcompact} computes the right-hand side, while \cite[Example 4.11]{HoffmannSpitzweck} gives the same answer for the left-hand side in the discrete case, and \cite[Proposition 4.15]{HoffmannSpitzweck} also gives the same answer when $B=\mathbb R$.
\end{proof}

We remark that more important than this equivalence with the (rather abstract) $D^b(\LCA)$ is the fact that all $R\intHom_{D(\Cond(\Ab))}(-,-)$ between locally compact abelian groups can be computed (and give the expected answer). In particular, the proof shows that
\[
\underline{\Ext}^i_{\Cond(\Ab)}(\underline{A},\underline{B})=0
\]
for all $i\geq 2$ for all locally compact abelian groups $A$, $B$.

We remark that we used critically (AB4*) in $\Cond(\Ab)$ when mapping into an infinite product of copies of $\mathbb T$. This is one situation where the existence of compact projective generators, and thus of extremally disconnected sets, is critical.

\newpage

\section*{Appendix to Lecture IV: Resolutions of abelian groups}

The goal of this appendix is to provide a proof of the following unpublished theorem due to Deligne.

\begin{theorem}\label{thm:deligneresolutionapp} There is a functorial resolution of an abelian group $A$ of the form
\[
\ldots \to \bigoplus_{j=1}^{n_i} \mathbb Z[A^{r_{i,j}}]\to \ldots \to \mathbb Z[A^3]\oplus \mathbb Z[A^2]\to \mathbb Z[A^2]\to \mathbb Z[A]\to A\to 0
\]
where all $n_i$ and $r_{i,j}$ are nonnegative integers.
\end{theorem}

We start by providing appropriate categorical framework to address this question. Let $\Latt$ be the category of finite free $\mathbb Z$-modules, and consider the category $\Fun(\Latt,\Ab)$ of (arbitrary) functors $\Latt\to \Ab$. This is an abelian category where all limits and colimits are formed termwise. As a functor category, it again has all limits and colimits and satisfies (AB3-6) and (AB3*-4*). Moreover, it has compact projective generators. A set of such compact projective generators is given by the functors taking $P\in \Latt$ to $\mathbb Z[P^n]$ for some $n\geq 0$. Indeed, $P^n=\Hom_{\Latt}(\mathbb Z^n,P)$, so for any $F\in \Fun(\Latt,\Ab)$, one has
\[
\Hom_{\Fun(\Latt,\Ab)}(\mathbb Z[\Hom_{\Latt}(\mathbb Z^n,-)],F) = \Hom_{\Fun(\Latt,\Sets)}(\Hom_{\Latt}(\mathbb Z^n,-),F) = F(\mathbb Z^n)\ ,
\]
which commutes with all limits and colimits, and as $F$ is determined by the values $F(\mathbb Z^n)$, one also sees that they form a generating family.

Consider for the moment any abelian category $\mathcal A$ with all colimits admitting compact projective generators, and fix a set $\mathcal A_0\subset \mathcal A$ of such compact projective generators. (This implies in particular that filtered colimits are exact.) In particular, we may form the derived category $D(\mathcal A)$. Recall the following definition.

\begin{definition} An object $X\in \mathcal A$ is $n$-pseudocoherent for some integer $n\geq 1$ if the functors $\Ext^i_{\mathcal A}(X,-): \mathcal A\to \Ab$ commute with all filtered colimits for all $i=0,\ldots,n-1$. The object $X$ is pseudocoherent if it is $n$-pseudocoherent for all $n\geq 0$.
\end{definition}

\begin{proposition}\label{prop:pseudocoh} An object $X\in \mathcal A$ is $n$-pseudocoherent if and only if there is a partial resolution
\[
\bigoplus_{j=1}^{k_n} P_{n,j}\to \ldots\to \bigoplus_{j=1}^{k_i} P_{i,j}\to \ldots\to \bigoplus_{j=0}^{k_0} P_{0,j}\to X\to 0
\]
where all $P_{i,j}\in \mathcal A_0$ are in the fixed set of compact projective generators, and the $k_i$ are nonnegative integers. More precisely, given any shorter partial resolution of $X$ of this form, it can be prolonged to a partial resolution of length $n$.

In particular, $X\in \mathcal A$ is pseudocoherent if and only if there is a resolution
\[
\ldots\to \bigoplus_{j=1}^{k_i} P_{i,j}\to \ldots\to \bigoplus_{j=0}^{k_0} P_{0,j}\to X\to 0
\]
where all $P_{i,j}\in \mathcal A_0$ are in the fixed set of compact projective generators, and the $k_i$ are nonnegative integers.
\end{proposition}

\begin{remark} Extending the above definition to $n=0$, we say that an object $X$ is $0$-pseudocoherent if it is finitely generated; this way, the characterization of the proposition in terms of resolutions extends to $n=0$.
\end{remark}

\begin{proof} It is clear that if $X$ admits such a resolution, then it is $n$-pseudocoherent. For the converse, we argue by induction on $n$: Assume first that $n=1$. In that case, $X$ is $1$-pseudocoherent if and only if $\Hom_{\mathcal A}(X,-)$ commutes with all filtered colimits, i.e.~$X$ is compact in $\mathcal A$. Choose some presentation of $X$ of the form
\[
\bigoplus_{j\in J_1} P_{1,j}\to \bigoplus_{j\in J_0} P_{0,j}\to X\to 0\ .
\]
Exhausting $J_0$ and $J_1$ by finite subsets writes $X$ as a filtered colimit of finitely presented $X_i$. As $X$ is compact, the isomorphism $X\to \varinjlim_i X_i$ factors over some $X_i$, so that $X$ is a retract of some $X_i$, and thus $X$ itself is finitely presented.

Now assume that $X$ is $n$-pseudocoherent and a resolution
\[
\bigoplus_{j=1}^{k_{n-2}} P_{n-2,j}\to \ldots\to \bigoplus_{j=0}^{k_0} P_{0,j}\to X\to 0
\]
has been constructed. Let $X_{n-1}$ be the kernel of the first map. It follows from a direct diagram chase that $X_{n-1}$ is $1$-pseudocoherent (argue by descending induction that the image of the $i$-th map is $i$-pseudocoherent). By the case $n=1$, we can find a two-step resolution of $X_n$, as desired.
\end{proof}

We can now reformulate Theorem~\ref{thm:deligneresolutionapp}.

\begin{theorem}\label{thm:deligneresolutionpseudocoh} The functor $\Latt\to \Ab: P\mapsto P$ is pseudocoherent in $\Fun(\Latt,\Ab)$.
\end{theorem}

Let us first explain how this proves Theorem~\ref{thm:deligneresolutionapp}.

\begin{proof}[Theorem~\ref{thm:deligneresolutionpseudocoh}$\Rightarrow$ Theorem~\ref{thm:deligneresolutionapp}] Applying Proposition~\ref{prop:pseudocoh} to $\mathcal A=\Fun(\Latt,\Ab)$, the functor $\Latt\to \Ab: P\mapsto P$, and the class of compact projective generators $P\mapsto \mathbb Z[P^r]$, we get a functorial resolution
\[
\ldots \to \bigoplus_{j=1}^{k_i} \mathbb Z[P^{r_{i,j}}]\to \ldots\to \bigoplus_{j=1}^{k_0} \mathbb Z[P^{r_{0,j}}]\to P\to 0
\]
in $P\in \Latt$. In fact, one can construct by hand a partial resolution
\[
\mathbb Z[P^3]\oplus \mathbb Z[P^2]\to \mathbb Z[P^2]\to \mathbb Z[P]\to P\to 0\ ,
\]
cf.~\cite{BerthelotBreenMessing}, and Proposition~\ref{prop:pseudocoh} ensures that the above resolution can be chosen to prolong this one.

It remains to extend this from finite free $\mathbb Z$-modules $P$ to all abelian groups $A$. For this, we observe that the differentials in the above resolution are given by universal formulas. Indeed, maps
\[
(P\mapsto \mathbb Z[P^{r_1}])\to (P\mapsto \mathbb Z[P^{r_2}])
\]
in $\Fun(\Latt,\Ab)$ are given by $\mathbb Z[(\mathbb Z^{r_1})^{r_2}]$, i.e.~finite sums of matrices $M\in M_{r_1\times r_2}(\mathbb Z)$. The map corresponding to a matrix $M$ is given by sending generators $[(p_1,\ldots,p_{r_1})]$, $p_i\in P$, of $\mathbb Z[P^{r_1}]$ to $[M(p_1,\ldots,p_{r_1})]$, where $M(p_1,\ldots,p_{r_1})\in P^{r_2}$ is a generator of $\mathbb Z[P^{r_2}]$ obtained by applying the matrix $M$.

In particular, these universal formulas extend to all abelian groups, leading to a functorial sequence
\[
\ldots \to \bigoplus_{j=1}^{k_i} \mathbb Z[A^{r_{i,j}}]\to \ldots\to \bigoplus_{j=1}^{k_0} \mathbb Z[A^{r_{0,j}}]\to A\to 0
\]
in the abelian group $A$. We claim that this is indeed a complex, and exact: if $A$ is free, this follows via passage to filtered colimits from the case of finite free $A$. In general, $A$ admits a surjection from a free abelian group, which is enough to ensure that it is indeed a complex. To see that it is a resolution, take a simplicial resolution $A_\bullet$ of $A$ by free abelian groups $A_k$, and write down the corresponding double complex. The resulting double complex is equivalent to the complex for $A$, as $A$ and $\mathbb Z[A^r]$ for all $r\geq 0$ are resolved by $A_\bullet$ and $\mathbb Z[A_\bullet^r]$. On the other hand, for all $k\geq 0$, the row of the corresponding double complex is exact by the free case already established, giving the result in general.
\end{proof}

Before embarking on the proof of Theorem~\ref{thm:deligneresolutionpseudocoh}, we establish the following general lemma that allows us to use approximate resolutions in an inductive proof of pseudocoherence.

\begin{lemma}\label{lem:approximateresolution} Let $\mathcal A$ be an abelian category as above. Let $X\in \mathcal A$ and assume given a complex
\[
C: \ldots\to P_i\to \ldots \to P_0\to X\to 0
\]
in $\mathcal A$ such that all $P_i\in \mathcal A$ are compact projective and $P_0\to X$ is surjective. If $H_i(C)$ is $n-i-1$-pseudocoherent for all $i=0,\ldots,n-1$, then $X$ is $n$-pseudocoherent.
\end{lemma}

\begin{proof} This follows from a direct diagram chase, noting that one can replace $C$ by its stupid truncation in degrees $\leq n$.
\end{proof}

\begin{proof}[Proof of Theorem~\ref{thm:deligneresolutionpseudocoh}] We argue by induction on $n$ that the functor $P\mapsto P$ is $n$-pseudo\-coherent. The partial resolution $\mathbb Z[P^2]\to \mathbb Z[P]\to P\to 0$ shows that it is $1$-pseudocoherent.

Recall the bar construction $BP$, which is the simplicial set with $i$-simplices $(BP)_i = P^i$ and where the simplicial structure maps encode the group structure on $P$. As $P$ is abelian, this is in fact a simplicial abelian group, so one can inductively form the $n$-fold simplicial abelian group $B^nP$. It is clear from the construction that each term $(B^nP)_{i_1,\ldots,i_n}$ is a finite power $P^{r_{i_1,\ldots,i_n,n}}$ of $P$, and $r_{i_1,\ldots,i_n,n}=0$ if one of the $i_j=0$. In particular, forming the associated chain complex
\[
\mathbb Z[B^nP]
\]
whose degree $i$ term is $\bigoplus_{i_1+\ldots+i_n=i} \mathbb Z[(B^nP)_{i_1,\ldots,i_n}]$, it is concentrated in homological degrees $\geq n$. Now we need the following result from algebraic topology about the homology of Eilenberg--MacLane spaces.

\begin{theorem}[\cite{EilenbergMacLane}] For a finite free abelian group $P$, the homology groups
\[
H_i(\mathbb Z[B^nP]) = H_i(K(P,n),\mathbb Z)
\]
vanish for $i<n$ and are given by $M_{i-n}\otimes_{\mathbb Z} P$ for $n\leq i<2n$, where $M_j$, $j\geq 0$, is a certain sequence of finitely generated abelian groups, with $M_0=\mathbb Z$.\footnote{In fact, $M_j = \pi_j(H\mathbb Z\otimes_{\mathbb S} H\mathbb Z)$ form the integral dual Steenrod algebra, and are finite for $j>0$.}
\end{theorem}

In particular, $H_n(\mathbb Z[B^nP]) = P$, and we get a complex
\[
\mathbb Z[B^nP][-n]\to P\to 0
\]
of the form in Lemma~\ref{lem:approximateresolution}. The homology groups considered in Lemma~\ref{lem:approximateresolution} are all of the form $P\mapsto M_i\otimes_{\mathbb Z} P$ for some $i=0,\ldots,n-1$. As all $M_i$ are finitely generated and $P\mapsto P$ is $n-1$-pseudocoherent by induction, all $P\mapsto M_i\otimes_{\mathbb Z} P$ are $n-1$-pseudocoherent. Thus Lemma~\ref{lem:approximateresolution} shows that $P\mapsto P$ is $n$-pseudocoherent, as desired.
\end{proof}

We also used the following result in the lecture that is an easy consequence of the categorical setup established here.

\begin{proposition}\label{prop:homotopymultiply} Consider a functorial resolution $F(A)_\bullet$ of abelian groups $A$ where all $F(A)_i\cong \bigoplus_{j=1}^{n_i} \mathbb Z[A^{r_{i,j}}]$. For any integer $n$, the maps $n: F(A)\to F(A)$ of multiplication by $n$ and $[n]: F(A)\to F(A)$ induced by multiplication by $n$ on $A$ are homotopic via a functorial homotopy $h_\bullet: F(A)_\bullet\to F(A)_{\bullet+1}$.
\end{proposition}

\begin{proof} The data of such a functorial resolution is equivalent to the datum of a certain projective resolution of $P\mapsto P$ in $\Fun(\Latt,\Ab)$. Both maps $n$ and $[n]$ lift the same map $(P\mapsto P)\to (P\mapsto P)$ (multiplication by $n$) to the resolution, so general nonsense about projective resolutions in abelian categories gives the result (a priori for finite free $A$, but again this extends to general $A$ as everything is given in terms of universal formulas).
\end{proof}
\newpage

\section{Lecture V: Solid abelian groups}

We have seen that condensed abelian groups have excellent categorical properties, and that one can compute $\Hom$ and $\Ext$-groups, getting good answers. However, this is not the case when forming a tensor product: Even when $A$ and $B$ are nice topological (or condensed) abelian groups, their tensor product $A\otimes B$ is usually pathological. For example, take $A=\mathbb Z_p$ and $B=\mathbb R$ or $\mathbb Z_\ell$: The tensor products $\mathbb Z_p\otimes \mathbb R$ or $\mathbb Z_p\otimes \mathbb Z_\ell$ are pathological (note that the underlying abelian group is the usual algebraic tensor product, in both the topological and the condensed setting). The standard response to this is to complete the tensor product. This is usually where the theory gets tricky, as there are usually many sensible ways of forming a completed tensor product.

Somewhat surprisingly, there is a very general theory of completion for condensed abelian groups with extremely favourable properties.\footnote{The only caveat is that, for the notion of completion we are about to develop, the completion of $\mathbb R$ is $0$: The current notion of completion is very much a ``nonarchimedean" notion. In \cite{Analytic}, \cite{AnalyticStacks} we develop variants in more general situations, notably over $\mathbb R$.} However, as the current notion does not exactly match the classical notion of completeness, and will come in various flavours once we enter the relative theory, we have chosen to use the different term ``solid''.

\begin{definition}\label{def:solid}\leavevmode
\begin{enumerate}
\item[{\rm (i)}] For a profinite set $S=\varprojlim_i S_i$, define the condensed abelian group
\[
\mathbb Z[S]^\solid := \varprojlim_i \mathbb Z[S_i]\ .
\]
There is a natural map $S=\varprojlim_i S_i\to \mathbb Z[S]^\solid = \varprojlim_i \mathbb Z[S_i]$, inducing a map $\mathbb Z[S]\to \mathbb Z[S]^\solid$.
\item[{\rm (ii)}] A solid abelian group is a condensed abelian group $A$ such that for all profinite sets $S$ and all maps $f: S\to A$, there is a unique map $\tilde{f}: \mathbb Z[S]^\solid\to A$ extending $f$.
\item[{\rm (iii)}] A complex $C\in D(\Cond(\Ab))$ of condensed abelian groups is solid if for all profinite sets $S$ the natural map
\[
R\Hom(\mathbb Z[S]^\solid,C)\to R\Gamma(S,C)=R\Hom(\mathbb Z[S],C)
\]
is an isomorphism.
\end{enumerate}
\end{definition}

\begin{remark} It follows formally that if $A$ is condensed abelian group such that $A[0]\in D(\Cond(\Ab))$ is solid, then $A$ is solid. However, it is not clear that conversely, if $A$ is a solid abelian group, then $A[0]\in D(\Cond(\Ab))$ is solid, nor that if $C\in D(\Cond(\Ab))$ is solid then a cohomology group $H^i(C)$ is solid. However, Theorem~\ref{thm:solid} will show that these implications are both true.
\end{remark}

\begin{remark} One might ask in condition (iii) that the internal $R\intHom$'s agree, which is a priori stronger. By Corollary~\ref{cor:solidproperties}~(iv), this is actually automatic.
\end{remark}

We would like to call $\mathbb Z[S]^\solid$ the free solid abelian group on $S$. This terminology is in fact justified by the definition of solid abelian groups once we have proved that $\mathbb Z[S]^\solid$ is indeed a solid abelian group. Let us analyze its structure. We have
\[
\mathbb Z[S]^\solid = \varprojlim_i \mathbb Z[S_i] = \varprojlim_i \intHom(C(S_i,\mathbb Z),\mathbb Z)=\intHom(C(S,\mathbb Z),\mathbb Z)\ ,
\]
where the abelian group $C(S,\mathbb Z)$ of continuous maps $S\to \mathbb Z$ is a discrete abelian group. In particular, the underlying abelian group of $\mathbb Z[S]^\solid$ is the group
\[
\mathcal M(S,\mathbb Z) = \Hom(C(S,\mathbb Z),\mathbb Z)
\]
of $\mathbb Z$-valued measures on $S$. This means that if $A$ is a solid abelian group, $f: S\to A$ is a map from a profinite set, and $\mu\in \mathcal M(S,\mathbb Z)$ is a measure on $S$, then one can define
\[
\int f\mu\in A
\]
via evaluating the extension $\tilde{f}: \mathbb Z[S]^\solid\to A$ on $\mu$.

To go on, we need the following surprising structure result on $C(S,\mathbb Z)$, due to N\"obeling, \cite{Noebeling}, generalizing previous work of Specker, \cite{Specker}.

\begin{theorem}\label{thm:specker} For any profinite set $S$, the abelian group $C(S,\mathbb Z)$ of continuous maps from $S$ to $\mathbb Z$ is a free abelian group.
\end{theorem}

\begin{proof} For the convenience of the reader, we give a translation of the argument (attributed to Bergman) in \cite[Theorem 97.2]{FuchsInfinite} to our setup. Pick an injection $S\hookrightarrow \prod_I \{0,1\}$ for some set $I$, and choose a well-ordering on $I$, so $I$ is some ordinal $\lambda$, and elements of $I$ identify with ordinals $\mu<\lambda$. For each $\mu<\lambda$, we get the idempotent $e_\mu\in C(S,\mathbb Z)$ given by the corresponding projection $S\to \{0,1\}\subset \mathbb Z$. Order the products $e_{\mu_1}\cdots e_{\mu_r}$ with $\mu_1>\ldots>\mu_r$  (including the empty product corresponding to $r=0$) lexicographically, and let $E$ be the set of such products that cannot be written as a linear combination of smaller such products. We claim that $E$ is a basis of $C(S,\mathbb Z)$.

We argue by induction on $\lambda$, the case $\lambda=0$ being trivial. For any $\mu<\lambda$, let $S_\mu$ be the image of $S$ in $\prod_{\mu^\prime<\mu} \{0,1\}$. If $\lambda$ is a limit ordinal, then the result for $S$ follows formally from the result for all $S_\mu$ with $\mu<\lambda$ by passage to the filtered colimit, noting that the basis for $E$ is the union of the bases $E_\mu$ for all $S_\mu$. Thus assume that $\lambda=\rho+1$ is a successor ordinal and write $\overline{S} = S_\rho$. We have a closed immersion $S\hookrightarrow \overline{S}\times \{0,1\}$. Let $S_i = S\cap (\overline{S}\times \{i\})$ for $i=0,1$; these are closed subsets of $S$ that project to closed subsets of $\overline{S}$ that cover $\overline{S}$. Let $\overline{S}^\prime$ be the intersection of $S_0$ and $S_1$ in $\overline{S}$. Then we have a short exact sequence
\[
0\to C(\overline{S},\mathbb Z)\to C(S,\mathbb Z)\to C(\overline{S}^\prime,\mathbb Z)\to 0\ ,
\]
where the second map sends $f\in C(S,\mathbb Z)$ to the difference of its two restrictions to $\overline{S}^\prime\times\{0\}$ and $\overline{S}^\prime\times \{1\}$.

By induction, the part of the basis vectors of $E$ that do not start with $e_\rho$ form a basis of $C(\overline{S},\mathbb Z)$. On the other hand, the basis vectors of $E$ that start with $e_\rho$ project to a basis of $C(\overline{S}^\prime,\mathbb Z)$ (by applying the induction hypothesis to $\overline{S}^\prime$ with its closed immersion into $\prod_{\mu<\rho} \{0,1\}$). Thus, $E$ defines a basis of $C(S,\mathbb Z)$, as desired.
\end{proof}

\begin{corollary}\label{cor:speckerfree} For any profinite set $S$, there is some set $I$, $|I|\leq 2^{|S|}$, and an isomorphism of condensed abelian groups
\[
\mathbb Z[S]^\solid\cong \prod_I \mathbb Z\ .
\]
\end{corollary}

\begin{proof} Take an isomorphism $C(S,\mathbb Z)\cong \bigoplus_I \mathbb Z$, and use
\[
\mathbb Z[S]^\solid = \intHom(C(S,\mathbb Z),\mathbb Z)\cong \intHom(\bigoplus_I \mathbb Z,\mathbb Z)\cong \prod_I\mathbb Z\ .\qedhere
\]
\end{proof}

As another application, we note that in Definition~\ref{def:solid} we could replace profinite sets with extremally disconnected sets without changing the notions. This follows formally from the following proposition.

\begin{proposition}\label{prop:freecondensedsheaf} Let $S_\bullet\to S$ be a hypercover of a profinite set $S$ by profinite sets $S_j$. Then the corresponding complex
\[
\ldots\to \mathbb Z[S_1]^\solid\to \mathbb Z[S_0]^\solid\to \mathbb Z[S]^\solid\to 0
\]
of condensed abelian groups is exact.
\end{proposition}

\begin{proof} We know that
\[
0\to C(S,\mathbb Z)\to C(S_0,\mathbb Z)\to C(S_1,\mathbb Z)\to \ldots
\]
is exact (as $H^i(S,\mathbb Z)=0$ for $i>0$). Taking $R\intHom(-,\mathbb Z)$ and using Theorem~\ref{thm:specker} gives the result.
\end{proof}

Now we can prove that indeed, all $\mathbb Z[S]^\solid$ are solid; in fact $\mathbb Z[S]^\solid[0]\in D(\Cond(\Ab))$ are solid.

\begin{proposition}\label{prop:examplesolid} For any profinite set $S$, $\mathbb Z[S]^\solid$ is solid both as a module and as a complex.
\end{proposition}

\begin{proof} We need to prove that for all profinite sets $T$, we have
\[
R\Hom(\mathbb Z[T],\mathbb Z[S]^\solid) = R\Hom(\mathbb Z[T]^\solid,\mathbb Z[S]^\solid)\ .
\]
By Corollary~\ref{cor:speckerfree}, it suffices to handle the case $\mathbb Z[S]^\solid = \mathbb Z$. In that case $\Ext^i(\mathbb Z[T],\mathbb Z) = H^i(T,\mathbb Z)=0$ for $i>0$, and is given by $C(T,\mathbb Z)\cong \bigoplus_J \mathbb Z$ for $i=0$, applying Theorem~\ref{thm:specker} again. We have to see that the derived biduality map
\[
\bigoplus_J \mathbb Z\to R\Hom(\prod_J \mathbb Z,\mathbb Z)
\]
for $\bigoplus_J \mathbb Z$ is an isomorphism, noting that $\mathbb Z[T]^\solid\cong \prod_J \mathbb Z=R\intHom(\bigoplus_J \mathbb Z,\mathbb Z)$. We use the short exact sequence
\[
0\to \prod_J \mathbb Z\to \prod_J \mathbb R\to \prod_J \mathbb R/\mathbb Z\to 0\ .
\]
Here $\prod_J \mathbb R$ is a module over the condensed ring $\mathbb R$, and thus
\[
R\Hom(\prod_J \mathbb R,\mathbb Z) = R\Hom_{\mathbb R}(\prod_J \mathbb R,R\intHom(\mathbb R,\mathbb Z)) = 0
\]
by Theorem~\ref{thm:RHomoutofcompact}. On the other hand,
\[
R\Hom(\prod_J \mathbb R/\mathbb Z,\mathbb Z)=\bigoplus_J \mathbb Z[-1]
\]
by the same result. Together, we see that
\[
R\Hom(\prod_J \mathbb Z,\mathbb Z)=\bigoplus_J \mathbb Z\ ,
\]
as desired.
\end{proof}

The goal of the next lecture is to prove the following theorem.

\begin{theorem}\label{thm:solid}\leavevmode
\begin{enumerate}
\item[{\rm (i)}] The category $\Solid\subset \Cond(\Ab)$ of solid abelian groups is an abelian subcategory stable under all limits, colimits and extensions. The objects $\prod_I \mathbb Z\in \Solid$, where $I$ is any set, form a family of compact projective generators. The inclusion $\Solid\subset \Cond(\Ab)$ admits a left adjoint
\[
M\mapsto M^\solid: \Cond(\Ab)\to \Solid
\]
that is the unique colimit-preserving extension of $\mathbb Z[S]\mapsto \mathbb Z[S]^\solid$.
\item[{\rm (ii)}] The functor $D(\Solid)\to D(\Cond(\Ab))$ is fully faithful and its essential image are precisely the solid objects of $D(\Cond(\Ab))$. An object $C\in D(\Cond(\Ab))$ is solid if and only if all $H^i(C)\in \Cond(\Ab)$ are solid. The inclusion $D(\Solid)\to D(\Cond(\Ab))$ admits a left adjoint
\[
C\mapsto C^{L\solid}: D(\Cond(\Ab))\to D(\Solid)
\]
which is the left derived functor of $M\mapsto M^\solid$. 
\end{enumerate}
\end{theorem}

The following lemma summarizes the abstract setup, and isolates the key property that we need to prove in our situation.

\begin{lemma}\label{lem:everythingnice} Let $\mathcal A$ be an abelian category with all colimits that admits a subcategory $\mathcal A_0\subset \mathcal A$ of compact projective objects generating $\mathcal A$. Assume that $F: \mathcal A_0\to \mathcal A$ is a functor equipped with a natural transformation $X\to F(X)$, with the following property:\\

For any $X\in \mathcal A_0$ and any $Y,Z\in \mathcal A$ that can be written as (possibly infinite) direct sums of objects in the image of $F$, and any map $f: Y\to Z$ with kernel $K\in \mathcal A$, the map
\[
R\Hom(F(X),K)\to R\Hom(X,K)
\]
is an isomorphism.\\

Let $\mathcal A_F\subset \mathcal A$ be the full subcategory of all $Y\in \mathcal A$ such that for all $X\in \mathcal A_0$, the map
\[
\Hom(F(X),Y)\to \Hom(X,Y)
\]
is an isomorphism, and let $D_F(\mathcal A)\subset D(\mathcal A)$ be the full subcategory of all $C\in D(\mathcal A)$ such that for all $X\in \mathcal A_0$, the map
\[
R\Hom(F(X),C)\to R\Hom(X,C)
\]
is an isomorphism. Then the following properties hold true.

\begin{enumerate}
\item[{\rm (i)}] The category $\mathcal A_F\subset \mathcal A$ is an abelian subcategory stable under all limits, colimits and extensions, and the objects $F(X)$ for $X\in \mathcal A_0$ are compact projective generators. The inclusion $\mathcal A_F\subset \mathcal A$ admits a left adjoint $L: \mathcal A\to \mathcal A_F$ that is the unique colimit-preserving extension of $F: \mathcal A_0\to \mathcal A_F$.
\item[{\rm (ii)}] The functor $D(\mathcal A_F)\to D(\mathcal A)$ is fully faithful and identifies $D(\mathcal A_F)$ with $D_F(\mathcal A)$. A complex $C\in D(\mathcal A)$ lies in $D_F(\mathcal A)$ if and only if all $H^i(C)\in \mathcal A$ lie in $\mathcal A_F$. The inclusion $D(\mathcal A_F)=D_F(\mathcal A)\to D(\mathcal A)$ admits a left adjoint which is the left derived functor of $L$.
\end{enumerate}
\end{lemma}

\begin{proof} It is clear that $\mathcal A_F\subset \mathcal A$ is stable under passage to kernels and all limits. Let $f: Y\to Z$ be any map in $\mathcal A_F$; we want to see that the cokernel of $f$, taken in $\mathcal A$, still lies in $\mathcal A_F$. We may find a surjection $\bigoplus_{i\in I} P_i\to Z$ where all $P_i\in \mathcal A_0$, which extends to a surjection $\bigoplus_i F(P_i)\to Z$. Replacing $Y\to Z$ by the pullback, we may assume that $Z=\bigoplus_i F(P_i)$. Picking a similar surjection for $Y$, we may assume that both $Y$ and $Z$ are sums of objects in the image of $F$. We get an exact sequence
\[
0\to K\to Y\to Z\to Q\to 0
\]
in $\mathcal A$. Applying the hypothesis, we see that
\[
R\Hom(F(X),K)=R\Hom(X,K)
\]
and also $R\Hom(F(X),Y)=R\Hom(X,Y)$, $R\Hom(F(X),Z)=R\Hom(X,Z)$ (by applying the hypothesis also to $Y\to 0$ and $Z\to 0$). This implies that also $R\Hom(F(X),Q) = R\Hom(X,Q)$, so in particular $Q\in \mathcal A_F$. 

The preceding argument actually shows that $Q[0]\in D_F(\mathcal A)$. Moreover, one sees that the objects of $\mathcal A_F$ are precisely the cokernels of maps $f: Y\to Z$ between objects $Y,Z\in \mathcal A$ that are direct sums of objects in the image of $F$. This description shows that $\mathcal A_F$ is stable under all direct sums and under extensions. Together with stability under passage to cokernels, this implies stability under all colimits. It is also clear that the $F(X)\in \mathcal A_F$ for $X\in\mathcal A_0$ are compact projective generators, and that the left adjoint $L$ exists and agrees with $F$ on $\mathcal A_0$ (and is automatically colimit-preserving).

To see that $D(\mathcal A_F)\to D(\mathcal A)$ is fully faithful, it suffices to see that for any $X\in \mathcal A_0$ with associated compact projective generator $F(X)\in \mathcal A_F$, and any $C\in D(\mathcal A_F)$, the map
\[
R\Hom_{D(\mathcal A_F)}(F(X),C)\to R\Hom_{D(\mathcal A)}(F(X),C) = R\Hom_{D(\mathcal A)}(X,C)
\]
is an isomorphism. We may replace $C$ by $\tau^{\leq i} C$ as this does not change either side in any given degree if $i$ is large enough; and by passing to a Postnikov limit, we can also assume that $C$ is bounded, and then that $C$ is concentrated in one degree, so $C=Y[0]$ for some $Y\in \mathcal A_F$. It remains to see that
\[
\Ext^i_{\mathcal A_F}(F(X),Y)\to \Ext^i_{\mathcal A}(X,Y)
\]
is an isomorphism for all $i\geq 0$. But both sides vanish for $i>0$ as $X\in \mathcal A$ and $F(X)\in \mathcal A_F$ are projective; and for $i=0$ they agree as $Y\in \mathcal A_F$.

Consider the full subcategory $D_F^\prime(\mathcal A)\subset D(\mathcal A)$ of all $C\in D(\mathcal A)$ such that all $H^i(C)\in \mathcal A_F$. By the results on $\mathcal A_F$ already established, this is a triangulated subcategory stable under arbitrary direct sums and arbitrary products. Moreover, $D_F^\prime(\mathcal A)\subset D_F(\mathcal A)$. Indeed, if $C\in D_F^\prime(\mathcal A)$ is bounded, this reduces in finitely many steps to the case that $C = X[0]$ for $X\in \mathcal A_F$, where we have already established $X[0]\in D_F(\mathcal A)$. If $C$ is unbounded to the left, it follows via passage to a Postnikov limit. In the general case, any $\Ext^i(F(X),C)$ resp.~$\Ext^i(X,C)$ only depends on a truncation of $C$ on the right.

On the other hand, $D_F(\mathcal A)$ is generated by $F(X)[0]$ for $X\in \mathcal A_0$, by definition. As all $F(X)[0]\in D_F^\prime(\mathcal A)$, it follows that $D_F^\prime(\mathcal A)=D_F(\mathcal A)$.

The functor $D(\mathcal A_F)\to D(\mathcal A)$ clearly factors over $D(\mathcal A_F)\to D_F^\prime(\mathcal A)=D_F(\mathcal A)$ and induces an equivalence on hearts. As it is fully faithful and commutes with all products and direct sums, it is easy to see that it is an equivalence. It is also formal that the inclusion $D_F(\mathcal A)\subset D(\mathcal A)$ admits a left adjoint, which necessarily commutes with arbitrary direct sums and by definition of $D_F(\mathcal A)$ takes $X\in \mathcal A_0$ to $F(X)\in \mathcal A_0$, so is the left derived functor of $L: \mathcal A\to \mathcal A_F$.
\end{proof}

In some situations, the property isolated in the lemma can be proved directly. In our case, we will use the following criterion.

\begin{lemma}\label{lem:showeverythingnice} Let $\mathcal A$ be an abelian category with all colimits that admits a subcategory $\mathcal A_0\subset \mathcal A$ of compact projective objects generating $\mathcal A$. Assume that $F: \mathcal A_0\to \mathcal A$ is a functor equipped with a natural transformation $X\to F(X)$, satisfying the following property:\\

For all $X\in\mathcal A_0$ and any complex
\[
C: \ldots \to C_i\to \ldots \to C_1\to C_0\to 0\ ,
\]
where all $C_i$ are direct sums of objects in the image of $F$, the map
\[
R\Hom(F(X),C)\to R\Hom(X,C)
\]
is an isomorphism.\\

Then for any $Y$ and $Z$ that are direct sums of objects in the image of $F$ and any map $f: Y\to Z$ with kernel $K$, the map
\[
R\Hom(F(X),K)\to R\Hom(X,K)
\]
is an isomorphism, and so Lemma~\ref{lem:everythingnice} applies.
\end{lemma}

\begin{remark} Both in Lemma~\ref{lem:everythingnice} and in Lemma~\ref{lem:showeverythingnice}, the hypothesis imposed is clearly a consequence of the conclusion of Lemma~\ref{lem:everythingnice}, and so both are equivalent to it. We note that in Lemma~\ref{lem:everythingnice}, one could also impose the similar condition on the cokernel of $f$, which would again be equivalent.
\end{remark}

\begin{proof} Take any $f: Y\to Z$ as in the statement, and choose a resolution
\[
\ldots \to B_1\to B_0\to K\to 0
\]
of $K$ in $\mathcal A$, where all $B_i = \bigoplus_{j\in J_i} X_{i,j}$ are direct sums of elements $X_{i,j}\in \mathcal A_0$. Let $B$ be the complex $\ldots\to B_1\to B_0\to 0$. Let $C_i = \bigoplus_{j\in J_i} F(X_{i,j})$; these form a complex $C$, and there is a natural map $B\to C$. Note that as
\[
R\Hom(X_{i,j},Y) = R\Hom(F(X_{i,j}),Y)
\]
for all $i,j$ (by applying the hypothesis to the complex $Y[0]$), one has $R\Hom(B,Y) = R\Hom(C,Y)$, and similarly $R\Hom(B,Z) = R\Hom(C,Z)$. It follows that the map $B\to K\subset Y$ extends uniquely to a map $C\to Y$ whose composite to $Z$ vanishes, i.e.~the map $B\to K$ extends to a map $C\to K$. But the map $B\to K$ is a quasi-isomorphism, so $B$ is a retract of $C$ in $D(\mathcal A)$.

In particular, the condition
\[
R\Hom(F(X),C)\cong R\Hom(X,C)
\]
passes to $B$, which is a resolution of $K$, i.e.
\[
R\Hom(F(X),K)\cong R\Hom(X,K)\ ,
\]
as desired.
\end{proof}
\newpage

\section{Lecture VI: Solid abelian groups, II}

The goal of this lecture is to finish the proof of Theorem~\ref{thm:solid}.

\begin{proof}[Proof of Theorem~\ref{thm:solid}] We want to apply Lemma~\ref{lem:everythingnice} to $\mathcal A=\Cond(\Ab)$, the compact projective generators $\mathcal A_0=\mathbb Z[S]$ for $S$ extremally disconnected, and the functor $F: \mathcal A_0\to \mathcal A$ taking $\mathbb Z[S]$ to $\mathbb Z[S]^\solid$. Note that
\[
\mathbb Z[S]^\solid = \intHom(C(S,\mathbb Z),\mathbb Z) = \intHom(\intHom(\mathbb Z[S],\mathbb Z),\mathbb Z)
\]
is indeed functorial in $\mathbb Z[S]$. To verify the hypothesis of Lemma~\ref{lem:everythingnice}, we will use Lemma~\ref{lem:showeverythingnice}. In other words, we need to see that for any complex
\[
C: \ldots\to C_i\to \ldots \to C_1\to C_0\to 0
\]
such that each $C_i$ is a direct sum of condensed abelian groups of the form $\mathbb Z[T]^\solid$ for varying extremally disconnected sets $T$, then
\[
R\Hom(\mathbb Z[S]^\solid,C) = R\Hom(\mathbb Z[S],C)\ ,
\]
for all profinite sets $S$, or, rewriting the right-hand side,
\[
R\Hom(\mathbb Z[S]^\solid,C) = R\Gamma(S,C)\ .
\]

Note that $\mathbb Z[T]^\solid\cong \prod_I \mathbb Z$ for some set $I$; we will then in fact more generally allow complexes $C$ such that each $C_i$ is a direct sum of condensed abelian groups of the form $\prod_I \mathbb Z$ for varying sets $I$.

To prepare for the proof, we recall that for any profinite set $S$ the condensed abelian group
\[
\mathbb Z[S]^\solid = \intHom(C(S,\mathbb Z),\mathbb Z)=\mathcal M(S,\mathbb Z)
\]
is the space of integer-valued measures on $S$. We will also consider
\[
\mathcal M(S,\mathbb R) = \intHom(C(S,\mathbb Z),\mathbb R)
\]
and
\[
\mathcal M(S,\mathbb R/\mathbb Z)=\intHom(C(S,\mathbb Z),\mathbb R/\mathbb Z)\ ,
\]
the spaces of $\mathbb R$ and $\mathbb R/\mathbb Z$-valued measures. Note that by Theorem~\ref{thm:specker}, one has $\mathcal M(S,\mathbb R)\cong \prod_I \mathbb R$ and $\mathcal M(S,\mathbb R/\mathbb Z)\cong \prod_I \mathbb R/\mathbb Z$, and the sequence
\[
0\to \mathcal M(S,\mathbb Z)\to \mathcal M(S,\mathbb R)\to \mathcal M(S,\mathbb R/\mathbb Z)\to 0
\]
is exact.

We will prove the following claim.\\

{\bf Claim.} Let
\[
C: \ldots\to C_i\to \ldots \to C_1\to C_0\to 0
\]
be any complex such that each $C_i$ is a direct sum of condensed abelian groups of the form $\prod_I \mathbb Z$ for varying sets $I$. Then
\[
R\intHom(\mathcal M(S,\mathbb R/\mathbb Z),C)(S^\prime) = R\Gamma(S\times S^\prime,C)[-1]
\]
for all profinite sets $S$, $S^\prime$.\\

Let us first explain how this implies the desired result
\[
R\Hom(\mathbb Z[S]^\solid,C) = R\Gamma(S,C)\ ;
\]
in fact, the slightly finer result that for all profinite sets $S^\prime$, one has
\begin{equation}\label{eq:inthomright}
R\intHom(\mathbb Z[S]^\solid,C)(S^\prime) = R\Gamma(S\times S^\prime,C)\ .
\end{equation}

First, applying the claim for $S=\{\ast\}$ being a point, we find that
\[
R\intHom(\mathbb R/\mathbb Z,C)(S^\prime) = R\Gamma(S^\prime,C)[-1] = R\intHom(\mathbb Z[1],C)(S^\prime)
\]
for all $S^\prime$, so that
\begin{equation}\label{eq:completionR}
R\intHom(\mathbb R,C) = 0\ .
\end{equation}
In particular, as $\mathcal M(S,\mathbb R)$ is a module over the condensed ring $\mathbb R$, this implies
\[
R\intHom(\mathcal M(S,\mathbb R),C) = R\intHom_{\mathbb R}(\mathcal M(S,\mathbb R),R\intHom(\mathbb R,C)) = 0\ .
\]
In other words,
\[
R\intHom(\mathcal M(S,\mathbb Z),C) = R\intHom(\mathcal M(S,\mathbb R/\mathbb Z),C)[1]\ ,
\]
so the claim gives the desired
\[
R\intHom(\mathbb Z[S]^\solid,C)(S^\prime) = R\Gamma(S\times S^\prime,C)\ .
\]

We prove the claim first when $C$ is concentrated in degree $0$, so $C=M[0]$ where $M = \bigoplus_{i\in I} \prod_{j\in J_i} \mathbb Z$ for certain sets $I$, $J_i$, $i\in I$. For this, we use Theorem~\ref{thm:deligneresolution} (and the pseudocoherence of $\mathbb Z[X]$ for $X$ compact Hausdorff, via simplicially resolving $X$ by extremally disconnected sets) to see that the condensed abelian group $\mathcal M(S,\mathbb R/\mathbb Z)\cong \prod \mathbb R/\mathbb Z$ is pseudocoherent, as then is $\mathcal M(S,\mathbb R/\mathbb Z)\otimes \mathbb Z[S^\prime]$. This implies that
\[\begin{aligned}
R\intHom(\mathcal M(S,\mathbb R/\mathbb Z),\bigoplus_{i\in I} \prod_{j\in J_i}\mathbb Z)(S^\prime) &= \bigoplus_{i\in I} R\intHom(\mathcal M(S,\mathbb R/\mathbb Z),\prod_{j\in J_i} \mathbb Z)(S^\prime)\\
&= \bigoplus_{i\in I} \prod_{j\in J_i} R\intHom(\mathcal M(S,\mathbb R/\mathbb Z), \mathbb Z)(S^\prime)\\
&= \bigoplus_{i\in I} \prod_{j\in J_i} C(S\times S^\prime,\mathbb Z)[-1]\ ,
\end{aligned}\]
using Theorem~\ref{thm:RHomoutofcompact} in the last equation. But we also have for any profinite set $S$
\[
R\Gamma(S,\bigoplus_{i\in I} \prod_{j\in J_i} \mathbb Z) = \bigoplus_{i\in I} R\Gamma(S,\prod_{j\in J_i} \mathbb Z) = \bigoplus_{i\in I} \prod_{j\in J_i} C(S,\mathbb Z)[0]\ ,
\]
so we get the desired
\[
R\intHom(\mathcal M(S,\mathbb R/\mathbb Z),M)(S^\prime)\cong R\Gamma(S\times S^\prime,M)[-1]\ .
\]

The case where $C$ is bounded follows formally. In general, we see that it suffices to prove that for any such complex $C$, 
\[
R\intHom(\mathcal M(S,\mathbb R/\mathbb Z),C)\ ,\ R\Gamma(S,C)
\]
are concentrated in cohomological degrees $\leq 1$ for all profinite sets $S$. Indeed, if $C_{\leq i}$ is the stupid truncation of $C$, we know the result for $C_{\leq i}$, and for the cone $C_{\geq i+1}$ of $C_{\leq i}\to C$ both sides will be concentrated in homological degrees $\geq i$; so in any given degree, we get the desired result by taking $i$ large enough.\footnote{We may assume that $S^\prime$ is extremally disconnected in proving the claim, in which case the bound on $R\intHom(\mathcal M(S,\mathbb R/\mathbb Z),C)$ implies the same bound on $S^\prime$-valued points.}

Let $C_i = \bigoplus_{j\in J_i} \prod_{k\in K_{i,j}} \mathbb Z$ for certain sets $J_i$, $K_{i,j}$. Define a new complex $C_{\mathbb R}$ whose degree $i$ term is
\[
C_{\mathbb R,i} = \bigoplus_{j\in J_i} \prod_{k\in K_{i,j}} \mathbb R\ .
\]
We claim that the differential $d_i: C_{i+1}\to C_i\subset C_{i,\mathbb R}$ extends uniquely to a map $d_{\mathbb R,i}: C_{\mathbb R,i+1}\to C_{\mathbb R,i}$, i.e.
\[
\Hom(C_{i+1},C_{\mathbb R,i}) = \Hom(C_{\mathbb R,i+1},C_{\mathbb R,i})\ .
\]
To prove this claim, we may assume $C_{i+1} = \prod_{K_{i+1}} \mathbb Z$ for some set $K_{i+1}$ (as the case of infinite direct sums follows formally), and then it is enough to see that
\[
R\Hom(\prod_{K_{i+1}} \mathbb R/\mathbb Z,C_{\mathbb R,i}) = 0\ .
\]
As $\prod_{K_{i+1}} \mathbb R/\mathbb Z$ is pseudocoherent, we can now assume that $C_{\mathbb R,i} = \prod_{K_i}\mathbb R$, in which case it follows from Theorem~\ref{thm:RHomoutofcompact}~(ii).

We also define the complex $C_{\mathbb R/\mathbb Z}$ with terms
\[
C_{\mathbb R/\mathbb Z,i} = C_{\mathbb R,i}/C_i = \bigoplus_{j\in J_i} \prod_{k\in K_{i,j}} \mathbb R/\mathbb Z\ ,
\]
so there is a short exact sequence
\[
0\to C\to C_{\mathbb R}\to C_{\mathbb R/\mathbb Z}\to 0
\]
of complexes. Thus, it is enough to prove that
\[
R\intHom(\mathcal M(S,\mathbb R/\mathbb Z),C_{\mathbb R})\ ,\ R\Gamma(S,C_{\mathbb R})\ ,\ R\intHom(\mathcal M(S,\mathbb R/\mathbb Z),C_{\mathbb R/\mathbb Z})\ ,\ R\Gamma(S,C_{\mathbb R/\mathbb Z})
\]
are concentrated in cohomological degrees $\leq 0$.

By writing $C_{\mathbb R}$ and $C_{\mathbb R/\mathbb Z}$ as the limits of their canonical truncations $\tau_{\leq i} C_{\mathbb R}$, $\tau_{\leq i} C_{\mathbb R/\mathbb Z}$, it suffices to prove the same result for these canonical truncations. These are finite complexes, and so it suffices to prove the result for all of their terms. The hardest part is the kernel of the $i$-th differential, where it suffices to see that
\[
R\intHom(\mathcal M(S,\mathbb R/\mathbb Z),\ker d_{\mathbb R,i})\ ,\ R\Gamma(S,\ker d_{\mathbb R,i})\ ,\ R\intHom(\mathcal M(S,\mathbb R/\mathbb Z),\ker d_{\mathbb R/\mathbb Z,i})\ ,\ R\Gamma(S,\ker d_{\mathbb R/\mathbb Z,i})
\]
are all concentrated in cohomological degrees $\leq 1$. By passing to filtered colimits over finite subsets of $J_{i+1}$ and using pseudocoherence, we can for this assume that $J_{i+1}$ is finite, in which case $C_{i+1} = \prod_{K_{i+1}} \mathbb Z$ is a product of copies of $\mathbb Z$. Then the map
\[
d_i: C_{i+1}=\prod_{K_{i+1}} \mathbb Z\to C_i=\bigoplus_{j\in J_i} \prod_{k\in K_{i,j}} \mathbb Z
\]
factors over a finite direct sum, so we can also assume that $J_i$ is finite, so that $C_i=\prod_{K_i} \mathbb Z$ is a product of copies of $\mathbb Z$. Then
\[
d_{\mathbb R/\mathbb Z,i}: \prod_{K_{i+1}} \mathbb R/\mathbb Z\to \prod_{K_i} \mathbb R/\mathbb Z
\]
is a map of compact abelian groups, and so $A=\ker d_{\mathbb R/\mathbb Z,i}$ is a compact abelian group. Using Theorem~\ref{thm:RHomoutofcompact}, this implies that $R\intHom(\mathcal M(S,\mathbb R/\mathbb Z),A)$ is concentrated in degrees $\leq 1$, and also $R\Gamma(S,A)$ is concentrated in degrees $\leq 1$, by using a two-step resolution of $A$ in terms of products of $\mathbb R/\mathbb Z$'s.

In order to understand $\ker d_{\mathbb R,i}$, we use that $d_i: \prod_{K_{i+1}} \mathbb Z\to \prod_{K_i} \mathbb Z$ is dual to a map $f: \bigoplus_{K_i} \mathbb Z\to \bigoplus_{K_{i+1}} \mathbb Z$. The corresponding map $f_{\mathbb R}: \bigoplus_{K_i} \mathbb R\to \bigoplus_{K_{i+1}} \mathbb R$ is a map of $\mathbb R$-vector spaces, and as such one can factor it as a split surjection followed by a split injection, and the cokernel $Q$ of $f_{\mathbb R}$ is also of the form $\bigoplus_L \mathbb R$ for some set $L$. Passing to the duals again, one sees that $\ker d_{\mathbb R,i} = \Hom_{\mathbb R}(Q,\mathbb R)\cong \prod_L \mathbb R$. This implies that
\[
R\intHom(\mathcal M(S,\mathbb R/\mathbb Z),\ker d_{\mathbb R,i}) = 0
\]
by Theorem~\ref{thm:RHomoutofcompact}~(ii), while $R\Gamma(S,\ker d_{\mathbb R,i})$ is concentrated in degree $0$ by Theorem~\ref{thm:enoughcontfcts}, finally finishing the proof.
\end{proof}

Let us collect some of the information obtained.

\begin{corollary}\label{cor:solidproperties}\leavevmode
\begin{enumerate}
\item[{\rm (i)}] The compact projective objects of $\Solid$ are exactly the condensed abelian groups $\prod_I \mathbb Z$ for varying sets $I$.
\item[{\rm (ii)}] The derived category $D(\Solid)$ is compactly generated, and the full subcategory
\[
D(\Solid)^\omega
\]
of compact objects consists of the bounded complexes all of whose terms are of the form $\prod_I \mathbb Z$. The category $D(\Solid)^\omega$ is contravariantly equivalent to the category $D^b(\mathbb Z)$ via the functor
\[
D^b(\mathbb Z)^{\mathrm{op}}\to D(\Solid)\ :\ C\mapsto R\intHom(C,\mathbb Z)\ .
\]
\item[{\rm (iii)}] The derived solidification $\mathbb R^{L\solid}$ of $\mathbb R$ is equal to $0$.
\item[{\rm (iv)}] Let $C\in D(\Solid)$. Then for all profinite sets $S$, the map
\[
R\intHom(\mathbb Z[S]^\solid,C)\to R\intHom(\mathbb Z[S],C)
\]
is an isomorphism.
\end{enumerate}
\end{corollary}

\begin{proof} For (i), note that by Lemma~\ref{lem:everythingnice} compact projective generators of $\Solid$ are given by $\mathbb Z[S]^\solid$ for extremally disconnected $S$. By Corollary~\ref{cor:speckerfree}, these are isomorphic to products of copies of $\mathbb Z$ (and the cardinality of the index set can be arbitrarily large). It follows that any product of copies of $\mathbb Z$ is compact projective, by writing it as a retract of a larger product of the form $\mathbb Z[S]^\solid$. If $P\in \Solid$ is any compact projective object, one can write it as retract of $\prod_I \mathbb Z$ for some set $I$; but any such retract is itself a product of copies of $\mathbb Z$ (as dually any retract of a free abelian group is a free abelian group).

For (ii), the first sentence follows formally from (i). The functor $D^b(\mathbb Z)^{\mathrm{op}}\to D(\Solid)$ lands in the compact objects as one can represent any bounded complex of abelian groups by a bounded complex of free abelian groups. Similarly, one has a functor $D(\Solid)^\omega\to D^b(\mathbb Z)^{\mathrm{op}}$ defined by $C\mapsto R\Hom(C,\mathbb Z)$, and one immediately verifies that these functors are quasi-inverse.

For (iii), we need to see that $R\Hom(\mathbb R,C)=0$ whenever $C$ is solid. We may assume that $C$ is bounded to the right, and then by shifting that $C$ sits in homological degrees $\geq 0$. In that case, picking a projective resolution, we may represent $C$ by a complex all of whose terms are direct sums of products of copies of $\mathbb Z$. But then Equation~\eqref{eq:completionR} gives the claim.

For (iv), we may similarly assume that $C$ is a complex all of whose terms are direct sums of products of copies of $\mathbb Z$. We need to see that for all profinite sets $S^\prime$, we have
\[
R\intHom(\mathbb Z[S]^\solid,C)(S^\prime) = R\Gamma(S\times S^\prime,C)\ ,
\]
which is Equation \eqref{eq:inthomright} from the previous proof.
\end{proof}

The corollary implies that the completed tensor product behaves well. More precisely, we have the following theorem.

\begin{theorem}\label{thm:solidtensor}\leavevmode
\begin{enumerate}
\item[{\rm (i)}] There is a unique way to endow $\Solid$ with a symmetric monoidal tensor product $\otimes^\solid$ such that the solidification functor
\[
M\mapsto M^\solid: \Cond(\Ab)\to \Solid
\]
is symmetric monoidal.
\item[{\rm (ii)}] There is a unique way to endow $D(\Solid)$ with a symmetric monoidal tensor product $\otimes^{L\solid}$ such that the solidification functor
\[
C\mapsto C^{L\solid}: D(\Cond(\Ab))\to D(\Solid)
\]
is symmetric monoidal. The functor $\otimes^{L\solid}$ is the left derived functor of $\otimes^\solid$.
\end{enumerate}
\end{theorem}

Concretely, one can in both cases define $M\otimes^\solid M^\prime$ resp.~$C\otimes^{L\solid} C^\prime$ as $(M\otimes M^\prime)^\solid$ resp.~$(C\otimes^L C^\prime)^{L\solid}$. In fact, this is a formal consequence of the assertion that completion is symmetric monoidal, and proves the uniqueness assertion.

\begin{proof} As observed, uniqueness is clear. For existence in (i), we define a symmetric monoidal tensor product on $\Solid$ via $M\otimes^\solid N = (M\otimes N)^\solid$. We need to see that the functor $M\mapsto M^\solid$ is symmetric monoidal, i.e.~the map
\[
(M\otimes N)^\solid\to (M^\solid\otimes N^\solid)^\solid
\]
is an isomorphism. This can be written as the composite
\[
(M\otimes N)^\solid\to (M^\solid\otimes N)^\solid\to (M^\solid\otimes N^\solid)^\solid
\]
and it is enough to prove that the first map is an isomorphism (as then applying the same for $M$ replaced by $N$ and $N$ replaced by $M^\solid$ shows that the second map is an isomorphism). As all functors in question commute with all colimits, we can assume that $M=\mathbb Z[S]$, $N=\mathbb Z[T]$ are free condensed abelian groups on profinite sets $S$, $T$. We need to see that
\[
\mathbb Z[S\times T]^\solid\to (\mathbb Z[S]^\solid\otimes \mathbb Z[T])^\solid
\]
is an isomorphism. This amounts to showing that for all solid $A$, the map
\[
\Hom(\mathbb Z[S]^\solid\otimes \mathbb Z[T],A) = \intHom(\mathbb Z[S]^\solid,A)(T)\to \intHom(\mathbb Z[S],A)(T) = A(S\times T)
\]
is an isomorphism, which follows from Corollary~\ref{cor:solidproperties}~(iv).

The existence proof in (ii) proceeds in exactly the same way. To see that $\otimes^{L\solid}$ is the left derived functor of $\otimes^\solid$, it remains to see that for compact projective $M,N\in \Solid$, one has
\[
M\otimes^{L\solid} N = M\otimes^\solid N\ ,
\]
i.e.~$M\otimes^{L\solid} N$ is concentrated in degree $0$. But we can assume $M=\mathbb Z[S]^\solid$, $N=\mathbb Z[T]^\solid$, with $S$ and $T$ extremally disconnected, in which case $M\otimes^{L\solid} N = \mathbb Z[S\times T]^\solid$ (as all functors are symmetric monoidal), which is still concentrated in degree $0$.
\end{proof}

The completed tensor product behaves nicely:

\begin{proposition}\label{prop:tensorinfproducts} Let $M=\prod_I \mathbb Z$ and $N=\prod_J \mathbb Z$ be products of copies of $\mathbb Z$. Then
\[
M\otimes^{L\solid} N = \prod_{I\times J} \mathbb Z\ .
\]
\end{proposition}

\begin{proof} We need to see that for any free abelian groups $M^\prime$, $N^\prime$, the natural map
\[
\intHom(M^\prime,\mathbb Z)\otimes^{L\solid}\intHom(N^\prime,\mathbb Z)\to \intHom(M^\prime\otimes N^\prime,\mathbb Z)
\]
is an isomorphism. To prove this, write $M^\prime$ as a retract of $C(S,\mathbb Z)$ for some profinite set $S$, and similarly $N^\prime$ as a retract of $C(T,\mathbb Z)$ for some profinite set $T$; we can then reduce to $M^\prime=C(S,\mathbb Z)$ and $N^\prime=C(T,\mathbb Z)$. Then $M^\prime\otimes N^\prime = C(S\times T,\mathbb Z)$, and the claim is equivalent to
\[
\mathbb Z[S]^\solid\otimes^{L\solid} \mathbb Z[T]^\solid\to \mathbb Z[S\times T]^\solid
\]
being an isomorphism, which is a consequence of all functors being symmetric monoidal.
\end{proof}

\begin{example} Let $p$ and $\ell$ be distinct primes. One has the following identities:
\[\begin{aligned}
\mathbb Z_p\otimes^{L\solid} \mathbb R&=0\ ,\\
\mathbb Z_p\otimes^{L\solid} \mathbb Z_\ell&=0\ ,\\
\mathbb Z_p\otimes^{L\solid} \mathbb Z_p&=\mathbb Z_p\ ,\\
\mathbb Z_p\otimes^{L\solid} \mathbb Z[[T]]&=\mathbb Z_p[[T]]\ ,\\
\mathbb Z[[U]]\otimes^{L\solid} \mathbb Z[[T]]&=\mathbb Z[[U,T]]\ .
\end{aligned}\]
We see that in these situations the completed tensor product seems to ask both factors for what $I$-adic topology they have, and then combines the adic topologies of both factors.

To verify the examples, the first one follows from $\otimes^{L\solid}$ being symmetric monoidal and $\mathbb R^{L\solid}=0$. Now we continue with the last one, which follows from Proposition~\ref{prop:tensorinfproducts}. This implies the second-to-last by taking the quotient by multiplication with $U-p$. This in turn implies the second and third examples by passing to the quotient by multiplication with $T-\ell$ resp.~$T-p$.
\end{example}

\begin{example} Let $X$ be a CW complex, and let $H_\bullet(X)$ denote the singular homology complex. Then there is a canonical isomorphism
\[
\mathbb Z[X]^{L\solid}\cong H_\bullet(X)\ .
\]
In particular, this shows that the derived solidification of a condensed abelian group can sit in all nonnegative homological degrees.

To prove this, it suffices to handle the case that $X$ is a finite CW complex, as both sides commute with filtered colimits. Then $X$ is compact Hausdorff, and thus $\mathbb Z[X]$ is pseudocoherent. In particular, $\mathbb Z[X]^{L\solid}$ is a pseudocoherent complex of solid abelian groups. On such, the functor $R\intHom(-,\mathbb Z)$ takes discrete values in $D^{\geq 0}(\mathbb Z)$, and the resulting functor is fully faithful. Note that as $X$ is a finite CW complex, $H_\bullet(X)$ is also pseudocoherent. In other words, it is enough to show that there is a canonical isomorphism
\[
R\Hom(\mathbb Z[X]^{L\solid},\mathbb Z)\cong R\Hom(H_\bullet(X),\mathbb Z)\in D(\mathbb Z)\ .
\]
But the left-hand side is given by $R\Hom(\mathbb Z[X],\mathbb Z) = R\Gamma(X,\mathbb Z)$, as is the right-hand side, as desired.
\end{example}
\newpage

\section{Lecture VII: Analytic rings}

As our next topic, we want to extend the results obtained for abelian groups to $A$-modules for any discrete ring $A$. Although our main interest is certainly in commutative rings, we allow general unital associative rings $A$ in the following.

From now on, we will repeatedly run into a setup of a condensed ring $A$ together with a specification of ``free complete'' $A$-modules $A[S]^\wedge$, such as $\mathbb Z$ and $\mathbb Z[S]^\solid$. It will be convenient to have a general language for this situation, which we will establish in this lecture. The theory developed in this lecture is a first approximation to the theory of analytic rings later developed in \cite{Analytic}, \cite{AnalyticStacks} --- in \cite{Analytic}, it is generalized to the animated setting where general constructions have a better behaviour, while in \cite{AnalyticStacks} its analogue in the setting of light condensed mathematics is studied.

In the following, for any condensed ring $A$, we denote by $A\Mod$ the category of $A$-modules in condensed abelian groups. If $A$ is discrete, this is ambiguous; however, we will for the moment not need to consider the category of abstract $A$-modules.

\begin{definition}\label{def:preanalytic} A theory of measures on a condensed ring $A$ is given by a functor
\[
\mathcal M: \{\mathrm{extremally\ disconnected\ sets}\}\to A\Mod: S\mapsto \mathcal M[S]
\]
to the category $A\Mod$ of $A$-modules in condensed abelian groups, taking finite disjoint unions to products, and a natural transformation $S\to \mathcal M[S]$ of ``Dirac measures''.
\end{definition}

\begin{remark} If the underlying condensed ring $A$ is discrete, the functor $S\mapsto \mathcal M[S]$ factors uniquely over an additive functor $A[S]\mapsto \mathcal M[S]$ from the full subcategory of $A$-modules of the form $A[S]$ for $S$ extremally disconnected. To see this, we need to see that for any extremally disconnected sets $S$, $T$, we can get functoriality in $\Hom(A[S],A[T]) = A[T](S)$, which is the sheafification of $S\mapsto A[T(S)]$. We certainly have functoriality in $A[T(S)]$ as the functor takes values in $A$-modules, and we can pass to the sheafification by the condition that the functor takes finite disjoint unions to products.

The functor $A[S]\mapsto \mathcal M[S]$ is then moreover equipped with a natural transformation $A[S]\to \mathcal M[S]$. We will see that if $(A,\mathcal M)$ is analytic, then these properties hold also when $A$ is general.
\end{remark}

Let us give some examples. We include here already a number of examples where the condensed ring $A$ is non-discrete. However, in these lectures we will primarily consider the case where $A$ is discrete.

\begin{examples}\leavevmode
\begin{enumerate}
\item[{\rm (i)}] The theory $\mathbb Z_\solid$-measures is given by the discrete condensed ring $\mathbb Z$ together with the functor
\[
S\mapsto \mathbb Z_\solid[S]  := \mathbb Z[S]^\solid = \mathcal M(S,\mathbb Z)\ .
\]
\item[{\rm (ii)}] Let $p$ be a prime. The theory of $\mathbb Z_{p,\solid}$-measures is given by the condensed ring $\mathbb Z_p$ of $p$-adic integers, together with the functor
\[
S\mapsto \mathbb Z_{p,\solid}[S] := \mathbb Z_p[S]^\solid = \varprojlim_i \mathbb Z_p[S_i]\ .
\]
This is the space $\mathcal M(S,\mathbb Z_p)$ of $\mathbb Z_p$-valued measures on $S$.
\item[{\rm (iii)}] For a discrete ring $A$, the theory of $(A,\mathbb Z)_\solid$-measures is given by the discrete condensed ring $A$ together with the functor
\[
S\mapsto \mathbb Z_\solid[S]\otimes_{\mathbb Z} A\ .
\]
\item[{\rm (iv)}] For any finitely generated $\mathbb Z$-algebra $A$, the theory of $A_\solid$-measures is given by the discrete condensed ring $A$ together with the functor taking an extremally disconnected $S=\varprojlim_i S_i$ to the condensed $A$-module
\[
A_\solid[S]:= \varprojlim_i A[S_i]\ .
\]
The underlying abelian group is the space $\mathcal M(S,A)$ of $A$-valued measures on $S$.\footnote{Obviously, this definition extends immediately to the case of a general discrete ring $A$. We restrict it to the case of finitely generated $\mathbb Z$-algebras as only then $A_\solid$ fits into a general framework (e.g., taking a general Huber pair $(A,A^+)$ to a pre-analytic ring $(A,A^+)_\solid$). In fact, for general discrete $A$, the theory of measures $A_\solid$ does not define an analytic ring.}
\end{enumerate}
\end{examples}

For any theory of measures $(A,\mathcal M)$, one can mimic Definition~\ref{def:solid}. We would like this to behave well. This amounts to the following definition.

\begin{definition}\label{def:analytic} An analytic ring is a condensed ring $A$ equipped with a theory of measures $\mathcal M$ such that for any complex
\[
C: \ldots \to C_i\to \ldots\to C_1\to C_0\to 0
\]
of $A$-modules in condensed abelian groups, such that all $C_i$ are direct sums of objects of the form $\mathcal M[T]$ for varying extremally disconnected $T$, the map
\[
R\intHom_A(\mathcal M[S],C)\to R\intHom_A(A[S],C)
\]
in the derived category of condensed abelian groups is an isomorphism for all extremally disconnected sets $S$.
\end{definition}

Note that for any (associative) condensed ring $A$ and any two $A$-modules $M$, $N$ in condensed abelian groups, one can define a complex of condensed abelian groups $R\intHom_A(M,N)$ whose $S$-valued points are $R\Hom_A(M\otimes_{\mathbb Z} \mathbb Z[S],N)$.

Then Lemma~\ref{lem:showeverythingnice} implies the following result.

\begin{proposition}\label{prop:analyticnice} Let $(A,\mathcal M)$ be an analytic ring.
\begin{enumerate}
\item[{\rm (i)}] The full subcategory
\[
(A,\mathcal M)\Mod\subset A\Mod
\]
of all $A$-modules $M$ in condensed abelian groups such that for all extremally disconnected sets $S$, the map
\[
\Hom_A(\mathcal M[S],M)\to M(S)
\]
is an isomorphism, is an abelian category stable under all limits, colimits, and extensions. The objects $\mathcal M[S]$ for $S$ extremally disconnected form a family of compact projective generators. It admits a left adjoint
\[
A\Mod\to (A,\mathcal M)\Mod: M\mapsto M\otimes_A (A,\mathcal M)
\]
that is the unique colimit-preserving extension of $A[S]\mapsto \mathcal M[S]$. If $A$ is commutative, there is a unique symmetric monoidal tensor product $\otimes_{(A,\mathcal M)}$ on $(A,\mathcal M)\Mod$ making the functor
\[
A\Mod\to (A,\mathcal M)\Mod: M\mapsto M\otimes_A (A,\mathcal M)
\]
symmetric monoidal.
\item[{\rm (ii)}] The functor
\[
D((A,\mathcal M)\Mod)\to D(A\Mod)
\]
is fully faithful, and its essential image is stable under all limits and colimits and given by those $C\in D(A\Mod)$ such that for all extremally disconnected $S$, the map
\[
R\Hom_A(\mathcal M[S],C)\to R\Hom_A(A[S],C)
\]
is an isomorphism; in that case, also the internal $R\intHom$'s agree. An object $C\in D(A\Mod)$ lies in $D((A,\mathcal M)\Mod)$ if and only if all $H^i(C)$ are in $(A,\mathcal M)\Mod$. The inclusion $D((A,\mathcal M)\Mod)\subset D(A\Mod)$ admits a left adjoint
\[
D(A\Mod)\to D((A,\mathcal M)\Mod): C\mapsto C\otimes^L_A (A,\mathcal M)
\]
that is the left derived functor of $M\mapsto M\otimes_A (A,\mathcal M)$. If $A$ is commutative, there is a unique symmetric monoidal tensor product $\otimes^L_{(A,\mathcal M)}$ on $D((A,\mathcal M)\Mod)$ making the functor
\[
D(A\Mod)\to D((A,\mathcal M)\Mod): C\mapsto C\otimes^L_A (A,\mathcal M)
\]
symmetric monoidal.
\end{enumerate}
\end{proposition}

In the following, for an analytic ring $(A,\mathcal M)$, we abbreviate
\[
D(A,\mathcal M) = D((A,\mathcal M)\Mod)
\]
and refer to it as the derived category of $(A,\mathcal M)$-modules.

\begin{warning} In the most general case, it could happen that $\otimes^L_{(A,\mathcal M)}$ is not the left derived functor of $\otimes_{(A,\mathcal M)}$. In every instance known to us, this holds true, however. It is equivalent to the question whether for all extremally disconnected sets $S$, $T$, the $A$-module in condensed abelian groups $\mathcal M[S\times T]$ (computed by simplicially resolving $S\times T$ through extremally disconnected sets) is concentrated in degree $0$. In practice, this holds true for any profinite set in place of $S\times T$.
\end{warning}

\begin{proof} Most of this follows directly from Lemma~\ref{lem:showeverythingnice}. What remains is to prove that if $C\in D((A,\mathcal M)\Mod)$, then for all extremally disconnected sets $S$, the map
\[
R\intHom_A(\mathcal M[S],C)\to R\intHom_A(A[S],C)
\]
is an isomorphism, and the existence of the symmetric monoidal tensor product if $A$ is commutative. For the question about $R\intHom$, we may assume that $C$ is connective, in which case it is exactly the condition that $(A,\mathcal M)$ is analytic. If $A$ is commutative, this implies that for any extremally disconnected set $T$ and any $C\in D((A,\mathcal M)\Mod)$, the map
\[
R\Hom_A(\mathcal M[S]\otimes_A A[T],C)\to R\Hom_A(A[S\times T],C)
\]
is an isomorphism. This means that the kernel of the Bousfield localization
\[
D(A\Mod)\to D((A,\mathcal M)\Mod)
\]
is a $\otimes$-ideal, which is exactly what we need to prove.
\end{proof}

In particular, Theorem~\ref{thm:solid} says that $\mathbb Z_\solid$ is an analytic ring. With this notation, we have in particular
\[
\Solid = \mathbb Z_\solid\Mod
\]
and the solidification functor
\[
\Cond(\Ab)\to \Solid: M\mapsto M^\solid
\]
has been renamed to
\[
\mathbb Z\Mod\to \mathbb Z_\solid\Mod: M\mapsto M\otimes_{\mathbb Z} \mathbb Z_\solid
\]
and the symmetric monoidal tensor product $\otimes^\solid$ to $\otimes_{\mathbb Z_\solid}$. On the derived level, the functor
\[
D(\Cond(\Ab))\to D(\Solid): C\mapsto C^{L\solid}
\]
is now denoted
\[
D(\mathbb Z\Mod)\to D(\mathbb Z_\solid): C\mapsto C\otimes^L_{\mathbb Z} \mathbb Z_\solid\ .\footnote{We might even go as far as denoting the source by $D(\mathbb Z)$.}
\]
The symmetric monoidal tensor product $\otimes^{L\solid}$ on $D(\mathbb Z_\solid)$ is now denoted $\otimes^L_{\mathbb Z_\solid}$.

There is an obvious notion of a map between condensed rings equipped with theories of measures: A map $(A,\mathcal M)\to (B,\mathcal N)$ is a map $A\to B$ of condensed rings together with $A$-linear maps
\[
\mathcal M[S]\to \mathcal N[S]
\]
natural in the extremally disconnected set $S$ and commuting with the map from $S$.

We will actually define maps of analytic rings slightly differently, as follows. A map of analytic rings $f: (A,\mathcal M)\to (B,\mathcal N)$ is a map of underlying condensed rings $f: A\to B$ with the property that for all extremally disconnected $S$, the $A$-module $\mathcal N[S]$ lies in $(A,\mathcal M)\Mod$. Note that this implies that the map $S\to \mathcal B[S]$ extends uniquely to a map $\mathcal M[S]\to \mathcal N[S]$, functorial in $S$, so this recovers a map of pre-analytic rings in the sense above. In the appendix, we will check that conversely, a map of pre-analytic rings between analytic rings actually defines a map of analytic rings under a very mild assumption.

The next proposition ensures that the formation of $(A,\mathcal M)\Mod$ is functorial in the analytic ring $(A,\mathcal M)$.

\begin{proposition}\label{prop:functoriality} Let $f: (A,\mathcal M)\to (B,\mathcal N)$ be a map of analytic rings.
\begin{enumerate}
\item[{\rm (i)}] The composite functor
\[
A\Mod\xrightarrow{-\otimes_A B} B\Mod\xrightarrow{-\otimes_B (B,\mathcal N)} (B,\mathcal N)\Mod
\]
factors over $(A,\mathcal M)\Mod$, via a functor denoted
\[
(A,\mathcal M)\Mod\to (B,\mathcal N)\Mod: M\mapsto M\otimes_{(A,\mathcal M)} (B,\mathcal N)\ .
\]
\item[{\rm (ii)}] The composite functor
\[
D(A\Mod)\xrightarrow{-\otimes^L_A B} D(B\Mod)\xrightarrow{-\otimes^L_B (B,\mathcal N)} D(B,\mathcal N)
\]
factors over $D(A,\mathcal M)$, via the left derived functor
\[
D(A,\mathcal M)\to D(B,\mathcal N): C\mapsto C\otimes^L_{(A,\mathcal M)} (B,\mathcal N)
\]
of $M\mapsto M\otimes_{(A,\mathcal M)} (B,\mathcal N)$.
\end{enumerate}
\end{proposition}

\begin{proof} For (i), it is equivalent to show the factorization for right adjoints, i.e.~that the forgetful functor
\[
(B,\mathcal N)\Mod\to A\Mod
\]
has image in $(A,\mathcal M)\Mod$. By our definition of maps of analytic rings, this is automatic for all $\mathcal N[S]$, and thus all objects of $(B,\mathcal N)\Mod$, as these generate.

For part (ii), we similarly have to see that if $C\in D(B,\mathcal N)$, then $C$ regarded as an object of $D(A\Mod)$ lies in $D(A,\mathcal M)$. As this can be tested on cohomology groups, this reduces to part (i). The induced functor must be the left derived functor as it commutes with all colimits and takes the compact projective generators $\mathcal M[S]$ to $\mathcal N[S]$.
\end{proof}

Moreover, we have the following examples of analytic rings.

\begin{proposition}\label{prop:exanalytic} The following are analytic rings.
\begin{enumerate}
\item[{\rm (i)}] The theory of $p$-adic solid measures $\mathbb Z_{p,\solid}$.
\item[{\rm (ii)}] For any discrete ring $A$, the theory of $(A,\mathbb Z)_\solid$-measures.
\end{enumerate}
\end{proposition}

\begin{proof} We can prove both parts simultaenously; let $A$ denote $\mathbb Z_p$ resp.~$A$. Let
\[
C: \ldots \to C_i\to\ldots\to C_1\to C_0\to 0
\]
be a complex of $A$-modules in condensed abelian groups as in Definition~\ref{def:analytic}. We need to see that for all extremally disconnected sets $S$, the map
\[
R\intHom_{A}(A[S]^{L\solid},C)\to R\intHom_{A}(A[S],C)=R\intHom(\mathbb Z[S],C)
\]
is an isomorphism.

We note that $C$ is in particular in $D(\mathbb Z_\solid)$. Moreover, in $D(\mathbb Z_\solid)$, we have in both cases
\[
\mathcal A[S] = \mathbb Z_\solid[S]\otimes^L_{\mathbb Z_\solid} A\ .
\]
But as $C$ is a module over $A$ in $D(\mathbb Z_\solid)$, we have formally that
\[
R\intHom(\mathbb Z_\solid[S],C) = R\intHom_A(\mathbb Z_\solid[S]\otimes^L_{\mathbb Z_\solid} A,C)\ .
\]
The left-hand side is $R\intHom(\mathbb Z[S],C)$ as $C$ is solid, while the right-hand side is $R\intHom_A(\mathcal A[S],C)$, giving the desired result.
\end{proof}

\begin{remark} In fact, for any Huber pair $(A,A^+)$, one can define an analytic ring $(A,A^+)_\solid$. The underlying condensed ring is the condensed ring $\underline{A}$ associated with the topological ring $A$, and the subring $A^+\subset A$ is needed for the specification of the spaces $(A,A^+)_\solid[S]$, cf.~\cite[Section 3.3]{Andreychev}. Some notable specializations are $(A,\mathbb Z)_\solid$ as defined above, and $A_\solid = (A,A)_\solid$ for finitely generated $\mathbb Z$-algebras $A$.
\end{remark}

In particular, in the case of $\mathbb Q_p$, the free modules
\[
(\mathbb Q_p,\mathbb Z_p)_\solid[S] := \mathcal M^b(S,\mathbb Q_p) = \mathcal M(S,\mathbb Z_p)[\tfrac 1p]
\]
are the spaces of bounded $\mathbb Q_p$-valued measures on $S$. It is a natural question whether a similar prescription might work over the real numbers.

\begin{example} One can define a theory of measures over $\mathbb R$ by using the classical theory of signed Radon measures. In other words,
\[
\mathcal M_1[S] := \mathcal M^b(S,\mathbb R)
\]
is the space of bounded $\mathbb R$-valued measures on $S$, with the topology dual to the Banach space $C(S,\mathbb R)$; this makes $\mathcal M^b(S,\mathbb R)$ a Smith space. Analogously to the formula
\[
\mathbb Q_{p,\solid}[S] = \bigcup_n (\varprojlim_i p^{-n} \mathbb Z_p[S_i])
\]
when $S=\varprojlim_i S_i$, one can write
\[
\mathcal M_1[S] = \bigcup_{r>0} \varprojlim_i \mathbb R[S_i]_{\ell^1\leq r}\ ,
\]
where $\mathbb R[S_i]_{\ell^1\leq r}\subset \mathbb R[S_i]$ is the subspace of $\ell^1$-norm at most $r$. Note that the latter formula also gives a description of $\mathcal M_1[S]$ as a condensed set, where all the $\varprojlim_i \mathbb R[S_i]_{\ell^1\leq r}$ are compact Hausdorff.
\end{example}

Unfortunately, $(\mathbb R,\mathcal M_1)$ is not an analytic ring. The resulting theory would be closely related to complete locally convex topological vector spaces. The basic issue is the existence of non-locally convex extensions of complete locally convex topological vector spaces, as found by Ribe, \cite{Ribe}; we refer to \cite{Analytic} for details of the construction. Still Kalton, \cite{Kalton}, showed that such extensions are always $p$-convex for all $p<1$.

To remedy this, we note that for any real number $0<p\leq 1$, one can define a theory of measures $\mathcal M_p$ over $\mathbb R$, where for any extremally disconnected set $S$ we set
\[
\mathcal M_p[S] := \bigcup_{r>0} \varprojlim_i \mathbb R[S_i]_{\ell^p\leq r}
\]
where the $\ell^p$-norm is given by $||(x_i)||_{\ell^p} = (\sum_i |x_i|^p)^{1/p}$. (One needs to assume $p\leq 1$ in order for the transition maps in the projective system to respect the condition $\ell^p\leq r$). The Ribe extension persists for any given $p$; thus, $(\mathbb R,\mathcal M_p)$ is not an analytic ring either. However, in \cite{Analytic}, we prove the following theorem:

\begin{theorem} For any $p\leq 1$, the theory of measures
\[
\mathcal M_{<p}[S] := \varinjlim_{q<p} \mathcal M_q[S]
\]
defines an analytic ring $(\mathbb R,\mathcal M_{<p})$.
\end{theorem}
\newpage

\section*{Appendix to Lecture VII: Constructing maps of analytic rings}

To define a map of analytic rings $f:(A,\mathcal M)\to (B,\mathcal N)$, one has to ensure that $\mathcal N[S]$ lies in $(A,\mathcal M)\Mod$ for all extremally disconnected $S$. This might a priori seem subtle. In practice, one can however often directly define a map of theories of measures in the sense of a map $f:A\to B$ of condensed rings together with a natural transformation $\mathcal M[S]\to \mathcal N[S]$ linear over $f: A\to B$ and commuting with the structure map from $S$.

\begin{proposition} The map $f$ is a map of analytic rings if and only if the following square commutes for all extremally disconnected $S$ and any map $g: S\to A$.

The map $g$ yields a map $S\to \mathcal M[S]: s\mapsto g(s)[s]$, which induces a unique map $\mathcal M[S]\to \mathcal M[S]$ in $(A,\mathcal M)\Mod$. Similarly the composite $S\xrightarrow{g} A\to B$ gives rise to a map $\mathcal N[S]\to \mathcal N[S]$ in $(B,\mathcal N)\Mod$. Then the diagram
\[\xymatrix{
\mathcal M[S]\ar[r]\ar[d] & \mathcal M[S]\ar[d] \\
\mathcal N[S]\ar[r] & \mathcal N[S]
}\]
commutes.
\end{proposition}

\begin{remark} This condition is automatic in the following cases: If $A$ is discrete (as then any map $S\to A$ is locally constant); or if the underlying condensed set of $\mathcal N[S]$ is quasiseparated and $A[S]\to \mathcal M[S]$ has dense image in the sense that there are no maps from the quotient to a nontrivial quasiseparated condensed set (as the diagram commutes after restriction to $A[S]$). This handles all practical cases.
\end{remark}

\begin{proof} We prove that any $N\in (B,\mathcal N)\Mod$ lies in $(A,\mathcal M)\Mod\subset A\Mod$. To check this, we use that for any theory of measures $(A,\mathcal M)$, we get a functor $X\mapsto \mathcal M[X]$ from condensed sets to $A$-modules, the unique colimit-preserving extension of $S\mapsto \mathcal M[S]$. If $(A,\mathcal M)$ is analytic, then this is the left adjoint to the forgetful functor from $(A,\mathcal M)\Mod$ to condensed sets, and in particular, this functor has a natural monad structure. In that case, for any $(A,\mathcal M)$-module $M$, the diagram
\[
\mathcal M[\mathcal M[M]]\rightrightarrows \mathcal M[M]\to M
\]
is a coequalizer diagram. (To check this, note that this property is stable under colimits in $M$, so one can assume that $M=\mathcal M[S]$ for some extremally disconnected set $S$. In that case, the diagram is split.)

Coming back to the general $N\in (B,\mathcal N)\Mod$, we know that $N$ is the coequalizer of the diagram
\[
A[A[N]]\rightrightarrows A[N]
\]
and of the diagram
\[
\mathcal N[\mathcal N[N]]\rightrightarrows \mathcal N[N]\ .
\]
This implies, once we have checked that some diagrams commute, that $N$ is a direct summand (in $A\Mod$) of the coequalizer of the diagram
\[
\mathcal M[A[N]]\rightrightarrows \mathcal M[N].
\]
Here, note that the natural transformation $\mathcal M[S]\to \mathcal N[S]$ induces a natural transformation $\mathcal M[X]\to \mathcal N[X]$ which we can apply to $X=N$ and $X=A[N]$ to get a map from the last displayed diagram to the second-to-last displayed diagram. This will show that $N$ is a retract of an object of $(A,\mathcal M)\Mod$, and thus itself in $(A,\mathcal M)\Mod$.

The nontrivial part is to see that
\[\xymatrix{
\mathcal M[A[N]]\ar[r]\ar[d] & \mathcal M[N]\ar[d] \\
\mathcal N[A[N]]\ar[r] & \mathcal N[N]
}\]
commutes, where the the horizontal maps are the unique maps in $(A,\mathcal M)\Mod$ resp.~$(B,\mathcal N)\Mod$ induced by the map of condensed sets $A[N]\to \mathcal M[N]$ resp.~$A[N]\to \mathcal N[N]$. It is enough to see that for any map $T\to A[N]$ from an extremally disconnected $T$, the associated diagram
\[\xymatrix{
\mathcal M[T]\ar[r]\ar[d] & \mathcal M[N]\ar[d] \\
\mathcal N[T]\ar[r] & \mathcal N[N]
}\]
commutes. Now locally on $T$, any such map $T\to A[N]$ lies in $A(T)[\Hom(T,N)]$, so can be written as a finite sum of terms of the form $g [h]$ where $g\in A(T)$ and $h\in \Hom(T,N)$; we can thus assume that it is of this form $g[h]$. Functoriality in $h: T\to N$ is clear, and functoriality in $g$ is precisely the condition.
\end{proof}

\newpage

\section{Lecture VIII: Solid modules}

The goal of this lecture is to prove the following theorem.

\begin{theorem}\label{thm:solidA} For any finitely generated $\mathbb Z$-algebra $A$, the theory of measures $A_\solid$ on $A$ defines an analytic ring. The forgetful functor
\[
j_\ast: D(A_\solid)\hookrightarrow D((A,\mathbb Z)_\solid)
\]
admits the left adjoint $j^\ast = -\otimes^L_{(A,\mathbb Z)_\solid} A_\solid$, which in turn admits a left adjoint
\[
j_!: D(A_\solid)\hookrightarrow D((A,\mathbb Z)_\solid)\ .
\]
The left adjoint $j_!$ satisfies the projection formula
\[
j_! j^\ast M = M\otimes_{(A,\mathbb Z)_\solid}^L j_! A\ .
\]
\end{theorem}

A consequence of the theorem is that one can define ``coherent cohomology with compact support'' (for now in the affine case):

\begin{theorem}\label{thm:Rfshriek} Let $A$ be a finitely generated $\mathbb Z$-algebra and let $f: \Spec A\to \Spec \mathbb Z$ denote the projection. Consider the functor
\[
f_!: D(A_\solid)\to D(\mathbb Z_\solid)
\]
defined as the composite
\[
D(A_\solid)\xrightarrow{j_!} D((A,\mathbb Z)_\solid)\to D(\mathbb Z_\solid)\ .
\]
Then $f_!$ commutes with all direct sums, preserves compact objects, and satisfies the projection formula
\[
f_!((M\otimes_{\mathbb Z_\solid}^L A_\solid)\otimes^L_{A_\solid} N)\cong M\otimes^L_{\mathbb Z_\solid} f_! N
\]
for $M\in D(\mathbb Z_\solid)$ and $N\in D(A_\solid)$. The functor $f_!$ admits a right adjoint
\[
f^!: D(\mathbb Z_\solid)\to D(A_\solid)
\]
that commutes with all direct sums, and is given by
\[
f^! M = (M\otimes_{\mathbb Z_\solid}^L A_\solid)\otimes^L_{A_\solid} f^! \mathbb Z\ .
\]
The object $f^!\mathbb Z\in D(A_\solid)$ is a bounded complex of finitely generated discrete $A$-modules. If $f$ is a complete intersection, then $f^!\mathbb Z\in D(A)$ is an invertible object, i.e.~locally a line bundle concentrated in some degree.
\end{theorem}

In fact, $f^!\mathbb Z\in D(A)\subset D(A_\solid)$ is exactly the dualizing complex, as defined in coherent duality! Proving this identification with the dualizing complex will be the goal of the forthcoming lectures.\footnote{As explained in the appendix to this lecture, the formalism extends in the same way to the relative case, except that some assertions require the hypothesis that $f$ is of finite Tor-dimension; notably, that $f_!$ preserves compact objects, or equivalently that $f^!$ commutes with all direct sums.} Note that the functor $f^!$ preserves discrete objects, inducing a functor $D(\mathbb Z)\to D(A)$ of usual derived categories; however, our definition of this functor makes critical use of condensed mathematics! Indeed, $f_!$ does not preserve discrete objects, and does not exist (naively) in classical approaches to coherent duality.

\begin{remark}[Local finiteness of coherent cohomology] The assertion that $f_!$ preserves compact objects can be seen as a local version of finiteness of coherent cohomology. Note that $f_!$ manifestly does not preserve discrete objects (unless $f$ is finite); but if it would, then it would send compact discrete objects to compact discrete objects. Now it is easy to see that the compact discrete objects are exactly the usual perfect complexes of $A$-modules, i.e.~complexes that can be represented by finite complexes of finite projective $A$-modules. Thus, this would give preservation of perfect complexes. This argument does indeed apply once we globalize and prove that for proper $f$, the functors $f_!$ and $Rf_\ast$ agree, using that $Rf_\ast$ always preserves discrete objects, thus giving a new proof of finiteness of coherent cohomology for proper maps.
\end{remark}

\begin{remark}[Local Serre duality] The definition of $f^!$ as a right adjoint implies formally that for any $M\in D(A_\solid)$, one has
\[
R\intHom_{\mathbb Z}(f_! M,\mathbb Z)= R\intHom_A(M,f^! \mathbb Z)\ .
\]
If $f$ is a complete intersection, $\omega_A:= f^!\mathbb Z$ is invertible. In this case, if $P$ is a finite projective $A$-module or more generally a perfect complex of $A$-modules, one can set $M=R\Hom_A(P,\omega_A)=P^\vee\otimes_A \omega_A$ and the equation becomes
\[
R\intHom_{\mathbb Z}(f_!(P^\vee\otimes_A \omega_A),\mathbb Z)=P\ .
\]
In other words, the coherent cohomology of $P$, which in this affine case is just $P$ itself, considered as a $\mathbb Z$-module, is written as the dual of the compactly supported cohomology of the usual twist $P^\vee\otimes_A \omega_A$.

In fact, by biduality in $D(\mathbb Z_\solid)$ for compact objects and the compactness of $f_!(P^\vee\otimes_A \omega_A)$, we can also rewrite this as
\[
f_!(P^\vee\otimes_A \omega_A) = R\intHom_{\mathbb Z}(P,A)
\]
or as a perfect pairing
\[
f_!(P^\vee\otimes_A \omega_A)\otimes_{\mathbb Z} P\to \mathbb Z\ .
\]
\end{remark}

\begin{remark}[A new formula for the dualizing complex] The adjointness of $f^!$ implies that
\[
f^! \mathbb Z = R\intHom_{\mathbb Z}(f_! A,\mathbb Z)\ .
\]
By Theorem~\ref{thm:solidA}, $f_! A$ is compact as an object of $D(\mathbb Z_\solid)$ so that its dual is discrete, and still carries an $A$-module structure. As we will see below, $f_! A$ can be computed in practice in terms of ``functions near the boundary of $\Spec A$"; this yields a new formula for the dualizing complex. For example, if $A=\mathbb Z[X,Y]/XY$, then the functions near the boundary are $A_\infty = \mathbb Z((X^{-1}))\times \mathbb Z((Y^{-1}))$, and one has
\[
f_! A [1] = A_\infty/A = (\mathbb Z((X^{-1}))\times \mathbb Z((Y^{-1})))/(\mathbb Z[X,Y]/XY))
\]
and so the dualizing complex is given by
\[
f^! \mathbb Z = R\intHom_{\mathbb Z}((\mathbb Z((X^{-1}))\times \mathbb Z((Y^{-1})))/(\mathbb Z[X,Y]/XY),\mathbb Z)[1]\ .
\]
\end{remark}

In the rest of this lecture, we will prove Theorem~\ref{thm:solidA} in the special case $A=\mathbb Z[T]$. In the appendix of this lecture, we will explain the modifications required to handle the general case.

Thus, we fix $A=\mathbb Z[T]$. Let $A_\infty = \mathbb Z((T^{-1}))$ with its natural condensed structure. Then $A_\infty$ is naturally a condensed $A$-algebra, and its underlying condensed abelian group is solid. In particular, it defines an object $A_\infty\in D((A,\mathbb Z)_\solid)$. We like to think of $\mathbb Z((T^{-1}))$ as ``functions near the boundary of $\Spec \mathbb Z[T]$".

\begin{observation}\label{obs:compact} The object $A_\infty\in D((A,\mathbb Z)_\solid)$ is compact (in the general sense of triangulated categories: $\Hom(A_\infty,-)$ commutes with all direct sums). Indeed, we have the short exact sequence
\[
0\to \mathbb Z[[U]]\otimes_{\mathbb Z} A\xrightarrow{UT-1} \mathbb Z[[U]]\otimes_{\mathbb Z} A\to A_\infty\to 0
\]
where the first two terms are compact projective in $(A,\mathbb Z)_\solid\Mod$.
\end{observation}

\begin{observation}\label{obs:Ainftyidempotent} The ring $A_\infty\in D((A,\mathbb Z)_\solid)$ is idempotent in the sense that the multiplication map
\[
A_\infty\otimes^L_{(A,\mathbb Z)_\solid} A_\infty\to A_\infty
\]
is an isomorphism. Indeed, this is a simple computation from the presentation of $A_\infty$ in the previous observation.

In particular, the $A_\infty$-modules form a full subcategory of $D((A,\mathbb Z)_\solid)$: Any $M\in D((A,\mathbb Z)_\solid)$ admits at most one $A_\infty$-module structure, and such a module structure exists precisely when the map
\[
M\to M\otimes_{(A,\mathbb Z)_\solid}^L A_\infty
\]
is an isomorphism.
\end{observation}

\begin{observation}\label{obs:nomaps} Let $C\in D(A\Mod)$ be any complex of condensed $A$-modules all of whose terms are direct sums of products of copies of $A$. Then
\[
R\intHom_A(A_\infty,C)=0\ .\footnote{This is an analogue of the assertion $R\intHom_{\mathbb Z}(\mathbb R,C)=0$ when $C$ is a complex of solid abelian groups.}
\]
To see this, we first observe that $C\in D((A,\mathbb Z)_\solid)$ as $(A,\mathbb Z)_\solid$ is an analytic ring and so $D((A,\mathbb Z)_\solid)\subset D(A\Mod)$ is stable under all limits and colimits. We may assume that $C$ is connective by writing it as the limit of its stupid truncations. Now $A_\infty\in D((A,\mathbb Z)_\solid)$ is compact by Observation~\ref{obs:compact}, so it is enough to consider the case that $C=\prod_I A$ for some set $I$. This reduces to $C=A$, so it remains to prove that
\[
R\intHom_A(A_\infty,A)=0\ .
\]
For this, we use the resolution from Observation~\ref{obs:compact} to see that $R\intHom_A(A_\infty,A)$ can be computed by the two term complex
\[
R\intHom_{\mathbb Z}(\mathbb Z[[U]],A)\xrightarrow{UT-1} R\intHom_{\mathbb Z}(\mathbb Z[[U]],A)
\]
which is given by
\[
A[U^{-1}]/A\xrightarrow{UT-1} A[U^{-1}]/A
\]
which is acyclic.
\end{observation}

\begin{observation}\label{obs:cokernelAinftymodule} For any set $I$, the cokernel of the injective map
\[
\prod_I \mathbb Z\otimes_{\mathbb Z} A\to \prod_I A
\]
is an $A_\infty$-module. Indeed, this cokernel can also be written as the cokernel of the map
\[
\prod_I \mathbb Z[[T^{-1}]]\otimes_{\mathbb Z[[T^{-1}]]} \mathbb Z((T^{-1}))\to \prod_I \mathbb Z((T^{-1}))
\]
as there is a natural map from the first sequence to the second sequence, and on cokernels one gets the identity map
\[
\prod_I T^{-1} \mathbb Z[[T^{-1}]]\to \prod_I T^{-1} \mathbb Z[[T^{-1}]]\ .
\]
But the second sequence is naturally a map of $A_\infty = \mathbb Z((T^{-1}))$-modules.
\end{observation}

Now we can prove that $\mathbb Z[T]_\solid$ is an analytic ring. Let $C$ be any connective complex of condensed $A$-modules all of whose terms are sums of products of copies of $A$. We need to see that for any profinite set $S$, we have
\[
R\intHom_A(A[S],C)\cong R\intHom_A(A_\solid[S],C)\ .
\]
As $C\in D((A,\mathbb Z)_\solid)$, we already know that
\[
R\intHom_A(A[S],C)\cong R\intHom_A((A,\mathbb Z)_\solid[S],C)\ .
\]
Now $\mathbb Z_\solid[S]\cong \prod_I \mathbb Z$ for some set $I$, and then $(A,\mathbb Z)_\solid[S] = \prod_I \mathbb Z\otimes_{\mathbb Z} A$ and $A_\solid[S] = \prod_I A$. It remains to see that
\[
R\intHom_A(\prod_I \mathbb Z\otimes_{\mathbb Z} A,C) = R\intHom_A(\prod_I A,C)\ .
\]
This follows from Observations~\ref{obs:cokernelAinftymodule} and~\ref{obs:nomaps}.

In particular, we get the full inclusion
\[
j_\ast: D(A_\solid)\hookrightarrow D((A,\mathbb Z)_\solid)
\]
as both are full subcategories of $D(A\Mod)$. This admits a left adjoint $j^\ast$ which is $M\mapsto M\otimes^L_{(A,\mathbb Z)_\solid} A_\solid$, by Proposition~\ref{prop:functoriality}.

\begin{observation}\label{obs:kernelAinftymodules} The kernel of the localization functor
\[
j^\ast: D((A,\mathbb Z)_\solid)\to D(A_\solid)
\]
is precisely the full subcategory of all $M\in D((A,\mathbb Z)_\solid)$ that are $A_\infty$-modules. Indeed, if $M$ is an $A_\infty$-module, then $j^\ast M$ is a module over $A_\infty\otimes^L_{(A,\mathbb Z)_\solid} A_\solid=0$ by Observation~\ref{obs:nomaps}. Conversely, the kernel is generated by the objects considered in Observation~\ref{obs:cokernelAinftymodule}, which we have seen there to be $A_\infty$-modules. By Observation~\ref{obs:Ainftyidempotent}, we see that all objects in the kernel are $A_\infty$-modules.
\end{observation}

\begin{observation}\label{obs:jshriekexists} Consider the endofunctor
\[
j_!: D((A,\mathbb Z)_\solid)\to D((A,\mathbb Z)_\solid): M\mapsto M\otimes_{A\otimes^L_{\mathbb Z}\mathbb Z_\solid} (A_\infty/A)[-1]\ .
\]
This functor factors uniquely over a fully faithful functor
\[
j_!: D(A_\solid)\to D((A,\mathbb Z)_\solid)
\]
that is a left adjoint of $j^\ast$. To prove this, note first that fully faithfulness of $j_!$ follows from the other assertions, as $j^\ast j_!$ will be the left adjoint of $j^\ast j_\ast$ which is the identity, and so $j^\ast j_!$ is itself the identity, giving full faithfulness of $j_!$.

Now it suffices to see that for any $M,N\in D((A,\mathbb Z)_\solid)$, we have
\[
R\intHom_A(j_!M,N)\buildrel\cong\over\rightarrow R\intHom_A(j_!M,N\otimes^L_{(A,\mathbb Z)_\solid} A_\solid)\buildrel\cong\over\leftarrow R\intHom_A(M,N\otimes^L_{(A,\mathbb Z)_\solid} A_\solid)\ .
\]
Indeed, this implies that $R\Hom_A(j_!M,-)$ depends only on the image of $M$ in $D(A_\solid)$, giving the desired factorization; and then the displayed equation proves the desired adjointness. Note that the second equation follows from Observation~\ref{obs:nomaps} and the observation that the cone of $j_!M\to M$ is by definition of $j_!M$ an $A_\infty$-module. 

For the left map, we use that the cone of $N\to N\otimes^L_{(A,\mathbb Z)_\solid} A_\solid$ is an $A_\infty$-module by Observation~\ref{obs:kernelAinftymodules}, so that it is enough to see that
\[
j_!M\otimes^L_{(A,\mathbb Z)_\solid} A_\infty = 0\ .
\]
But this follows from the definition of $j_!M$ and Observation~\ref{obs:Ainftyidempotent}.
\end{observation}

At this point, we have proved all of Theorem~\ref{thm:solidA}. We observe that the arguments are roughly similar to the arguments for $\mathbb Z_\solid$, with $\mathbb Z((T^{-1}))$ taking the role of $\mathbb R$. The case of $\mathbb Z_\solid$ was somewhat more subtle as we did not have yet have a nice analytic base ring to base our arguments on, and so the arguments involving $\mathbb R$ had to be done in a very ad-hoc fashion.

In order to prepare for the proof of Theorem~\ref{thm:Rfshriek}, in particular the preservation of compactness under $f_!$, we make another observation.

\begin{observation}\label{obs:compact2} For any set $I$, the natural map
\[
j_! \prod_I A\to \prod_I (A_\infty/A)[-1]
\]
is an isomorphism (and the right-hand side is compact as an object of $D(\mathbb Z_\solid)$). This follows from the following computation:
\[\begin{aligned}
j_! \prod_I A &= j_! j^\ast(\prod_I\mathbb Z\otimes_{\mathbb Z} A) = (\prod_I\mathbb Z\otimes_{\mathbb Z} A)\otimes^L_{(A,\mathbb Z)_\solid} (A_\infty/A)[-1]\\
&= \prod_I\mathbb Z\otimes^L_{\mathbb Z_\solid} (A_\infty/A)[-1] = \prod_I (A_\infty/A)[-1]\ .
\end{aligned}\]
Here we use the definition of $j_!j^\ast$ in the second equation, and Proposition~\ref{prop:tensorinfproducts} in the last equation.
\end{observation}

Finally, we give a proof of Theorem~\ref{thm:Rfshriek} when $A=\mathbb Z[T]$. By definition $f_!$ commutes with all direct sums (as $j_!$ does as a left adjoint, and the forgetful functor does). This implies formally that $f_!$ admits a right adjoint $f^!$. Moreover, in the case of compactly generated triangulated categories, it is known that $f^!$ commutes with all direct sums exactly when $f_!$ preserves compact objects. The compact objects of $D(A_\solid)$ are generated by $\prod_I A$'s, and for these the claim is Observation~\ref{obs:compact2}. For the projection formula, it is enough to prove the finer assertion
\[
j_!((M\otimes^L_{\mathbb Z_\solid} A_\solid)\otimes^L_{A_\solid} N)\cong (M\otimes^L_{\mathbb Z} A)\otimes^L_{(A,\mathbb Z)_\solid} j_! N
\]
in $D((A,\mathbb Z)_\solid)$. This holds true as it holds true after applying $j^\ast$, and as both sides vanish after tensoring with $A_\infty$.

Note that the adjunction
\[
R\Hom_{\mathbb Z}(f_! M,\mathbb Z) = R\Hom_A(M,f^!\mathbb Z)
\]
implies, by taking $M=A$, that
\[
f^! \mathbb Z = R\Hom_{\mathbb Z}(f_! A,\mathbb Z)
\]
which, as $f_! A$ is compact, implies that $f^! \mathbb Z$ is discrete. In fact, we can compute
\[
f^!\mathbb Z = R\Hom_{\mathbb Z}(\mathbb Z((T^{-1}))/\mathbb Z[T],\mathbb Z)[1]\cong \mathbb Z[T][1]\ ,
\]
which is an invertible $A$-module.

By formal nonsense, we have a natural map
\[
(M\otimes^L_{\mathbb Z_\solid} A_\solid)\otimes^L_{A_\solid} f^! \mathbb Z\to f^! M\ :
\]
indeed, defining this map is equivalent to defining an adjoint map
\[
f_!((M\otimes^L_{\mathbb Z_\solid} A_\solid)\otimes^L_{A_\solid} f^! \mathbb Z)\to M
\]
in $D(A_\solid)$. This comes from the projection formula
\[
f_!((M\otimes^L_{\mathbb Z_\solid} A_\solid)\otimes^L_{A_\solid} f^! \mathbb Z)\cong M\otimes^L_{\mathbb Z_\solid} f_! f^! \mathbb Z
\]
and the counit $f_! f^! \mathbb Z\to \mathbb Z$.

To prove that the map
\[
(M\otimes^L_{\mathbb Z_\solid} A_\solid)\otimes^L_{A_\solid} f^! \mathbb Z\to f^! M
\]
is an isomorphism, we note that both sides commute with all direct sums, so we can reduce to $M=\prod_I \mathbb Z$. As $f^!$ commutes with products, we need to see the same for the left-hand side. But $\prod_I \mathbb Z\otimes^L_{\mathbb Z_\solid} A_\solid = \prod_I A$ and $f^!\mathbb Z\cong A[1]$ is just a shift.
\newpage

\section*{Appendix to Lecture VIII: General case}

In this appendix, we prove the general version of Theorem~\ref{thm:solidA}, and in fact formulate a relative version. Fix a map of finitely generated $\mathbb Z$-algebras $R\to A$. We define a theory of measures $(A,R)_\solid$ on $A$, by
\[
(A,R)_\solid[S] := R_\solid[S]\otimes_R A\ .
\]
Note that $R_\solid[S]\otimes^L_R A = R_\solid[S]\otimes_R A$: More generally, for any $R$-module $M$ and any set $I$, one has
\[
\prod_I R\otimes^L_R M = \prod_I R\otimes_R M\ .
\]
As both sides commute with filtered colimits in $M$, it suffices to check this when $M$ is finitely generated. But then in fact
\[
\prod_I R\otimes^L_R M = \prod_I M
\]
(which is concentrated in degree $0$), as one sees by resolving $M$ by finite free $R$-modules.

With this definition, we have the following generalization of Theorems~\ref{thm:solidA} and~\ref{thm:Rfshriek}.

\begin{theorem}\label{thm:solidAR}\leavevmode
\begin{enumerate}
\item[{\rm (i)}] For any map of finitely generated $\mathbb Z$-algebras $R\to A$, the theory of measures $(A,R)_\solid$ defines an analytic ring.
\item[{\rm (ii)}] For any maps $R\to S\to A$ of finitely generated $\mathbb Z$-algebras, the forgetful functor
\[
j_\ast: D((A,S)_\solid)\hookrightarrow D((A,R)_\solid)
\]
admits the left adjoint $j^\ast = -\otimes^L_{(A,R)_\solid} (A,S)_\solid$, which in turn admits a left adjoint
\[
j_!: D((A,S)_\solid)\hookrightarrow D((A,R)_\solid)\ .
\]
The left adjoint $j_!$ satisfies the projection formula
\[
j_! j^\ast M = M\otimes_{(A,R)_\solid}^L j_! A\ .
\]
\item[{\rm (iii)}] For a map $f: \Spec A\to \Spec R$ of affine finite type $\mathbb Z$-schemes, consider the functor
\[
f_!: D(A_\solid)\to D(R_\solid)
\]
defined as the composite
\[
D(A_\solid)\xrightarrow{j_!} D((A,R)_\solid)\to D(R_\solid)\ .
\]
Then $f_!$ commutes with all direct sums and satisfies the projection formula
\[
f_!((M\otimes_{R_\solid}^L A_\solid)\otimes^L_{A_\solid} N)\cong M\otimes^L_{R_\solid} f_! N
\]
for $M\in D(R_\solid)$ and $N\in D(A_\solid)$. If $f$ is of finite Tor-dimension, then $f_!$ preserves compact objects. The formation of $f_!$ is compatible with composition, i.e.~if $g: \Spec R\to \Spec R^\prime$ is another map of affine finite type $\mathbb Z$-schemes, then $(g\circ f)_!\cong g_!\circ f_!$.
\item[{\rm (iv)}] In the situation of (iii), the functor $f_!$ admits a right adjoint
\[
f^!: D(R_\solid)\to D(A_\solid)\ .
\]
The object $f^! R\in D(A_\solid)$ is discrete and a bounded to the left complex of finitely generated $A$-modules. If $f$ is of finite Tor-dimension, then $f^!R$ is bounded, $f^!$ commutes with all direct sums, and is given by
\[
f^! M = (M\otimes_{R_\solid}^L A_\solid)\otimes^L_{A_\solid} f^! R\ .
\]
If $f$ is a complete intersection, then $f^! R\in D(A)$ is an invertible object, i.e.~locally a line bundle concentrated in some degree. The formation of $f^!$ is compatible with composition, i.e.~if $g: \Spec R\to \Spec R^\prime$ is another map of affine finite type $\mathbb Z$-schemes, then $(g\circ f)^!\cong f^!\circ g^!$.
\end{enumerate}
\end{theorem}

\begin{proof} For part (i), we note that by the proof of Proposition~\ref{prop:exanalytic}~(ii) and the observed identity $R_\solid[S]\otimes^L_R A = R_\solid[S]\otimes_R A$, the case of $(A,R)_\solid$ reduces to the case of $R_\solid$. Thus, it is enough to prove that for any finitely generated $\mathbb Z$-algebra $A$, $A_\solid$ is an analytic ring. But now we can choose a surjection $R=\mathbb Z[X_1,\ldots,X_n]\to A$, in which case $A_\solid = (A,R)_\solid$ as then $A$ is a finitely generated $R$-module and so
\[
\prod_I A\cong \prod_I R\otimes_R A\ .
\]
Thus, the case of $A_\solid$ is equivalent to the case of $(A,R)_\solid$, which reduces to $R_\solid$. In other words, it is enough to handle the case of $A_\solid$ when $A=\mathbb Z[X_1,\ldots,X_n]$ is a polynomial algebra. We handle this case by induction on $n$. Let $R=\mathbb Z[X_1,\ldots,X_{n-1}]\to A$. Then $R_\solid$ and hence $(A,R)_\solid$ are analytic rings by induction. Now we can repeat all arguments from the lecture with the base ring $\mathbb Z$ replaced by $R$, in particular using the ring $A_\infty = R((X_n^{-1}))$, to get the result for $A$.

For part (ii), the existence and description of $j^\ast$ is now a consequence of Proposition~\ref{prop:functoriality}. For the properties of $j_!$, we note that the question reduces formally to the case $A=S$, as afterwards one simply carries around an $A$-module structure on top. In particular, we can assume that $S\to A$ is surjective. We can then actually assume that $S=R[X_1,\ldots,X_n]$ (changing now $S$ while fixing $A$ and $R$), noting that $(A,S)_\solid=A_\solid$ is independent of $S$ when $S\to A$ is surjective. Now using again the reduction from $R\to S\to A$ to $R\to S=S$, we are reduced to the case that $A=S=R[X_1,\ldots,X_n]$ is a polynomial algebra over $R$. In that case, we use the factorization
\[
(A,R)_\solid\to (A,R[X_1,\ldots,X_{n-1}])_\solid\to (A,A)_\solid
\]
using which we can by induction reduce to the case $n=1$. Hence finally again $A=R[X]$ over some general base ring $R$, where we can repeat the arguments from the lecture.

In part (iii), one first checks compatibility with composition, i.e.~$(g\circ f)_!\cong g_!\circ f_!$. This comes down to the commutativity of the following diagram:
\[\xymatrix{
D((A,R)_\solid)\ar[r]^{j_!}\ar[d] & D((A,R^\prime)_\solid)\ar[d]\\
D(R_\solid)\ar[r]^{j_!} & D((R,R^\prime)_\solid)
}\]
As in the proof of (ii), this is formal as the functor on top is simply carrying around an additional $A$-module structure.

For the other assertions in (iii), we can now reduce to the case that either $A=R[X_1,\ldots,X_n]$ is a polynomial algebra over $R$ or $R\to A$ is surjective; and the case of the polynomial algebra can further by induction be reduced to $A=R[X]$. This case was handled in the lecture (modulo replacing $\mathbb Z$ by $R$ everywhere). If $R\to A$ is surjective, then $A_\solid=(A,R)_\solid$ and thus $j_!$ is the identity and $f_!$ is simply the forgetful functor. The projection formula is then easily verified, while if $f$ is of finite Tor-dimension, then the compact projective generators $\prod_I A$ of $D(A_\solid)$ are also compact as objects of $D(R_\solid)$, by taking a finite resolution of $A$ by finite projective $R$-modules.

In part (iv), the existence of $f^!$ is formal, and the commutation with colimits is automatic when $f$ is of finite Tor-dimension as then $f_!$ preserves compact objects. It is also clear that the formation of $f^!$ is compatible with composition. For the other assertions about $f^!$ we can now again reduce to the case $A=R[X]$ or $R\to A$ is surjective. The case $A=R[X]$ was analyzed in the lecture (modulo replacing $\mathbb Z$ by $R$). In particular, note that
\[
f^! R = R\intHom_R(R((X^{-1}))/R[X],R)[1]\cong R[X][1]
\]
is invertible over $A=R[X]$. If $R\to A$ is surjective, then
\[
f^! R = R\intHom_R(A,R)
\]
which is a bounded to the left complex of finitely generated $A$-modules. If $f$ is of finite Tor-dimension, it is in fact bounded (as seen by resolving $A$ by a finite complex of finite projective $R$-modules). In that case, one can also see that
\[
f^! M = (M\otimes^L_{R_\solid} A_\solid)\otimes^L_{A_\solid} f^! R
\]
by constructing the map as in the lecture, and then reducing (as both sides commute with colimits) to the compact projective generators $M=\prod_I R$, for which one observes that both sides commute with this product, as $f^! R$ is pseudocoherent.

Finally, if $f$ is a complete intersection, the usual Koszul computation shows that $f^! R = R\intHom_R(A,R)$ is invertible in $D(A)$.
\end{proof}
\newpage

\section{Lecture IX: Globalization}

In this lecture, we want to globalize the previous constructions. Recall that to any map $R\to A$ of finitely generated $\mathbb Z$-algebras, we have associated the derived category
\[
D((A,R)_\solid)\ .
\]
In fact, if $A^+\subset A$ is the integral closure of the image of $R$ in $A$, then $(A,R)_\solid = (A,A^+)_\solid$ as
\[
\prod_I R\otimes_R A = \prod_I A^+\otimes_{A^+} A\ ,
\]
noting that as $A^+$ is finite over $R$, one has
\[
\prod_I R\otimes_R A^+ = \prod_I A^+\ .
\]
In particular, we really have a functor that takes pairs $(A,A^+)$ consisting of a finitely generated $\mathbb Z$-algebra $A$ and an integrally closed finitely generated subalgebra $A^+\subset A$ to
\[
D((A,A^+)_\solid)\ .
\]

This indicates that the globalization is naturally done in the language of adic spaces. We refer to \cite{HuberContVal}, \cite{HuberGeneralization} for some basic assertions in the following.

\begin{definition} A discrete Huber pair is a pair $(A,A^+)$ consisting of a discrete ring $A$ together with an integrally closed subring $A^+\subset A$.
\end{definition}

There is a functor from discrete Huber pairs to analytic rings, given by
\[
(A,A^+)\mapsto (A,A^+)_\solid := \varinjlim_{(B,B^+)\to (A,A^+)} (B,B^+)_\solid\ ,
\]
where $(B,B^+)$ runs over discrete Huber pairs mapping to $(A,A^+)$ with both $B$ and $B^+$ finitely generated $\mathbb Z$-algebras.

To any discrete Huber pair $(A,A^+)$, one can associate the space
\[
\Spa(A,A^+) = \{|\cdot|: A\to \Gamma\cup \{0\}\mid |A^+|\leq 1\}/\cong
\]
of equivalence classes of valuations on $A$ that are $\leq 1$ on $A^+$. Here $\Gamma$ is a totally ordered abelian group, written multiplicatively, depending on the valuation $|\cdot|$, and the conditions of being a valuation are
\[
|0|=0, |1| = 1, |xy|=|x||y|, |x+y|\leq \max(|x|,|y|)
\]
(with the convention that $0< \gamma$ and $0\gamma = 0$ for all $\gamma\in \Gamma$). We will generally abuse notation and write $x\in \Spa(A,A^+)$ for a point of $\Spa(A,A^+)$ and $f\mapsto |f(x)|$ for the associated valuation. The space $\Spa(A,A^+)$ is given a topology where a basis of quasicompact open subsets is given by the rational subsets
\[
U(\frac{g_1,\ldots,g_n}f) := \{x\mid \forall i=1,\ldots,n: |g_i(x)|\leq |f(x)|\neq 0\}
\]
for varying $f,g_1,\ldots,g_n\in A$. With this topology $\Spa(A,A^+)$ is a spectral space.

The following proposition clarifies the role of $A^+\subset A$. We let
\[
\Spv(A) = \Spa(A,\widetilde{\mathbb Z}) = \{|\cdot|: A\to \Gamma\cup\{0\}\}/\cong
\]
be the space of all equivalence classes of valuations on $A$, where $\widetilde{\mathbb Z}\subset A$ denotes the integral closure of $\mathbb Z$ in $A$.

\begin{proposition}\label{prop:characterizeAplus} There is a bijection between integrally closed subrings $A^+\subset A$ and subsets $U\subset \Spv(A)$ that are intersections of subsets of the form $U_{1,f}$, given by
\[
A^+\mapsto U=\Spa(A,A^+)=\{x\in \Spv(A)\mid \forall f\in A^+, |f(x)|\leq 1\}=\bigcap_{f\in A^+} U_{1,f}\ ,
\]
\[
U\mapsto A^+=\{f\in A\mid \forall x\in U, |f(x)|\leq 1\}\ .
\]
In particular,
\[
A^+=\{f\in A\mid \forall x\in \Spa(A,A^+), |f(x)|\leq 1\}\ .
\]
\end{proposition}

\begin{proof} It is clear that for all $U$, the subset $A^+=\{f\in A\mid \forall x\in U, |f(x)|\leq 1\}\subset A$ is an integrally closed subring. We note to see that the maps are inverse. The association $A^+\mapsto U=\Spa(A,A^+)$ is by definition surjective: If $U=\bigcap_{f\in I} U_{1,f}$ for some set $I$, then we can take $A^+$ to be the integral closure of the subring of $A$ generated by all $f$ for $f\in I$. Thus, it remains to see that starting with $A^+$, one has
\[
A^+ = \{f\in A\mid \forall x\in \Spa(A,A^+), |f(x)|\leq 1\}\ .
\]
Assume that $f\not\in A^+$. Then $f$ is also not in $A^+[\tfrac 1f]\subset A[\tfrac 1f]$ as otherwise $f$ would be integral over $A^+$. In particular, we can find a prime ideal $\mathfrak p$ of $A^+[\tfrac 1f]$ that contains $\tfrac 1f$. Let $\mathfrak q$ be a minimal prime ideal of $A^+[\tfrac 1f]$ contained in $\mathfrak p$. We may then find a valuation ring $V$ with a map $\Spec V\to \Spec A^+[\tfrac 1f]$ taking the generic point to $\mathfrak q$ and the special point to $\mathfrak p$. As the image of $\Spec A[\tfrac 1f]\to \Spec A^+[\tfrac 1f]$ contains the minimal prime $\mathfrak q$, we can lift the valuation corresponding to $\Spec V\to \Spec A^+[\tfrac 1f]$ to $A[\tfrac 1f]$. The resulting valuation of $A$ takes values $\leq 1$ on all elements of $A^+$ by construction, and value $>1$ on $f$ as $\tfrac 1f\in \mathfrak p$. This gives the desired contradiction.
\end{proof}

\begin{proposition} Let $(A,A^+)\to (B,B^+)$ be a map of pairs as above such that the induced map
\[
\Spa(B,B^+)\to \Spa(A,A^+)
\]
factors over the rational subset $U=U(\frac{g_1,\ldots,g_n}f)\subset \Spa(A,A^+)$. Then the map $(A,A^+)\to (B,B^+)$ factors uniquely over the pair $(A[\tfrac 1f],\widetilde{A^+[\tfrac{g_1}f,\ldots,\tfrac{g_n}f]})$ where $\widetilde{A^+[\tfrac{g_1}f,\ldots,\tfrac{g_n}f]}\subset A[\tfrac 1f]$ is the integral closure of $A^+[\tfrac{g_1}f,\ldots,\tfrac{g_n}f]$ in $A[\tfrac 1f]$. Moreover, the map
\[
\Spa(A[\tfrac 1f],\widetilde{A^+[\tfrac{g_1}f,\ldots,\tfrac{g_n}f]})\to \Spa(A,A^+)
\]
is a homeomorphism onto $U$.
\end{proposition}

\begin{proof} As $\Spa(B,B^+)\to \Spa(A,A^+)$ factors over the locus where $|f(x)|\neq 0$, it follows that there are no $y\in \Spa(B,B^+)$ with $|f(y)|=0$. But there is a natural map $\Spec(B)\hookrightarrow \Spa(B,B^+)$ sending any prime ideal $\mathfrak p$ to the valuation $B\to \mathrm{Frac}(B/\mathfrak p)\to \{0,1\}$; if $y$ is not invertible on $B$, then any point of $\Spec(B/fB)$ would define a point $y\in \Spa(B,B^+)$ with $|f(y)|=0$, contradiction. Thus, $y$ is invertible in $B$, and we get the desired unique map $A[\tfrac 1f]\to B$. Moreover, we know that $h_i:=\frac{g_i}f\in B$ is an element such that $|h_i(y)|\leq 1$ for all $y\in \Spa(B,B^+)$. This implies that $h_i\in B^+$ by Proposition~\ref{prop:characterizeAplus}. In particular, we get the desired map $A^+[\tfrac{g_1}f,\ldots,\tfrac{g_n}f]\to B^+$, which moreover extends to the integral closure.

The final statement follows easily by unraveling definitions.
\end{proof}

It follows that the pair $(A[\tfrac 1f],\widetilde{A[\tfrac gf]})$ depends only on $U=U(\frac{g_1,\ldots,g_n}f)$, and we can define presheaves $\mathcal O_X$, $\mathcal O_X^+$ on the basis of rational opens of $X=\Spa(A,A^+)$ via $\mathcal O_X(U) = A[\tfrac 1f]$, $\mathcal O_X^+(U) = \widetilde{A[\tfrac gf]}$.

\begin{proposition} The presheaves $\mathcal O_X$ and $\mathcal O_X^+$ are sheaves. For all $x\in \Spa(A,A^+)$, the valuation $f\mapsto |f(x)|$ extends uniquely to $\mathcal O_{X,x}$. Moreover,
\[
\mathcal O_X^+(U) = \{f\in \mathcal O_X(U)\mid \forall x\in U, |f(x)|\leq 1\}\ .
\]
\end{proposition}

\begin{proof} We have already seen that the valuations extend uniquely to $\mathcal O_X(U)$ whenever $x\in U$, so by passing to the colimit they extend uniquely to the stalks. Moreover, $\mathcal O_X^+(U) = \{f\in \mathcal O_X(U)\mid \forall x_\in U, |f(x)|\leq \}$ holds for the $U=U(\frac{g_1,\ldots,g_n}f)$ by Proposition~\ref{prop:characterizeAplus} applied to the pair $(\mathcal O_X(U),\mathcal O_X^+(U))$. This implies that $\mathcal O_X^+$ is a sheaf as soon as $\mathcal O_X$ is.

Now note that there is a map $\psi: \Spec A\to \Spa(A,A^+)$ taking any prime ideal $\mathfrak p$ to the composite $A\to \mathrm{Frac}(A/\mathfrak p)\to \{0,1\}$. The pullback of $U(\frac{g_1,\ldots,g_n}f)$ under this map is $D(f)$. We see that $\mathcal O_{\Spa(A,A^+)}$ is the pushforward of $\mathcal O_{\Spec A}$ under $\psi$, on the level of presheaves. As pushforward preserves sheaves, this shows that $\mathcal O_{\Spa(A,A^+)}$ is a sheaf.
\end{proof}

\begin{definition} A discrete adic space is a triple $(X,\mathcal O_X,(|\cdot(x)|)_{x\in X})$ consisting of a topological space $X$ equipped with a sheaf of rings $\mathcal O_X$ and with an equivalence class of valuations $|\cdot(x)|$ on $\mathcal O_{X,x}$ for all $x\in X$, such that $(X,\mathcal O_X,(|\cdot(x)|)_{x\in X})$ is locally of the form $(\Spa(A,A^+),\mathcal O_{\Spa(A,A^+)},(|\cdot(x)|)_{x\in \Spa(A,A^+)})$ for a discrete Huber pair $(A,A^+)$.
\end{definition}

There are two fully faithful functors $X\mapsto X^{\mathrm{ad}}$ from schemes $X$ to discrete adic spaces. The first is given by $X=\Spec A\mapsto X^{\mathrm{ad}} = \Spa(A,A)$, while the second is given by $X=\Spec A\mapsto X^{\mathrm{ad}/\mathbb Z} = \Spa(A,\widetilde{\mathbb Z})$. It is easy to see that both functors glue, and that there is in general a comparison map $X^{\mathrm{ad}}\to X^{\mathrm{ad}/\mathbb Z}$. In fact, more generally, for any ring $R$, one can define a fully faithful functor $X\mapsto X^{\mathrm{ad}/R}$ from schemes over $\Spec R$ to discrete adic spaces over $\Spa(R,R)$, given by sending $X=\Spec A$ to $X^{\mathrm{ad}/R} = \Spa(A,\widetilde{R})$.

One can specify the points of $X^{\mathrm{ad}}$ and $X^{\mathrm{ad}/R}$ in terms of scheme-theoretic data. Indeed, points of $X^{\mathrm{ad}}$ are equivalent to maps $\Spec V\to X$ from spectra of valuation rings, up to equivalence (where equivalence is generated by declaring $\Spec V\to X$ and $\Spec W\to \Spec V\to X$ equivalent when $V\to W$ is a faithfully flat map of valuation rings). Similarly, points of $X^{\mathrm{ad}/R}$ are equivalent to valuation rings $V$ with fraction field $K$ and a diagram
\[\xymatrix{
\Spec K\ar[d]\ar[r] &X\ar[d]\\
\Spec V\ar[r] & \Spec R\ ,
}\]
again up to the same notion of equivalence.

\begin{proposition} If $X$ is separated and of finite type over $\Spec R$ then $X^{\mathrm{ad}}\to X^{\mathrm{ad}/R}$ is an open immersion. If $X$ is proper over $\Spec R$, then $X^{\mathrm{ad}}\to X^{\mathrm{ad}/R}$ is an isomorphism. 
\end{proposition}

\begin{proof} As long as $X$ is of finite type over $\Spec R$, the map $X^{\mathrm{ad}}\to X^{\mathrm{ad}/R}$ is locally on the source an open immersion: Indeed, this can be checked when $X$ is affine, in which case the map is indeed an open immersion. It follows that the assertions reduce to injectivity resp.~bijectivity of the underlying point sets, where it is precisely the valuative criterion of seperatedness/properness using the description of points.
\end{proof}

\begin{remark} Later, when we recover finiteness of coherent cohomology and Serre duality for proper maps, the properness of the map will be used via the previous proposition, as ensuring that $X^{\mathrm{ad}}$ and $X^{\mathrm{ad}/R}$ agree. In fact, the category of modules on $X^{\mathrm{ad}}$ will be glued from $D(A_\solid)$'s whereas the category of modules on $X^{\mathrm{ad}/R}$ will be glued from $D((A,R)_\solid)$'s; their agreement globally means that the functor $j_!$ appearing in the previous lecture will simply be the identity for proper maps.
\end{remark}

Finally, we can state the theorem on gluing. We would like to glue $(A,A^+)_\solid$-modules. Unfortunately, this cannot be done on the abelian level. Indeed, if $U\subset \Spa(A,A^+)$ is a rational subset, the base change functor
\[
-\otimes_{(A,A^+)_\solid} (\mathcal O_X(U),\mathcal O_X^+(U))_\solid
\]
is not exact. For example, if $(A,A^+)=(\mathbb F_p[T],\mathbb F_p)$ and $U=\{|T|\leq 1\}$, then $(\mathcal O_X(U),\mathcal O_X^+(U))=(\mathbb F_p[T],\mathbb F_p[T])$, and the injection
\[
\mathbb F_p[T]\hookrightarrow \mathbb F_p((T^{-1}))
\]
of $(\mathbb F_p[T],\mathbb F_p)_\solid$-modules becomes the map
\[
\mathbb F_p[T]\to 0
\]
of $(\mathbb F_p[T],\mathbb F_p[T])_\solid$-modules after base change.

We see that we should work with derived tensor products throughout, and thus glue on the level of derived categories. Unfortunately, gluing in derived categories is not possible because one has to keep track of homotopies between homotopies (between homotopies...). Thus, we need to pass to the $\infty$-categorical enhancements. Recall that for any abelian category $\mathcal A$, the derived category $D(\mathcal A)$ has a natural $\infty$-categorical enhancement $\mathcal D(\mathcal A)$. With this change, everything works:

\begin{theorem}\label{thm:globalization} Let $X$ be a discrete adic space. The association taking any open affinoid $U=\Spa(A,A^+)\subset X$ to the $\infty$-category
\[
\mathcal D((A,A^+)_\solid)\subset \mathcal D(A\Mod)
\]
defines a sheaf of $\infty$-categories on $X$. More precisely, the given association defines in an obvious way a covariant functor from the category of affinoid open $U\subset X$ (where maps induce the forgetful functors) towards the $\infty$-category of $\infty$-categories such that all functors admit left adjoints; passing to the corresponding diagram of left adjoint functors gives the implied functor $U=\Spa(A,A^+)\mapsto \mathcal D((A,A^+)_\solid)$.
\end{theorem}

In particular, we can define the $\infty$-category $\mathcal D((\mathcal O_X,\mathcal O_X^+)_\solid)$ by taking global sections; passing to the homotopy category, we also get a triangulated category $D((\mathcal O_X,\mathcal O_X^+)_\solid)$.

We will give the proof in the next lecture.

\newpage

\section{Lecture X: Globalization, II}

In this lecture, we prove Theorem~\ref{thm:globalization}. Although the statement is $\infty$-categorical, most results needed for the proof can actually be stated more concretely. The main idea of the proof is as for Zariski descent for schemes: If $X=\Spec A$ is covered by $U_i=\Spec A[\tfrac 1{f_i}]$ for some $f_1,\ldots,f_n$ that generate the unit ideal, then for any $A$-module $M$ the sequence
\[
0\to M\to \prod_{i=1}^n M[\tfrac 1{f_i}]\to \prod_{1\leq i<j\leq n} M[\tfrac 1{f_if_j}]\to \ldots \to M[\tfrac 1{f_1\cdots f_n}]\to 0
\]
is exact. To check this, it suffices to check after inverting each $f_i$; but then the cover is split, so the sequence is automatically exact.

The two key points needed for this is that localizations commute, and that a module is zero once it is locally zero. We will establish these results on the level of derived categories, and then observe that this gives the desired gluing once we formulate it in terms of $\infty$-categories.

First, we observe that all localizations can be regarded as full subcategories:

\begin{proposition}\label{prop:locfullyfaithful} Let $(A,A^+)$ be a discrete Huber pair, and let $U\subset X=\Spa(A,A^+)$ be a rational subset. Then the forgetful functor
\[
D((\mathcal O_X(U),\mathcal O_X^+(U))_\solid)\to D((A,A^+)_\solid)
\]
is fully faithful, and admits the left adjoint $-\otimes^L_{(A,A^+)_\solid} (\mathcal O_X(U),\mathcal O_X^+(U))_\solid$.
\end{proposition}

\begin{proof} We only need to prove that the functor is fully faithful. But we have compatible forgetful functors to $D(\mathcal O_X(U)\Mod)$ resp.~$D(A\Mod)$, and $\mathcal O_X(U) = A[\tfrac 1f]$ for some $f\in A$. As $D(A[\tfrac 1f]\Mod)\to D(A\Mod)$ is fully faithful, we get the result.
\end{proof}

Moreover, we need the following commutation of localization with base change.

\begin{proposition}\label{prop:locbasechange} Let $(A,A^+)\to (B,B^+)$ be a map of discrete Huber pairs, and let $U\subset X=\Spa(A,A^+)$ be a rational subset with preimage $V\subset Y=\Spa(B,B^+)$. Then the diagram of forgetful functors
\[\xymatrix{
D((\mathcal O_Y(V),\mathcal O_Y^+(V))_\solid)\ar@{^(->}[r]\ar[d] & D((B,B^+)_\solid)\ar[d]\\
D((\mathcal O_X(U),\mathcal O_X^+(U))_\solid)\ar@{^(->}[r] & D((A,A^+)_\solid)
}\]
is cartesian, i.e.~an object $M\in D((B,B^+)_\solid)$ lies in the essential image of $D((\mathcal O_Y(V),\mathcal O_Y^+(V))_\solid)$ if and only if the image of $M$ in $D((A,A^+)_\solid)$ lies in the essential image of $D((\mathcal O_X(U),\mathcal O_X^+(U))_\solid)$. Moreover, the diagram
\[\xymatrixcolsep{10pc}\xymatrix{
D((B,B^+)_\solid)\ar[r]^-{-\otimes^L_{(B,B^+)_\solid} (\mathcal O_Y(V),\mathcal O_Y^+(V))_\solid}\ar[d] & D((\mathcal O_Y(V),\mathcal O_Y^+(V))_\solid)\ar[d] \\
D((A,A^+)_\solid)\ar[r]^-{-\otimes^L_{(A,A^+)_\solid} (\mathcal O_X(U),\mathcal O_X^+(U))_\solid} & D((\mathcal O_X(U),\mathcal O_X^+(U))_\solid)
}\]
commutes naturally.
\end{proposition}

More precisely, in the final commutative diagram, there is a natural base change transformation, and we claim that this is an equivalence.

\begin{proof} Let $U=U(\frac{g_1,\ldots,g_n}f)\subset X$. This can be written as the intersection of the loci $\{|\frac{g_i}f(x)\leq 1\}$ inside $\{f\neq 0\}$. By induction, we may reduce to the case that either $U=U(\frac ff)=\{f\neq 0\}\subset X$ or $U=U(\frac g1)=\{|g|\leq 1\}$. In the first case, $\mathcal O_X(U) = A[\tfrac 1f]$ and $\mathcal O_X^+(U)$ is the integral closure of $A^+$ in $A[\tfrac 1f]$. This implies that
\[
(\mathcal O_X(U),\mathcal O_X^+(U))_\solid = (A[\tfrac 1f],A^+)_\solid = (A,A^+)_\solid\otimes_A A[\tfrac 1f]
\]
which implies that
\[
D((\mathcal O_X(U),\mathcal O_X^+(U))_\solid)\subset D((A,A^+)_\solid)
\]
is precisely the full subcategory of all $M\in D((A,A^+)_\solid)$ on which multiplication by $f$ is invertible. This also means that the left adjoint $M\mapsto M\otimes^L_{(A,A^+)_\solid} (\mathcal O_X(U),\mathcal O_X^+(U))_\solid$ is given by $M\mapsto M[\tfrac 1f]$. It is clear that this description is compatible with base change.

In the other case $U=U(\frac g1)=\{|g|\leq 1\}$, we can assume that $(A,A^+)=(\mathbb Z[T],\mathbb Z)$ and $g=T$. Then $(\mathcal O_X(U),\mathcal O_X^+(U))=(A,A)$ and $-\otimes^L_{(A,A^+)_\solid} (A,A)_\solid$ is given by
\[
R\intHom_A(A_\infty/A,-)[1]
\]
as proved in Lecture VIII; here we denote again $A_\infty = \mathbb Z((T^{-1}))$. To prove the assertions, it suffices to see that for any $M\in D((B,B^+)_\solid)$, the base change
\[
M\otimes^L_{(A,A^+)_\solid} (A,A)_\solid\in D(\Cond(B))
\]
lies in $D((B,\widetilde{B^+[T]})_\solid)\subset D(\Cond(B))$. To check this, we may assume that $(B,B^+)$ is of finite type over $\mathbb Z$, and that $B=B^+[T]$ is a polynomial algebra over $B^+$. In that case, we verify directly that
\[
M\otimes^L_{(A,A^+)_\solid} (A,A)_\solid = M\otimes^L_{(B,B^+)_\solid} (B,B)_\solid\ .
\]
Indeed $-\otimes_{(B,B^+)_\solid} (B,B)_\solid$ is given by
\[
R\intHom_B(B_\infty/B,-)[1]
\]
by the appendix to Lecture VIII, and $B_\infty/B = A_\infty/A\otimes_{\mathbb Z_\solid} B^+_\solid$, which formally implies the result.
\end{proof}

To prove Theorem~\ref{thm:globalization}, it is formal that we may assume that $X=\Spa(A,A^+)$ is affine, and work with the basis of rational open subsets $U=U(\frac{g_1,\ldots,g_n}f)\subset X$. For any such $U$, the forgetful functor
\[
D((\mathcal O_X(U),\mathcal O_X^+(U))_\solid)\to D((A,A^+)_\solid)
\]
is fully faithful, and admits a left adjoint. Thus, to $X=\Spa(A,A^+)$ we have associated the triangulated category $C:= D((A,A^+)_\solid)$, and to any rational open subset $U\subset X$, we have associated a full subcategory $C_U=D((\mathcal O_X(U),\mathcal O_X^+(U))_\solid)\subset C$ such that the inclusion admits a left adjoint $L_U: C\to C_U$.

First, we reformulate geometric properties in terms of these categories.

\begin{lemma}\leavevmode
\begin{enumerate}
\item[{\rm (i)}] Let $U,V\subset X$ be rational subsets. Then $C_{U\cap V} = C_U\cap C_V$ and $L_{U\cap V} = L_U\circ L_V = L_V\circ L_U$.
\item[{\rm (ii)}] Assume that $X=U_1\cup U_2\cup \ldots\cup U_n$ where all $U_i\subset X$ are rational subsets. If $M\in C$ and $L_{U_i}(M)=0$ for all $i=1,\ldots,n$, then $M=0$.
\end{enumerate}
\end{lemma}

\begin{proof} For part (i), first observe that if $U\subset V$, then $C_U\subset C_V\subset C$ as there are natural maps $(A,A^+)_\solid\to (\mathcal O_X(V),\mathcal O_X^+(V))_\solid\to (\mathcal O_X(U),\mathcal O_X^+(U))_\solid$ (inducing fully faithful forgetful functors). This shows that in general that $C_{U\cap V}\subset C_U\cap C_V$. For the converse inclusion, it is enough to show that for all $M\in C=D((A,A^+)_\solid)$ that $L_{U\cap V}(M)=L_U(L_V(M))$; indeed, for $M\in C_U\cap C_V$, we then have $M = L_U(L_V(M)) = L_{U\cap V}(M)\in C_{U\cap V}$. But the equation $L_{U\cap V} = L_U\circ L_V$ follows by the compatibility of localization with base change.

For part (ii), assume first that $n=2$ and that there is some $f\in A$ such that $U_1 = U(\frac 1f)=\{|f|\geq 1\}$ and $U_2 = U(\frac f1)=\{|f|\leq 1\}$; these rational open subsets always form a cover. In that case, we can assume that $A=\mathbb Z[T]$ and $f=T$ (as the localizations commute with base change). Note that $L_{U_1}(M)=0$ if and only if $M$ is a $\mathbb Z[[T]]$-module; and $L_{U_2}(M)=0$ if and only if $M$ is a $\mathbb Z((T^{-1}))$-module. Thus, the assumption means that $M$ is a module over $\mathbb Z[[T]]\otimes^L_{(\mathbb Z[T],\mathbb Z)_\solid} \mathbb Z((T^{-1})) = 0$, and thus $M=0$ as desired.

In general, we use Lemma~\ref{lem:standardrationalcovers} below to reduce to the case that $U_i=U(\frac{f_1,\ldots,f_n}{f_i})$ for some $f_1,\ldots,f_n\in A$ generating the unit ideal. As it is enough to show that $M[\tfrac 1{f_i}]=0$ for all $i=1,\ldots,n$, we can also assume that one of the $f_i$, say $f_n$, is equal to $1$. Now we argue by induction on $n$, the case $n=2$ being handled already. By that induction base, it is enough to show that the localizations of $M$ to $U(\frac 1{f_1})$ and $U(\frac{f_1}1)$ vanish, so we may also assume that either $f_1\in A^+$ or $f_1^{-1}\in A^+$. In the first case, the
\[
U_i=U\left(\frac{f_1,\ldots,f_{n-1},1}{f_i}\right) = U\left(\frac{f_2,\ldots,f_{n-1},1}{f_i}\right)
\]
for $i=2,\ldots,n$ already form a cover, so we are done by induction. In the second case, the
\[
U_i=U\left(\frac{f_1,\ldots,f_{n-1},1}{f_i}\right) = U\left(\frac{1,f_2f_1^{-1},\ldots,f_{n-1}f_1^{-1}}{f_if_1^{-1}}\right)
\]
for $i=1,\ldots,n-1$ already form a cover, so again we are done by induction.
\end{proof}

In the proof, we used that all covers of $X=\Spa(A,A^+)$ by rational subsets can be refined by standard rational covers:

\begin{lemma}[{\cite[Lemma 2.6]{HuberGeneralization}}]\label{lem:standardrationalcovers} Let $(A,A^+)$ be a discrete Huber pair and $X=\Spa(A,A^+)$. Assume that $U_1,\ldots,U_n\subset X$ are rational subsets of $X$ that cover $X$. Then there exist $s_1,\ldots,s_N\in A$ generating the unit ideal so that each $U(\frac{s_1,\ldots,s_N}{s_j})$ is contained in some $U_i$; in particular, the rational covering of $X$ by the $U(\frac{s_1,\ldots,s_N}{s_j})$ refines the given cover.
\end{lemma}

\begin{proof} Let $U_i=U(\frac{g_{1i},\ldots,g_{mi},f_i}{f_i})$ for various $g_{ji},f_i\in A$. Assume first that for all $i$, the ideal generated by $g_{1i},\ldots,g_{ni},f_i$ is all of $A$. In that case, let $T$ be the set of all products $h_1\cdots h_n$ where each $h_i\in \{g_{1i},\ldots,g_{mi},f_i\}$ and let $S\subset T$ be the subset of those elements for which at least one $h_i=f_i$. We claim that $S=\{s_1,\ldots,s_N\}$ has the required property. Take any $j$. First, note that
\[
U\left(\frac{s_1,\ldots,s_N}{s_j}\right) = U\left(\frac{T}{s_j}\right)
\]
as any $x\in X$ is contained in some $U_i$, and so replacing any $t\in T$ by a corresponding $s\in S$ changing only the factor at $i$ to $f_i$, we have $|t(x)|\leq |s(x)|$, giving the displayed equality. But now $s_j$ contains some factor $f_i$, and then
\[
U\left(\frac{T}{s_j}\right)\subset U\left(\frac{g_{1i},\ldots,g_{mi}}{f_i}\right)=U_i\ ,
\]
so we get the desired refinement.

In general, it now suffices to show that any closed point of $X$ has a basis of open neighborhoods of the form $U(\frac{g_1,\ldots,g_n}f)=U(\frac{g_1,\ldots,g_n,f}f)$ where $g_1,\ldots,g_n,f$ generate $A$. Let $\Gamma=\Gamma_x$ be the value group at $x$, and let $\Gamma^\prime\subset \Gamma$ be the subgroup generated by all $|f(x)|$, $f\in A$, that are $>1$. We get an induced valuation $x^\prime$ with values in $\Gamma^\prime$, where $|f(x^\prime)|=|f(x)|$ if $|f(x)|\in \Gamma^\prime$ and otherwise $|f(x^\prime)|=0$. Then $x^\prime$ is a specialization of $x$, so we actually see that $x^\prime=x$ by our assumption that $x$ is a closed point. Now assume that $x\in U(\frac{g_1,\ldots,g_n}f)$. Then $|f(x)|\neq 0$, so as $\Gamma=\Gamma^\prime$, we can find some $h\in A$ such that $|h(x)|>1$ and $|fh(x)|\geq 1$. Then
\[
x\in U\left(\frac{g_1h,\ldots,g_nh,1}{fh}\right)\subset U\left(\frac{g_1,\ldots,g_n}f\right)
\]
gives a rational neighborhood of the desired form.
\end{proof}

Finally, Theorem~\ref{thm:globalization} follows from the following general categorical statement.

\begin{proposition}\label{prop:inftydescent} Let $X$ be some spectral space equipped with a basis $B$ of quasicompact open subsets of $X$, stable under intersections. Let $\mathcal C$ some stable $\infty$-category. Let $U\mapsto \mathcal C_U\subset \mathcal C$ be a (covariant) functor from $B$ to full subcategories of $\mathcal C$, with the property that the inclusion $\mathcal C_U\hookrightarrow \mathcal C$ admits a left adjoint $L_U$. Moreover, assume that
\begin{enumerate}
\item[{\rm (i)}] If $U,V\in B$, then $\mathcal C_{U\cap V} = \mathcal C_U\cap \mathcal C_V$ and $L_{U\cap V} = L_U\circ L_V = L_V\circ L_U$.
\item[{\rm (ii)}] If $U\in B$ is covered by $U_1,\ldots,U_n\in B$ and $M\in \mathcal C_U$ with $L_{U_i}(M)=0$ for all $i=1,\ldots,n$, then $M=0$.
\end{enumerate}
Then the contravariant functor $U\mapsto \mathcal C_U$, with the left adjoints $L_U$ as functors, defines a sheaf of $\infty$-categories.
\end{proposition}

\begin{proof} Let $U\in B$ with a cover by $U_1,\ldots,U_n\in B$. For each nonempty finite subset $I\subset \{1,\ldots,n\}$ we get the category $\mathcal C_I = C_{\bigcap_{i\in I} U_i}\subset \mathcal C_U$, and the inclusion $\mathcal C_I\hookrightarrow \mathcal C_U$ has a left adjoint $L_I: \mathcal C_U\to \mathcal C_I$. We need to see that the natural functor
\[
F=(L_I)_I: \mathcal C_U\to \varprojlim_{I} \mathcal C_I
\]
is an equivalence. Formally, we may replace $X$ by $U$ to assume $U=X$ and so $\mathcal C_U=\mathcal C$.

It is formal that $F$ has a right adjoint, which is given by taking a collection of objects $M_I\in \mathcal C_I$, compatible under localization, to their limit $\varprojlim_I M_I\in \mathcal C$. First, we need to see that the natural unit map
\[
M\to \varprojlim_I L_I(M)
\]
is an equivalence. The limit is a finite limit and thus commutes with localization. By assumption (ii), it suffices to prove the result after applying $L_i$ for any $i=1,\ldots,n$. By assumption (i), the corresponding limit will look the same, except that we are now in a situation where some $U_i=X$. In that case, the cover is split and so $F$ is automatically an equivalence.

It follows that $F$ is fully faithful. Now assume that $(M_I)_I\in \varprojlim_I \mathcal C_I$ and define $M=\varprojlim_I M_I\in \mathcal C$. We need to see that the counit maps $L_I(M)\to M_I$ are equivalences. Again, $L_I$ commutes with the finite limit, and we are again formally reduced to the case that the cover is split.
\end{proof}

In particular, we can define $D((\mathcal O_X,\mathcal O_X^+)_\solid)$ by gluing. Let us end by giving an explicit description of it independently of $\infty$-categories. Assume that $X$ is separated, and let $X=\bigcup_{i=1}^n U_i$ be an affine cover of $X$, $U_i=\Spa(A_i,A_i^+)$. For any finite nonempty subset $I\subset \{1,\ldots,n\}$, we have the intersection $U_I=\bigcap_{i\in I} U_i=\Spa(A_I,A_I^+)$. Then we can consider the abelian category $\Cond_{\{U_I\}}(\mathcal O_X)$ of presheaves of condensed $\mathcal O_X$-modules which are defined only on the $\{U_I\}$; concretely, functors from $I$ to condensed $A_I$-modules $M_I$ together with maps $M_I\otimes_{A_I} A_J\to M_J$ for $I\subset J$. This is an abelian category with compact projective generators (here it is important to work with presheaves, not sheaves). Then
\[
D((\mathcal O_X,\mathcal O_X^+)_\solid)\to D(\Cond_{\{U_I\}}(\mathcal O_X))
\]
is fully faithful, and the essential image consists of all those $M\in D(\Cond_{\{U_I\}}(\mathcal O_X))$ such that all $M_I\in D(\Cond(A_I))$ lie in $D((A_I,A_I^+)_\solid)$, and the base change maps
\[
M_I\otimes^L_{(A_I,A_I^+)_\solid} (A_J,A_J^+)_\solid\to M_J
\]
are equivalences. Theorem~\ref{thm:globalization} asserts that this category does not depend on the chosen cover of $X$.
\newpage

\section{Lecture XI: Coherent duality}

Finally, we can give a proof of coherent duality. Let us first recall the classical statements.

Serre duality says that for a $d$-dimensional proper smooth variety $X$ over a field $k$, taking $\omega_X := \omega_{X/k} := \bigwedge^d \Omega^1_{X/k}$, there is a canonical trace map
\[
\mathrm{tr}: H^d(X,\omega_X)\to k
\]
such that for any coherent sheaf $\mathcal E$ on $X$, the pairing
\[
H^i(X,\mathcal E)\otimes_k \mathrm{Ext}^{d-i}_{\mathcal O_X}(\mathcal E,\omega_X)\to H^d(X,\omega_{X/k})\to k
\]
is perfect.

In the projective case, this can be proved by choosing an embedding into projective space and proving it by hand for the projective space. In this approach, it is not clear that the trace map is canonical. With more effort, one can give a proof in the proper case.

Grothendieck duality is a generalization to a general base. Let $f: X\to \Spec R$ be a proper smooth map of dimension $d$. Let $\omega_{X/R} = \bigwedge^d \Omega^1_{X/R}$, a line bundle on $X$. Then there is a canonical trace map
\[
\mathrm{tr}: H^d(X,\omega_{X/R})\to R
\]
such that for any quasicoherent sheaf $\mathcal M$ on $X$, the natural map
\[
R\Hom_{\mathcal O_X}(\mathcal M,\omega_{X/R})[d]\to R\Hom_R(R\Gamma(X,\mathcal M),R)
\]
is an isomorphism in the derived category $D(R)$ (where the map is again induced by the trace map and a Yoneda pairing, noting that the trace map can equivalently be regarded as a map $R\Gamma(X,\omega_{X/R})[d]\to R$).

Generalizations of this result to the nonsmooth situations have been widely studied. In this lecture, we will prove a generalization to the non-proper case (but retaining for simplicity the smoothness assumption, although our techniques will apply in general).

To state the generalization, we need to work with solid modules, and consider the functor taking a scheme $X$ to
\[
D(\mathcal O_{X,\solid}) := D((\mathcal O_{X^{\mathrm{ad}}},\mathcal O_{X^{\mathrm{ad}}}^+)_\solid)\ .
\]
If $X=\Spec A$, then by Theorem~\ref{thm:globalization} and the definition of $X^{\mathrm{ad}}=\Spa(A,A)$, we have $D(\mathcal O_{X,\solid}) = D(A_\solid)$; in general, it is defined by gluing.

\begin{theorem}\label{thm:openduality} Let $f: X\to \Spec R$ be a separated smooth map of finite type, of dimension $d$. Let $\omega_{X/R} = \bigwedge^d \Omega^1_{X/R}$. There is a canonical functor
\[
f_!: D(\mathcal O_{X,\solid})\to D(R_\solid)
\]
that agrees with $R\Gamma(X,-)$ in case $f$ is proper. It preserves compact objects. There is a natural trace map
\[
f_! \omega_{X/R}[d]\to R
\]
such that for all $C\in D(\mathcal O_{X,\solid})$, the natural map
\[
R\Hom_{\mathcal O_X}(C,\omega_{X/R})[d]\to R\Hom_R(f_! C,R)
\]
is an isomorphism.
\end{theorem}

In fact, we can build a whole $6$-functor formalism. Recall that a $6$-functor formalism consists essentially of the following data:

\begin{itemize}
\item A functor taking any scheme $X$ (or other type of geometric object currently under consideration) to a triangulated (or stable $\infty$-) category $D(X)$, which should in fact be a closed symmetric monoidal category, i.e.~be equipped with a symmetric monoidal tensor product $-\otimes_X-$ that admits a partial right adjoint $\intHom_X(-,-)$, i.e.~
\[
\Hom_{D(X)}(A\otimes_X B,C)\cong \Hom_{D(X)}(A,\intHom_X(B,C))\ .
\]
\item For each map $f: X\to Y$, a symmetric monoidal pullback functor $f^\ast: D(Y)\to D(X)$ that admits a right adjoint $f_\ast: D(X)\to D(Y)$.
\item For each map $f: X\to Y$ (possibly under extra conditions, for example separated of finite type), a ``pushforward with compact support" functor $f_!: D(X)\to D(Y)$ that admits a right adjoint $f^!: D(Y)\to D(X)$.
\end{itemize}

This gives in total the $6$ functors $(\otimes,\Hom,f^\ast,f_\ast,f_!,f^!)$. Usually, the first four of these are easy to construct, and the essential difficulty is to construct $f_!$.

The standard example of a $6$-functor formalism is the derived category of $\ell$-adic sheaves $D_{\mathrm{\acute{e}t}}(X,\mathbb Z_\ell)$.

Some properties that a $6$-functor formalism should satisfy are the following:

\begin{itemize}
\item If $f$ is proper, then $f_! = f_\ast$.
\item If $f$ is \'etale, then $f^! = f^\ast$. In particular, $f_!$ is in this case a left adjoint of $f^\ast$.
\end{itemize}

Note that these properties, together with the obvious requirement that the functors should behave well with respect to composition, already determine $f_!$ in general. Indeed, let $f: Y\to X$ be a separated map of finite type between (qcqs) schemes. Then by the Nagata compactification theorem, there is a factorization
\[\xymatrix{
Y\ar[r]^j\ar[dr]_f & \overline{Y}\ar[d]^{\overline{f}}\\
& X
}\]
where $j$ is an open immersion and $\overline{f}$ is proper. We necessarily have
\[
f_! = \overline{f}_! j_! = \overline{f}_\ast j_!
\]
where $j_!$ is a left adjoint of $j^\ast$; so $f_!$ is determined by the first four functors. It is then a nontrivial exercise to verify that this definition of $f_!$ is independent of the choices made.

Let us note some other properties a $6$-functor formalism should satisfy:

\begin{itemize}
\item (Verdier duality) Let $f: X\to Y$ be a map and $A\in D(X)$, $B\in D(Y)$. There is a natural functorial isomorphism
\[
f_\ast \Hom_X(A,f^!B)\cong \Hom_Y(f_! A,B)\ .
\]
\item (Projection formula) Let $f: X\to Y$ be a map and $A\in D(X)$, $B\in D(Y)$. There is a natural functorial isomorphism
\[
f_!(A\otimes_X f^\ast B)\cong f_!A\otimes_Y B\ .
\]
\item (Proper base change) Let
\[\xymatrix{
X^\prime\ar[r]^{g^\prime}\ar[d]_{f^\prime} & X\ar[d]^f \\
Y^\prime\ar[r]^g & Y
}\]
be a cartesian diagram.\footnote{In the coherent case, we will have to work with derived schemes here, in case the fibre product is not Tor-independent.} Then there is a natural equivalence
\[
g^\ast f_!\cong f^\prime_! g^{\prime\ast}: D(X)\to D(Y^\prime)\ .
\]
\item (Poincar\'e duality) Assume that $f: X\to Y$ is smooth. Then for all $B\in D(Y)$ the natural map
\[
f^! 1_Y\otimes_X f^\ast B\to f^! B
\]
(coming via adjointness from the projection formula) is an isomorphism, and $f^! 1_Y\in D(X)$ is an invertible object (ideally explicitly described). Here $1_Y\in D(Y)$ denotes the symmetric monoidal unit.
\end{itemize}

Much has happened in the general theory of six-functor formalisms in the last few years, starting from the foundational work of Liu--Zheng \cite{LiuZheng}, Gaitsgory--Rozenblyum \cite{GaitsgoryRozenblyum}, to the work of Heyer--Mann \cite{HeyerMann}, cf.~also \cite{ScholzeSixFunctors}. In particular, there are general construction mechanisms for six-functor formalisms at the $\infty$-categorical level. In this lecture, we will content ourselves with a sketch; a full construction is given in \cite[Lecture IX]{ScholzeSixFunctors}.

Let us now describe the $6$-functor formalism in coherent duality, based on solid modules. We take any scheme $X$ to the closed symmetric monoidal triangulated category $D(\mathcal O_{X,\solid})$, and any morphism $f: X\to Y$ to the usual pullback $f^\ast$ and pushforward $f_\ast$ functors. This could have been (and has been) done in the same way without any condensed mathematics.

It remains to define $f_!$ (then $f^!$ will be its right adjoint). If $f=j: \Spec A[\tfrac 1f]\to \Spec A$ is an open immersion, we want $j_!$ to be a left adjoint of $j^\ast$. This requires $j^\ast$ to commute with all limits, in particular we need that for any set $I$, we have
\[
\prod_I A\otimes_{A_\solid} A[\tfrac 1f]_\solid\cong \prod_I A[\tfrac 1f]\ .
\]
This is what fails entirely in abstract $A$-modules, but holds true in solid $A$-modules!

In fact, we can give the following general direct definition of $f_!$. Let $f: X\to Y$ be a separated map of finite type. We get a commutative diagram of adic spaces
\[\xymatrix{
X^{\mathrm{ad}}\ar[r]^j\ar[dr]^{f^{\mathrm{ad}}} & X^{\mathrm{ad}/Y}\ar[d]^{f^{\mathrm{ad}/Y}}\\
& Y^{\mathrm{ad}}\ .
}\]
The map $j$ is an open immersion, and an isomorphism if $f$ is proper.

\begin{proposition}\label{prop:jlowershriek} The functor
\[
j^\ast: D((\mathcal O_{X^{\mathrm{ad}/Y}},\mathcal O_{X^{\mathrm{ad}/Y}}^+)_\solid)\to D((\mathcal O_{X^{\mathrm{ad}}},\mathcal O_{X^{\mathrm{ad}}}^+)_\solid)
\]
admits a left adjoint
\[
j_!: D((\mathcal O_{X^{\mathrm{ad}}},\mathcal O_{X^{\mathrm{ad}}}^+)_\solid)\to D((\mathcal O_{X^{\mathrm{ad}/Y}},\mathcal O_{X^{\mathrm{ad}/Y}}^+)_\solid)\ .
\]
\end{proposition}

In the affine (and finite type over $\mathbb Z$) case, this follows from the results of the appendix to Lecture VIII; in general, one can deduce it from that case.

\begin{definition}\label{def:flowershriek} The functor
\[
f_!: D(\mathcal O_{X,\solid})\to D(\mathcal O_{Y,\solid})
\]
is defined as the composite $f_! = f^{\mathrm{ad}/Y}_\ast \circ j_!$.
\end{definition}

It is formal that $f_!$ commutes with all direct sums (as $j_!$ does as a left adjoint, and $f^{\mathrm{ad}/S}_\ast$ does as, after passing to an affine open subset of $S$, it can be written as a finite limit of forgetful functors by taking an affine open cover of $X$), and thus admits a right adjoint $f^!: D(\mathcal O_{S,\solid})\to D(\mathcal O_{X,\solid})$. This finishes the definition of the $6$ functors.

With this definition, all of the properties of a $6$-functor formalism stated above hold true. Much of this has been proved in the affine case in the appendix to Lecture 8, and in general the assertions can be reduced to that case.

In order to show that the present approach leads to a simplification in the foundations, we show that for smooth $f$, we can indeed identify $f^! \mathcal O_Y$ as $\omega_{X/Y}[d]$, which gives Theorem~\ref{thm:openduality}.

First, we note that $f^!$ is a twist of $f^\ast$ and commutes with flat base change if $f$ is of finite Tor-dimension. This can be checked in the affine case as everything localizes well (as for open immersions $j^!=j^\ast$).

\begin{proposition}\label{prop:basechange} Let $f: R\to A$ be a map of rings that is the base change of a finitely generated map of finite Tor-dimension between noetherian rings. Then the natural map
\[
f^! R\otimes_{A_\solid}^L f^\ast \to f^!
\]
of functors $D(R_\solid)\to D(A_\solid)$ is an equivalence. Moreover, if $g: R\to S$ is a flat map of rings and $f^\prime: S\to B=A\otimes_R S$ is the base change of $f$ along $g$, so we also have $g^\prime: A\to B$, then there is a natural equivalence $g^{\prime\ast} f^!\simeq f^{\prime !} g^\ast$ of functors
\[
D(R_\solid)\to D(B_\solid)\ .
\]
Here the map $g^{\prime\ast} f^!\to f^{\prime!} g^\ast$ is adjoint to
\[
f^!\to f^!g_\ast g^\ast\simeq g^\prime_\ast f^{\prime!} g^\ast
\]
using the equivalence $f^! g_\ast\simeq g^\prime_\ast f^{\prime!}$ adjoint to the proper base change isomorphism.
\end{proposition}

\begin{proof} As usual, we can reduce to the case that either $A=R[T]$ is a polynomial algebra, or $f$ is surjective. If $f$ is surjective, then $f_! = f_\ast$ (and $f^\prime_! = f^\prime_\ast$) are the forgetful functors. The right adjoint $f^!$ is given by $R\Hom_R(A,-) = R\Hom_R(A,R)\otimes^L_{A_\solid} f^\ast -$ (as $A$ is perfect as $R$-module by the assumption of finite Tor-dimension), and this description commutes with base change.

On the other hand, if $A=R[T]$, then $f_!$ is given by
\[
f_! M = M\otimes^L_{(A,R)_\solid} R((T^{-1}))/R[T] [-1]\ ,
\]
we have computed $f^!$ to be a twist of $f^\ast$ in Lecture VIII, and the statement amounts to the base change
\[
R\Hom_R(R((T^{-1}))/R[T],R)\otimes^L_{A_\solid} B_\solid\cong R\Hom_S(S((T^{-1}))/S[T],S)\ .
\]
But $f^! R=R\Hom_R(R((T^{-1}))/R[T],R)[1]\cong R[T][1] = A[1]$ and similarly for $f^\prime$, giving the result.
\end{proof}

In the case of regular closed immersions one can compute the dualizing complex.

\begin{proposition}\label{prop:regularclosed} Let $f: R\to A$ be a regular closed immersion, i.e.~$f$ is surjective and the kernel $I$ of $f$ is locally on $\Spec R$ generated by a regular sequence. Assume that $\Spec A\subset \Spec R$ is everywhere of pure codimension $c$. Then there is a natural isomorphism
\[
f^! R = R\Hom_R(A,R)\cong \det_A(I/I^2)^\ast[-c]\ .
\]
\end{proposition}

\begin{proof} This is a standard result. Locally one has $A=R/(f_1,\ldots,f_c)$ where $f_1,\ldots,f_c$ is a regular sequence. Then a Koszul complex calculation shows that
\[
R\Hom_R(A,R)\cong A[-c]\ .
\]
We also have an isomorphism $\det_A(I/I^2)\cong A$ where $f_1\wedge\ldots\wedge f_c$ serves as a basis element of $\det(I/I^2)$. Now one checks that the composite isomorphism
\[
R\Hom_R(A,R)\cong \det_A(I/I^2)^\ast[-c]
\]
is independent of the choice of $f_1,\ldots,f_c$.
\end{proof}

By using the diagonal, we can now also compute the dualizing complex for smooth morphisms.

\begin{theorem}\label{thm:smoothdualizing} Let $f: R\to A$ be a smooth map of dimension $d$. Then there is a natural isomorphism
\[
f^! R\cong \det_A(\Omega^1_{A/R})[d]\ .
\]
\end{theorem}

As the isomorphism is natural, and everything is defined locally, the result immediately globalizes.

\begin{proof} Any embedding of $\Spec A$ into an affine space is a regular closed immersion. By the computations for affine space and the case of regular closed immersions, we see that $f^!R$ is a line bundle concentrated in degree $d$. It remains to find a canonical isomorphism to $\det_A(\Omega^1_{A/R})[d]$.

Let $g: A\otimes_R A\to A$ be the multiplication map, geometrically corresponding to the diagonal $\Delta_f: \Spec A\hookrightarrow \Spec A\times_{\Spec R} \Spec A$. Then $\Delta_f$ is a regular closed immersion of codimension $c$. Let $p_1,p_2: \Spec A\times_{\Spec R} \Spec A\to \Spec A$ be the projection to one factor. We find that
\[\begin{aligned}
f^! R &= g^!p_1^!f^! R = g^!(p_1^\ast f^!R \otimes^L_{(A\otimes_R A)_\solid} p_1^! A)\\
&= g^!(p_1^\ast f^!R \otimes^L_{(A\otimes_R A)_\solid} p_2^\ast f^! R)\\
&= g^\ast(p_1^\ast f^!R \otimes^L_{(A\otimes_R A)_\solid} p_2^\ast f^! R)\otimes^L_{A_\solid} g^! (A\otimes_R A)\\
&= f^! R\otimes^L_{A_\solid} f^! R\otimes^L_{A_\solid} g^!(A\otimes_R A)\ .
\end{aligned}\]
Here, the first equality is clear, the second uses the formula for $p_1^!$ as the twist of $p_1^\ast$ by $p_1^! A$ from Proposition~\ref{prop:basechange}, the third uses the base change part of Proposition~\ref{prop:basechange}, the fourth uses the formula for $g^!$ as the twist of $g^\ast$ by $g^!(A\otimes_R A)$, and the last is simplifying all terms.

As $f^! R$ is invertible, this amounts to an isomorphism
\[
f^! R\cong (g^!(A\otimes_R A))^\ast\ .
\]
But by the previous proposition, $g^!(A\otimes_R A)\cong \det_A(I/I^2)^\ast[-d]$ where $I/I^2\cong \Omega^1_{A/R}$; this gives the desired result.
\end{proof}

\begin{remark} Theorem~\ref{thm:openduality} also states that $f_!$ preserves compact objects. This can be proved locally, where it can be proved by embedding into affine space and eventually reducing to the affine line, where it follows from the computations in Lecture VIII. Note that for proper $f$ and applied to discrete modules, this gives a new proof of finiteness of coherent cohomology!
\end{remark}

Thus, modulo establishing the formalism and checking various abstract results, this gives a proof of the finiteness of coherent cohomology and Grothendieck duality where the only computation one has to do is for $\mathbb A^1$.
\newpage

\bibliographystyle{amsalpha}

\bibliography{Condensed}

\end{document}